\documentclass[reqno,12pt,letterpaper]{amsart}
\usepackage[proof]{sdmacros}
\usepackage{xypic}
\usepackage{soul}

\title[Semiclassical measures for complex hyperbolic quotients]%
{Semiclassical measures for\\ complex hyperbolic quotients}
\author{Jayadev Athreya}
\email{jathreya@uw.edu}
\address{Department of Mathematics, University of Washington, Seattle, WA 98195}
\author{Semyon Dyatlov}
\email{dyatlov@math.mit.edu}
\address{Department of Mathematics, Massachusetts Institute of Technology, Cambridge, MA 02139}
\author{Nicholas Miller}
\email{nickmbmiller@ou.edu}
\address{Department of Mathematics, University of Oklahoma, Norman, OK, 73019}

\DeclareMathOperator{\Un}{U}
\DeclareMathOperator{\R}{\mathbb{R}}
\DeclareMathOperator{\Q}{\mathbb{Q}}
\DeclareMathOperator{\SU}{SU}
\DeclareMathOperator{\PU}{PU}

\def\su{\mathop{\mathfrak{su}}}

\begin{document}

\begin{abstract}
We study semiclassical measures for Laplacian eigenfunctions on compact complex hyperbolic quotients.
Geodesic flows on these quotients are a model case of hyperbolic dynamical systems
with different expansion/contraction rates in different directions.
We show that the support of any semiclassical measure is either equal to the entire cosphere bundle
or contains the cosphere bundle of a compact immersed totally geodesic complex submanifold.

The proof uses the one-dimensional fractal uncertainty principle of Bourgain--Dyatlov~\cite{fullgap}
along the fast expanding/contracting directions, in a way similar to the work
of Dyatlov--J\'ez\'equel~\cite{highcat} in the toy model of quantum cat maps,
together with a description of the closures of fast unstable/stable trajectories relying on Ratner theory. 
\end{abstract}

\maketitle

\addtocounter{section}{1}
\addcontentsline{toc}{section}{1. Introduction}

Let $(M,g)$ be a compact Riemannian manifold.
Consider a sequence of Laplacian eigenfunctions
\begin{equation}
  \label{e:laplace-eig}
u_j\in C^\infty(M),\quad
(-\Delta_g-\lambda_j^2)u_j=0,\quad
\lambda_j\to \infty,\quad
\|u_j\|_{L^2(M)}=1,
\end{equation}
where $\Delta_g$ is the Laplacian on $(M, g)$. Since the set of probability measures on a compact space is weak-* compact, we can, by passing to a subsequence, assume that the probability measures $|u_j|^2\,d\vol_g$ converge weak-* to some
measure $\tilde\mu$ as $j\to\infty$.
A quantum mechanical interpretation of this phenomenon is that $u_j$ are the pure states
of a free quantum particle on~$M$, and the limiting measure $\tilde \mu$ is a
macroscopic limit of the probability law of the position of the quantum particle
in the high energy r\'egime.

A major theme in quantum chaos is understanding which measures $\tilde\mu$ can arise as weak limits;
this includes the Quantum Ergodicity theorem~\cite{Shnirelman1,Zelditch-QE,CdV-QE} and
the Quantum Unique Ergodicity conjecture~\cite{Rudnick-Sarnak-QUE}. We will not discuss the full history of the field,
instead referring the reader to the reviews by Sarnak~\cite{Sarnak-QE-Review},
Zelditch~\cite{Zelditch-Notices}, and Dyatlov~\cite{Dyatlov-ICM,Dyatlov-Notices}.
The present paper is motivated by following conjecture;
see Theorems~\ref{t:meas-base}, \ref{t:meas}, \ref{t:esti} below for precise statements
of the results.
\begin{conj}
\label{c:full}
Let $(M,g)$ be a compact connected Riemannian manifold of negative sectional curvature.
Then each weak limit $\tilde\mu$ of a sequence of Laplacian eigenfunctions
satisfies $\supp\tilde\mu=M$. That is, for each nonempty open set $\Omega\subset M$ there exists
a constant $c_\Omega>0$ such that $\|u\|_{L^2(\Omega)}\geq c_\Omega\|u\|_{L^2(M)}$
for any Laplacian eigenfunction~$u$.
\end{conj}
Here the assumption of negative sectional curvature implies that the geodesic flow
on $(M,g)$ is strongly chaotic in the sense that it has an unstable/stable decomposition.
Conjecture~\ref{c:full} is one version of the informal statement `if the geodesic flow on~$M$
is chaotic, then Laplacian eigenfunctions spread out in the high energy limit'
and it would also follow from the Quantum Unique Ergodicity conjecture. It has
applications to control theory and damped wave equation, see the remark after Theorem~\ref{t:meas}.
Note that a bound $\|u\|_{L^2(\Omega)}\geq c_\Omega(\lambda)\|u\|_{L^2(M)}$,
with $c_\Omega(\lambda)$ exponentially decaying with $\lambda$, is well-known
(see e.g.~\cite{Logunov-Malinnikova-Remez}) and it is sharp in the case of the round sphere.

Conjecture~\ref{c:full} was proved by Dyatlov--Jin~\cite{meassupp}
for compact hyperbolic surfaces $M=\Gamma\backslash\mathbb H^2$.
Dyatlov--Jin--Nonnenmacher~\cite{varfup} later proved it for any negatively curved surface.
These results only applied to surfaces because they needed the unstable/stable
spaces for the geodesic flow to be one-dimensional. Adapting the methods of~\cite{meassupp,varfup}
to higher dimensions would have to overcome several major obstacles:
\begin{enumerate}
\item
a key ingredient, the \emph{fractal uncertainty principle (FUP)}
due to Bourgain--Dyatlov~\cite{fullgap}, was only known for subsets of $\mathbb R$;
\item
the geodesic flow might expand/contract at different rates along different directions
in the unstable/stable space;
\item
the unstable/stable foliations only have H\"older regularity, as opposed to $C^{1+}$ regularity
in the case of surfaces (which was crucially used in~\cite{varfup}).  
\end{enumerate}

It is natural to first consider Conjecture~\ref{c:full} in the setting of \emph{locally symmetric spaces},
where obstacle~(3) is not present as the unstable/stable foliations are smooth,
and try to generalize the result of~\cite{meassupp}. In particular, one
can study higher dimensional real hyperbolic manifolds,
where the geodesic flow is conformal on the unstable/stable spaces and thus obstacle~(2)
does not appear. Obstacle~(1) has been overcome in a breakthrough paper
of Cohen~\cite{Cohen-FUP} on higher dimensional FUP and an analogue of Theorems~\ref{t:meas-base}--\ref{t:meas} below for real hyperbolic manifolds has been recently proved by Kim--Miller~\cite{Kim-Miller}.

The present paper studies a different class of locally symmetric spaces, namely \emph{complex hyperbolic quotients}.
The geodesic flow on those is not conformal: the unstable/stable space splits into the \emph{fast}
direction where the flow expands/contracts like $e^{\pm 2t}$, and the \emph{slow}
directions where the flow expands/contracts like $e^{\pm t}$~-- see~\S\ref{s:complex-stun} below.
In particular, obstacle~(2) is present. The results of~\cite{Cohen-FUP,Kim-Miller} do
not apply to this setting~-- the FUP of~\cite{Cohen-FUP} makes the assumptions of ball and line porosity which
are hard to verify for complex hyperbolic quotients because stable/unstable
balls are mapped by the geodesic flow to high-eccentricity ellipses instead
of balls, due to the presence of the fast and slow directions.

However, as first observed in~\cite{highcat} in the toy model of quantum cat maps,
one can take advantage of the different expansion rates, choosing the propagation times in the argument
carefully and applying FUP only in the fast unstable/stable directions.
Those are one-dimensional for complex hyperbolic quotients, thus one can still use
the original one-dimensional FUP of~\cite{fullgap}. See the work of Kim~\cite{Kim-cat}
for more recent results on semiclassical measures for quantum cat maps.

Compared to~\cite{highcat} and~\cite{meassupp}, the complex hyperbolic case comes with several additional difficulties:
\begin{itemize}
\item As in~\cite{highcat}, potential obstructions to Conjecture~\ref{c:full} are non-dense flow lines of the fast unstable/stable
bundles. In the setting of~\cite{highcat}, these were relatively easy to classify
and the closures were given by subtori. For complex hyperbolic manifolds,
we use the classification of unipotent orbit closures proved by Ratner (Theorem~\ref{t:Ratner}). However, additional arguments
(using invariance under the geodesic flow, which is not unipotent) are needed to show that
the only obstructions are complex totally geodesic submanifolds. See Theorem~\ref{theo:orbitclosure}.
\item In~\cite{highcat} one used a local symplectomorphism which `straightened out' stable and unstable
leaves simultaneously. No such symplectomorphism exists in the complex hyperbolic case.
Moreover, the slow unstable/stable subbundles are not Frobenius integrable, so one
cannot make sense of slow unstable/stable leaves, see~\S\ref{s:slow-stun}. The solution
is to use a symplectomorphism which `straightens out' the spaces of interest only at a single
point, see Lemma~\ref{l:straighten-out} and~\S\ref{s:proof-of-porosity}.
\item The argument in~\cite{highcat} used the Weyl quantization on $\mathbb R^{n}$
to quantize rough symbols associated to any linear Lagrangian foliation, see~\cite[\S2.1.4]{highcat}
and \S\ref{s:calculus-lagrangian}.
In the present setting the unstable/stable foliations are not linear
and we have to use the quantization originating in Dyatlov--Zahl~\cite{hgap}.
That quantization depends on the foliation chosen and we have to carefully
study the position/frequency localization of the resulting pseudodifferential operators when transformed by the `straightening out'
symplectomorphism discussed in the previous item; see~\S\ref{s:localization-conjugated}.
\end{itemize}
See also the beginning of~\S\ref{s:key-proof} for an outline of part of the argument.

\subsection{Setting and the first result}\label{sec:setting}

Let us now state the results of the paper.
Let $(M,g)$ be a $2n$-dimensional compact complex hyperbolic quotient, that is, a quotient of the complex hyperbolic space $\mathbb{CH}^{2n}$ by a co-compact subgroup $\Gamma$ of the isometry group of $\mathbb{CH}^{2n}$ with the metric $g$ descending from $\mathbb{CH}^{2n}$. Then $(M,g)$ is in particular a K\"ahler manifold, and
conversely, any compact connected K\"ahler manifold $M$ of constant \emph{holomorphic} sectional
curvature~$-1$ is isometric to a quotient of $\mathbb{CH}^{2n}$, see for example Goldman~\cite{Goldman-Book}.

Assume that $\Sigma\subset M$ is a positive dimensional compact immersed real submanifold
(that is, $\Sigma$ is a compact abstract manifold with an immersion into~$M$).
We say that $\Sigma$ is \emph{totally geodesic} if its second fundamental form is zero;
alternatively, any geodesic on~$M$ which starts tangent to~$\Sigma$ stays on~$\Sigma$ for all times.
We say that $\Sigma$ is a \emph{complex submanifold} of~$M$ if the tangent spaces
of $\Sigma$ are invariant under the almost complex structure on $M$.
Our first result is the following theorem, which says that the support of each limit measure associated to Laplacian eigenfunctions contains some totally geodesic complex submanifold.
\begin{theo}
\label{t:meas-base}
Let $M$ be a compact complex hyperbolic quotient, and suppose
$\tilde\mu$ is a weak-* limit of the probability measures $|u_j|^2\,d\vol_g$
where $u_j$ is a sequence of Laplacian eigenfunctions satisfying~\eqref{e:laplace-eig}. Then there exists a compact immersed totally geodesic complex submanifold $\Sigma\subset M$ such that $\Sigma\subset\supp\tilde \mu$. In particular, if $M$ has no proper compact immersed totally geodesic complex submanifolds then $\supp\tilde \mu=M$.
\end{theo}
Note that there are examples of compact complex hyperbolic quotients which do not have any proper compact immersed totally geodesic complex submanifolds and there are also
examples of quotients with finitely many or infinitely many such submanifolds. We refer to~\S\ref{sec:geodesicsubman} below for a more detailed discussion of known examples.

\subsection{A semiclassical result}

Theorem~\ref{t:meas-base} follows from a more general result on
\emph{semiclassical measures} of Laplacian eigenfunctions. To introduce these,
we use \emph{semiclassical quantization}
$$
a\in\CIc(T^*M)\ \mapsto\ 
\Op_h(a):L^2(M)\to L^2(M),
$$
see~\S\ref{s:review-semi} below. Here $T^*M$ is the cotangent bundle of~$M$,
which we often identify with the tangent bundle $TM$ using the metric $g$.
We remark that in the (non-compact) case $M=\mathbb R^{2n}$ one can take the standard quantization
(see~\eqref{e:Op-h-def} below):
\begin{equation}
  \label{e:Op-h-intro}
\Op_h(a)f(x)=(2\pi h)^{-2n}\int_{\mathbb R^{4n}}e^{{i\over h}\langle x-y,\xi\rangle}a(x,\xi)f(y)\,dyd\xi
\end{equation}
and a quantization for general manifolds is typically constructed using standard quantization and coordinate charts.

Let $u_j$ be a sequence of Laplacian eigenfunctions satisfying~\eqref{e:laplace-eig}.
We say $u_j$ \emph{converges semiclassically} to a measure $\mu$ on $T^*M$ if,
denoting $h_j:=\lambda_j^{-1}$,
\begin{equation}
  \label{e:semi-measure-conv}
\langle \Op_{h_j}(a) u_j,u_j\rangle_{L^2(M)}\to\int_{T^*M}a\,d\mu\quad\text{for all}\quad
a\in \CIc(T^*M).
\end{equation}
If we interpret $u_j$ as the wave function of a quantum particle, then the left-hand side
of~\eqref{e:semi-measure-conv} is the average value of the classical observable $a(x,\xi)$
for this particle, where $x$ denotes the position variables and $\xi\in T_x^*M$ denotes
the momentum variables. Mathematically, $\xi$ is the frequency variable; for example,
\eqref{e:Op-h-intro} shows that the quantization of a function $a(\xi)$ is a Fourier multiplier.
Thus $\mu$ captures the macroscopic concentration of $u_j$ simultaneously in position and frequency,
in the high energy limit $h_j\to 0$. 

A measure $\mu$ on $T^*M$ is called a \emph{semiclassical measure} if it is the semiclassical limit of some sequence of Laplacian eigenfunctions~\eqref{e:laplace-eig}.
Such measures always exist, in fact each sequence satisfying~\eqref{e:laplace-eig}
has a subsequence converging to some measure~-- see~\cite[Theorem~5.2]{Zworski-Book}. 

If $\mu$ is a semiclassical measure, then $\mu$ is a geodesic-flow invariant probability measure with support contained in the unit cosphere bundle $S^*M =\{(x,\xi)\in T^*M\colon |\xi|_g=1\}$, and the pushforward of $\mu$ under the projection $S^*M\to M$ is the weak-*
limit of the probability measures $|u_j|^2\,d\vol_g$. To make sense of geodesic flow-invariance, we identify the sphere bundle $SM$ with the cosphere bundle $S^*M$ using the metric $g$, and consider the geodesic flow
$$
\varphi^t:S^*M\to S^*M.
$$
If $\Sigma\subset M$ is a submanifold, then we embed $S^*\Sigma$ into $S^*M$
using the orthogonal projection with respect to the metric $g$.
Note that $\Sigma$ is a totally geodesic submanifold if and only if
$S^*\Sigma$ is invariant under the geodesic flow. The next statement is a stronger, semiclassical, version of Theorem~\ref{t:meas-base}. 
\begin{theo}
  \label{t:meas}
Assume that $M$ is a compact complex hyperbolic quotient and $\mu$
is a semiclassical measure for a sequence of Laplacian eigenfunctions on $M$.
Then there exists a compact immersed totally geodesic complex submanifold $\Sigma\subset M$
such that $S^*\Sigma\subset\supp\mu$. In particular, if $M$
has no proper compact immersed totally geodesic complex submanifolds then $\supp\mu=S^*M$.  
\end{theo}
\Remark Note that Theorem~\ref{t:meas} immediately implies Theorem~\ref{t:meas-base} by characterization of pushforwards of semiclassical measures above. Theorem~\ref{t:meas} follows from a semiclassical estimate on
eigenfunctions $u_j$, Theorem~\ref{t:esti}~-- see~\S\ref{s:reduction-control}.
Theorem~\ref{t:esti} can be used to show an observability estimate
for the Schr\"odinger equation (see~\cite{Jin-Control}) and the proof in the present
paper can be modified to show exponential energy decay
for the damped wave equation (similarly to~\cite{Jin-DWE,varfup}).

As we see from Theorem~\ref{t:meas}, the obstacles to full support
of semiclassical measures are sets of the form $S^*\Sigma$ where $\Sigma$
are certain proper submanifolds of~$M$.  We nevertheless make the following:
\begin{conj}
  \label{c:meas}
Assume that $M$ is a compact complex hyperbolic quotient and $\mu$
is a semiclassical measure for a sequence of Laplacian eigenfunctions on~$M$.
Then $\supp\mu=S^*M$.
\end{conj}
Conjecture~\ref{c:meas} is in contrast with the setting of quantum cat maps studied by Kelmer~\cite{Kelmer-cat}, who gave examples of semiclassical measures supported on Lagrangian subtori. However, in our setting the set $S^*\Sigma$ is the intersection
of $S^*M$ with the symplectic submanifold $T^*\Sigma$.
The recent paper~\cite{Kim-cat} uses the basic uncertainty principle to show that no semiclassical measure for a quantum cat map can be supported on a single symplectic subtorus.
On the other hand, the same paper gives examples of semiclassical measures supported
on the union of two transversal symplectic subtori. A step towards Conjecture~\ref{c:meas}
would be show that on a complex hyperbolic quotient, no semiclassical measure can be localized on a finite union of the
sets $S^*\Sigma$.

\subsection{Structure of the paper}

\begin{itemize}
\item \S\ref{s:ch-quotients} reviews various geometric and dynamical properties
of complex hyperbolic quotients and sets up the notation used;
\item \S\ref{s:orbits-total} gives a description of orbit closures
for fast unstable/stable vector fields together with the geodesic flow
in terms of totally geodesic complex submanifolds;
\item \S\ref{s:usual-proof} reduces Theorem~\ref{t:meas} to
the key estimate, Proposition~\ref{l:key-estimate};
\item \S\ref{s:key-proof} proves this key estimate using the Fractal Uncertainty Principle.
\end{itemize}


\bigskip\noindent\textbf{Acknowledgements.}
We are thankful to the anonymous referees for the comments to improve the paper.
Athreya was partially supported by National Science Foundation (NSF) grant DMS-2003528 and DMS-2404705; the Pacific Institute for the Mathematical Sciences; the Royalty Research Fund and the Victor Klee fund at the University of Washington, and the Chaire Jean Morlet program at the Centre International de Recontres Math\'ematiques (CIRM) Luminy. Dyatlov was supported by NSF CAREER grant DMS-1749858,
NSF grant DMS-2400090, and a Simons Fellowship. Miller was partially supported by NSF grants DMS-2005438/2300370 and DMS-2405264.
This project originated in Fall 2019 at the Mathematical Sciences Research Institute at the programs on \emph{Holomorphic Differentials in Mathematics and Physics} and \emph{Microlocal Analysis}.

\section{Complex hyperbolic quotients}
\label{s:ch-quotients}

\subsection{Complex hyperbolic space}

We start by reviewing the geometry of complex hyperbolic space $\mathbb{CH}^n$, using the projective
(also known as hyperboloid) model.
Let $n\geq 2$ and consider the complex Minkowski space $\mathbb C^{n,1}=\mathbb C^{n+1}$
with the sesquilinear product
$$
\langle z,w\rangle_{\mathbb C^{n,1}}=-z_0\overline{w_0}+\langle z',w'\rangle_{\mathbb C^n}.
$$
Here we write elements of $\mathbb C^{n,1}$ as $(z_0,z')$ where $z_0\in\mathbb C$ and $z'\in\mathbb C^n$, and let $\langle\bullet,\bullet\rangle_{\mathbb C^n}$ be the standard Hermitian inner product:
$$
\langle z',w'\rangle_{\mathbb C^n}:=\sum_{j=1}^n z_j\overline{w_j}\quad\text{where}\quad
z'=(z_1,\dots,z_n),\
w'=(w_1,\dots,w_n).
$$
Define the `sphere' in $\mathbb C^{n,1}$
$$
\mathbb C\mathbb S^{n,1}:=\{z\in\mathbb C^{n,1}\mid \langle z,z\rangle_{\mathbb C^{n,1}}=-1\}
$$
which is a real manifold of dimension $2n+1$.

The inner product $\Re\langle \bullet,\bullet\rangle_{\mathbb C^{n,1}}$ induces a Lorentzian metric
on $\mathbb{CS}^{n,1}$, and the group $\Un(1)=\{e^{i\theta}\mid \theta\in\mathbb R\}$ acts by isometries
on $\mathbb{CS}^{n,1}$ by $e^{i\theta}.z:=e^{i\theta}z$. We define the complex hyperbolic space as the quotient
$$
\mathbb{CH}^n:=\mathbb {CS}^{n,1}/\Un(1).
$$
The Lorentzian metric on $\mathbb{CS}^{n,1}$ induces a Riemannian metric on $\mathbb{CH}^n$,
which we call the \emph{complex hyperbolic metric}.
In fact, the latter metric (together with the complex structure inherited from $\mathbb C^{n,1}$)
makes $\mathbb{CH}^n$ into a K\"ahler manifold. We refer the reader to~\cite{Goldman-Book,Parker-Book} for an introduction
to the geometry of complex hyperbolic space.

Denote by $S\mathbb{CH}^n$ the unit sphere bundle of $\mathbb{CH}^n$. We can write it as a quotient
$$
\begin{aligned}
S\mathbb{CH}^n&=S\mathbb{CS}^{n,1}/\Un(1),\\
S\mathbb{CS}^{n,1}&=\{(z,v)\in \mathbb C^{n,1}\times\mathbb C^{n,1}\mid
\langle z,z\rangle_{\mathbb C^{n,1}}=-1,\
\langle z,v\rangle_{\mathbb C^{n,1}}=0,\
\langle v,v\rangle_{\mathbb C^{n,1}}=1\}
\end{aligned}
$$
where the group $\Un(1)$ acts on $S\mathbb{CS}^{n,1}$ by $e^{i\theta}.(z,v)=(e^{i\theta}z,e^{i\theta}v)$.
To simplify notation, we often denote points in $S\mathbb{CH}^n$ by $(z,v)$
with the implication that the operations studied are equivariant under the $\Un(1)$ action.

\subsubsection{Isometry group}
\label{s:isometry-group}

We next write $\mathbb{CH}^n$ and $S\mathbb{CH}^n$ as homogeneous spaces.
Let
$$
G:=\SU(n,1)
$$
be the Lie group of complex linear automorphisms of $\mathbb C^{n,1}$
which preserve the product $\langle\bullet,\bullet\rangle_{\mathbb C^{n,1}}$
and have determinant~1.
Denote by $e_0,e_1,\dots,e_n$ the canonical (complex) basis of $\mathbb C^{n,1}$. 

Each $A\in G$ defines
a map $z\in\mathbb{CS}^{n,1}\mapsto Az\in\mathbb{CS}^{n,1}$, giving rise
to a transitive left action of $G$ on $\mathbb{CH}^n$
which is isometric with respect to the complex hyperbolic metric.
The isotropy group
of $e_0\in \mathbb{CH}^n$ with respect to this action is the maximal compact subgroup
of $G$:
\begin{equation}
  \label{e:K-iso-def}
K=\left\{\left.\begin{pmatrix} (\det B)^{-1} & 0 \\
0 & B\end{pmatrix}
\,\right|\, B\in U(n)\right\}.
\end{equation}
The action of $G$ on $\mathbb{CH}^n$ lifts to a transitive action
 on $S\mathbb{CH}^n$ by the formula $A.(z,v)=(Az,Av)$ where $A\in G$
and $(z,v)\in S\mathbb{CH}^{n}$. 
The isotropy group of $(e_0,e_1)\in S\mathbb{CH}^n$ with respect to this action is given by
the following double cover of the unitary group $\Un(n-1)$:
\begin{equation}
  \label{e:R-iso-def}
R=\left\{ \left.\begin{pmatrix}
e^{i\theta} & 0 & 0 \\
0 & e^{i\theta} & 0 \\
0 & 0 & B
\end{pmatrix}
\,\right|\,
B\in \Un(n-1),\
\det B=e^{-2i\theta}
\right\}.
\end{equation}
Here $R$ is a double cover since there are two choices of $\theta$ for each $B$.
This gives the following representations of~$\mathbb{CH}^n$ and~$S\mathbb{CH}^n$ 
as homogeneous spaces
(mapping $A\in G$ to $\widetilde\pi_K(A):=Ae_0\in\mathbb{CH}^n$ and $\widetilde\pi_R(A):=(Ae_0,Ae_1)\in S\mathbb{CH}^n$):
\begin{equation}
  \label{e:CH-quotient}
\mathbb{CH}^n\simeq G/K,\qquad
S\mathbb{CH}^n\simeq G/R.
\end{equation}

\subsubsection{Lie algebra}\label{sec:commutation}

For $j,k\in \{0,\dots,n\}$, denote by $\mathbf E_{jk}$ the matrix with entry $(\mathbf E_{jk})_{jk}=1$
and all other entries equal to~0. We use the following basis of the Lie algebra
$\mathfrak g=\su(n,1)$ of $G$:
\begin{equation}
  \label{e:lie-basis}
\begin{gathered}
X:=\mathbf E_{01}+\mathbf E_{10},\quad
V^\pm:=i(\mathbf E_{00}\mp \mathbf E_{01}\pm \mathbf E_{10}-\mathbf E_{11}),\\
W^\pm_j:=\mathbf E_{0j}\pm \mathbf E_{1j}+\mathbf E_{j0}\mp \mathbf E_{j1},\quad
Z^\pm_j:=i(\mathbf E_{0j}\pm \mathbf E_{1j}-\mathbf E_{j0}\pm \mathbf E_{j1}),\\
R_{jk}:=\mathbf E_{jk}-\mathbf E_{kj},\quad
R_{jk}':=i(\mathbf E_{jk}+\mathbf E_{kj}-\delta_{jk}(\mathbf E_{00}+\mathbf E_{11})).
\end{gathered}
\end{equation}
Here $j,k\in \{2,\dots,n\}$; for $R_{jk}$ we have $j<k$ and for $R'_{jk}$ we have $j\leq k$.
As an example, when $n=2$ we have
$$
\begin{gathered}
X = \begin{pmatrix}
0 & 1 & 0 \\
1 & 0 & 0 \\
0 & 0 & 0
\end{pmatrix},\quad
V^\pm = \begin{pmatrix}
i & \mp i & 0 \\
\pm i & -i & 0 \\
0 & 0 & 0
\end{pmatrix},\\
W^\pm_2=\begin{pmatrix}
0 & 0 & 1 \\
0 & 0 & \pm 1 \\
1 & \mp 1 & 0
\end{pmatrix},\quad
Z^\pm_2=\begin{pmatrix}
0 & 0 & i \\
0 & 0 & \pm i \\
-i & \pm i & 0
\end{pmatrix},\quad
R'_{22}=\begin{pmatrix}
-i & 0 & 0 \\
0 & -i & 0 \\
0 & 0 & 2i
\end{pmatrix}.
\end{gathered}
$$
Note that the Lie algebra $\mathfrak r$ of $R$ is spanned by the
fields $R_{jk},R'_{jk}$.

Recall that for a Lie algebra $\mathfrak g$, and $Y \in \mathfrak g$, we write $$\ad(Y)(\cdot) = [Y, \cdot]$$ for the adjoint action of $Y$ on $\mathfrak g$. We have the following relations in our Lie algebra. First, $V^{\pm}$ are eigenvectors for $\ad(X)$ with eigenvalues $\pm 2$,  and $W_j^{\pm}$ and $Z_j^{\pm}$ are both eigenvectors for $\ad(X)$ with eigenvalues $\pm 1$ respectively. That is, 
\begin{equation}\label{e:comm-rel}
\ad(X)(V^{\pm}) = \pm 2 V^{\pm},\quad \ad(X)(W_j^{\pm}) = \pm  W_j^{\pm},\quad \ad(X)(Z_j^{\pm}) = \pm  Z_j^{\pm}.
\end{equation}
Moreover, $X$ and $V^{\pm}$ are in the kernel of $\ad(R_{jk})$ and $\ad(R'_{jk})$. That is,
\begin{equation}\label{e:comm-rel:2}
\ad(R_{jk})(X) =  \ad(R'_{jk})(X) = \ad(R_{jk})(V^{\pm}) = \ad(R'_{jk})(V^{\pm}) = 0.
 \end{equation} 

We identify elements of the Lie algebra $\mathfrak g$ with left-invariant vector fields
on the group $G$. The vector fields $X,V^\pm$ commute with the group $R$ from~\eqref{e:R-iso-def}
and thus descend to vector fields on the sphere bundle $S\mathbb{CH}^n$, which we denote by the
same letters.

The flow of~$X$,
\begin{equation}
  \label{e:Ham-flow-CH}
\varphi^t:=e^{tX}:S\mathbb{CH}^n\to S\mathbb{CH}^n
\end{equation}
is the geodesic flow for $(\mathbb{CH}^n,g)$.

\subsection{Unstable/stable spaces}

In this section we study the unstable/stable spaces
for the geodesic flow $\varphi^t$ on~$S\mathbb{CH}^n$.

\subsubsection{Construction of the spaces}
\label{s:complex-stun}

The unstable/stable decomposition for the flow $\varphi^t$
is the following $\varphi^t$-invariant decomposition of the tangent bundle
to $S\mathbb{CH}^n$:
\begin{equation}
  \label{e:stun-CH}
T(S\mathbb{CH}^n)=\mathbb RX\oplus E_u\oplus E_s,\quad
E_u=\mathbb RV^-\oplus E^-,\quad
E_s=\mathbb RV^+\oplus E^+.
\end{equation}
Here we call $E_u,E_s$ the \emph{unstable/stable subbundles} and
\begin{itemize}
\item $\mathbb RV^-$ the \emph{fast unstable} subbundle,
\item $\mathbb RV^+$ the \emph{fast stable} subbundle,
\item $E^-$ the \emph{slow unstable} subbundle, and
\item $E^+$ the \emph{slow stable} subbundle.
\end{itemize}
To define the slow unstable/stable subbundles, consider the
$2n-2$-dimensional subspaces
\begin{equation}
  \label{e:tilde-E-pm-def}
\widetilde E^\pm=\{A\in\mathfrak g\colon [X,A]=\pm A\}=
\Span\{W_j^\pm,Z_j^\pm\colon j=2,\dots,n\}.
\end{equation}
Since $X$ commutes with the group $R$, the spaces $\widetilde E^\pm$
are mapped to themselves by the adjoint representation of $R$.
This can also be seen as follows:
consider the real linear isomorphisms $\kappa_E^\pm:\mathbb C^{n-1}\to \widetilde E^\pm$ defined by
\begin{equation}
  \label{e:E-pm-identified}
\kappa_E^\pm(w_2,\dots,w_n)=\sum_{j=2}^n
(\Re w_j)W^\pm_j-(\Im w_j)Z^\pm_j.
\end{equation}
Then we have for all $r=\diag(e^{i\theta},e^{i\theta},B)\in R$ and $w\in \mathbb C^{n-1}$
\begin{equation}
  \label{e:rot-action-slow}
r\kappa_E^\pm(w) r^{-1}=\kappa_E^\pm(e^{-i\theta}B w).
\end{equation}
Consider the real inner product on $\widetilde E^\pm$ obtained
from the standard real inner product on $\mathbb C^{n-1}\simeq \mathbb R^{2n-2}$
using the map $\kappa_E^\pm$. From~\eqref{e:rot-action-slow} we see that
the adjoint action of $R$ on~$\widetilde E^\pm$ is isometric.

The subspaces $\widetilde E^\pm$ induce subbundles of the tangent space to the group $G$
via left-invariant vector fields; these subbundles come with a real inner product induced
by the one on $\widetilde E^\pm$ and the right action of the subgroup $R$ maps them
isometrically to themselves. Thus we can pass to the quotient
$S\mathbb{CH}^n$, obtaining the slow unstable/stable subbundles $E^\pm$
endowed with an inner product.

Fix a Riemannian metric on $S\mathbb{CH}^n$ by requiring that
$X,V^-,V^+,E^-,E^+$ be orthogonal to each other,
$X,V^-,V^+$ be unit length, and the metric on $E^\pm$ coincide with the one fixed above.
From~\eqref{e:comm-rel} we see that the decompositions~\eqref{e:stun-CH} are
preserved by the geodesic flow $\varphi^t$ and moreover we have the expansion/contraction
property for all $q\in S\mathbb{CH}^n$
\begin{equation}
  \label{e:stun-exp}
|d\varphi^t(q)w|=\begin{cases}
e^{\mp 2t}|w|,& w\in \mathbb RV^\pm(q);\\
e^{\mp t}|w|,& w\in E^\pm(q).
\end{cases}
\end{equation}
This justifies the terminology `fast/slow unstable/stable subbundle' since
the flow expands/contracts on $\mathbb RV^\pm$ twice as fast as on $E^\pm$.

For later use we compute here the action of elements of the lifted unstable/stable bundles on~$\mathbb C^{n,1}$:
for all $z\in\mathbb C^{n,1}$ and $w\in \mathbb C^{n-1}$ 
\begin{equation}
  \label{e:stun-matrix-action}
\begin{aligned}
V^\pm z&=-i\langle z,e_0\pm e_1\rangle_{\mathbb C^{n,1}}(e_0\pm e_1),\\
\kappa_E^\pm(w)z&=\langle z,(0,w)\rangle_{\mathbb C^{n,1}}(e_0\pm e_1)
-\langle z,e_0\pm e_1\rangle_{\mathbb C^{n,1}}(0,w).
\end{aligned}
\end{equation}
This implies the matrix product identities (true for any
$c\in\mathbb R$ and $w\in\mathbb C^{n-1}$)
\begin{equation}
  \label{e:N+product}
(cV^\pm+\kappa^\pm_E(w))^2=-i|w|^2V^\pm,\quad
(cV^\pm+\kappa^\pm_E(w))^3=0
\end{equation}
and the commutation identities (true for any $w,\widetilde w\in\mathbb C^{n-1}$)
\begin{equation}
  \label{e:N+comm}
[V^\pm,\widetilde E^\pm]=0,\quad
[\kappa^\pm_E(w),\kappa^\pm_E(\widetilde w)]=-2\Im\langle w,\widetilde w\rangle_{\mathbb C^{n-1}}V^\pm.
\end{equation}

\subsubsection{Extension to the cotangent bundle}
\label{s:extend-cotangent}

This paper uses semiclassical analysis (see~\S\ref{s:review-semi} below),
the phase space for which is given by the cotangent bundle $T^*\mathbb{CH}^n$.
We thus need to bring the unstable/stable decomposition defined above
to the cotangent bundle.
We identify $T^*\mathbb{CH}^n$ with $T\mathbb{CH}^n$
using the complex hyperbolic metric~$g$. Denote
$$
T^*\mathbb{CH}^n\setminus 0:=\{(z,\zeta)\in T^*\mathbb{CH}^n\mid \zeta\neq 0\}.
$$
We extend the spaces $E_u,E_s$ from $S^*\mathbb{CH}^n\simeq S\mathbb{CH}^n$ to $T^*\mathbb{CH}^n\setminus 0$
by making them positively homogeneous, i.e. equivariant under the dilation map $(z,\zeta)\to (z,\tau\zeta)$ for $\tau>0$. Same applies to the vector fields $V^\pm$
and the spaces $E^\pm$.
Similarly we extend homogeneously the vector field $X$ to $T^*\mathbb{CH}^n\setminus 0$, and the flow~\eqref{e:Ham-flow-CH} extends to the
homogeneous geodesic flow
\begin{equation}
  \label{e:Ham-flow-CH-2}
\varphi^t=e^{tX}:T^*\mathbb{CH}^n\setminus 0\to T^*\mathbb{CH}^n\setminus 0.
\end{equation}
Introduce also the vector field
$$
\zeta\cdot\partial_\zeta
$$
on $T^*\mathbb{CH}^n$, which is the generator of dilations in the fibers.
Note that our choice of the extensions of $X,V^\pm$ from $S^*\mathbb{CH}^n$ to $T^*\mathbb{CH}^n$
implies that these vector fields commute with $\zeta\cdot\partial_\zeta$.

\subsubsection{Integrability of the weak unstable/stable foliations}

We will use semiclassical calculi associated to the weak unstable/stable bundles
(see~\S\ref{s:calculus-lagrangian} below), defined as follows:
\begin{equation}
  \label{e:L-def}
L_u:=\mathbb RX\oplus E_u,\quad
L_s:=\mathbb RX\oplus E_s.
\end{equation}
For that we will need to show that the bundles $L_u,L_s$
are integrable (in the sense of Frobenius) and
Lagrangian with respect to the standard symplectic form $\omega$ on $T^*\mathbb{CH}^n$.
We start with integrability; it follows
from the Unstable/Stable Manifold Theorem (see e.g.~\cite[\S6.1]{Fisher-Hasselblatt} or~\cite{Dyatlov-hypdyn}), but here we give a direct proof
by computation:
\begin{lemm}
\label{l:E-u-integrable}
Assume that $Y_1,Y_2$ are vector fields on $T^*\mathbb{CH}^n\setminus 0$ tangent to $L_u$
(at every point). Then
the Lie bracket $[Y_1,Y_2]$ is also tangent to $L_u$. The same is true
with $L_u$ replaced by~$L_s$. 
\end{lemm}
\begin{proof}
We consider the case of $L_u$, with the case of $L_s$ handled similarly.
It suffices to show the same property for vector fields on $S\mathbb{CH}^n$.
Denote by $\widetilde\pi_R:G\to S\mathbb{CH}^n$ the projection map
induced by~\eqref{e:CH-quotient}. Let $\widetilde Y_1,\widetilde Y_2$
be vector fields on $G$ which are lifts of $Y_1,Y_2$ in the sense
that $d\widetilde\pi_R(g)\widetilde Y_j(g)=Y_j(\widetilde\pi_R(g))$
for all $g\in G$.
Then $[\widetilde Y_1,\widetilde Y_2]$ is a lift of $[Y_1,Y_2]$.

Recalling the definition of $L_u$, we see that $\widetilde Y_1,\widetilde Y_2$ can be chosen as linear combinations with coefficients in $C^\infty(G)$ of
the left-invariant vector fields in the subspace
$\mathfrak l^-:=\mathbb RX\oplus \mathbb RV^-\oplus \widetilde E^-\subset\mathfrak g$.
As follows from~\eqref{e:comm-rel} and~\eqref{e:N+comm}, $\mathfrak l^-$ is a Lie subalgebra of~$\mathfrak g$,
so $[\widetilde Y_1,\widetilde Y_2]$ is a linear combination
of elements of $\mathfrak l^-$ as well, which implies
that its projection $[Y_1,Y_2]$ is tangent to $L_u$ as needed.
\end{proof}

\subsubsection{Symplectic structure}

We next study the behavior of the standard symplectic form $\omega$ on~$T^*\mathbb{CH}^n$
with respect to the decomposition
\begin{equation}
  \label{e:decor}
T(T^*\mathbb{CH}^n\setminus 0)=\mathbb R(\zeta\cdot\partial_\zeta)\oplus \mathbb RX
\oplus E_u\oplus E_s
\end{equation}
where we recall from~\eqref{e:stun-CH} that $E_u=\mathbb RV^-\oplus E^-$
and $E_s=\mathbb RV^+\oplus E^+$.
\begin{lemm}
  \label{l:E-su-symplectic}
We have
\begin{equation}
  \label{e:E-su-symplectic}
\begin{aligned}
\omega(\mathbb R(\zeta\cdot\partial_\zeta)\oplus\mathbb RX,
E_u\oplus E_s)&=0,\\
\omega(E_u,E_u)&=0,\\
\omega(E_s,E_s)&=0,\\
\omega(V^\pm,E^\mp)&=0.
\end{aligned}
\end{equation}
\end{lemm}
\begin{proof}
This can be shown by direct computation, but we instead use
the expansion/contraction property of the spaces involved with respect
to the flow $\varphi^t$. We show the last statement in~\eqref{e:E-su-symplectic}
for the pairing of $V^+$ with $E^-$,
with the rest proved similarly. It suffices to show this statement
on $S^*\mathbb{CH}^n\simeq S\mathbb{CH}^n$. Take $q\in S^*\mathbb{CH}^n$
and $W\in E^-(q)$. The flow $\varphi^t$ is a symplectomorphism
(as it is the Hamiltonian flow of~$|\xi|_g$), thus we have for all~$t\in\mathbb R$
$$
\omega(V^+(q),W)=\omega(d\varphi^t(q)V^+(q),d\varphi^t(q)W).
$$
The metric on $S^*\mathbb{CH}^n$ introduced before~\eqref{e:stun-exp}
is invariant under the transitive left action of the isometry group $G$,
and so is the symplectic form~$\omega$. Therefore,
the action of $\omega$ on a pair of vectors can be estimated in terms
of the norms of these vectors.
It follows that there exists a constant $C$ such that for all $t$
\begin{equation}
  \label{e:E-su-symint}
|\omega(V^+(q),W)|\leq C|d\varphi^t(q)V^+(q)|\cdot |d\varphi^t(q)W|.
\end{equation}
By~\eqref{e:stun-exp}, the right-hand side of~\eqref{e:E-su-symint}
is equal to $Ce^{-t}|V^+(q)|\cdot |W|$. Taking $t\to \infty$,
we see that $\omega(V^+(q),W)=0$ as needed.
\end{proof}
From Lemma~\ref{l:E-su-symplectic} we immediately obtain
\begin{corr}
\label{l:l-su-lagr}
For each $q\in T^*\mathbb{CH}^n\setminus 0$,
the spaces $L_u(q)$ and $L_s(q)$ are Lagrangian, that is they have dimension~$2n$
and the symplectic form $\omega$ vanishes on them.
\end{corr}
Another consequence of Lemma~\ref{l:E-su-symplectic} is the existence
of special symplectic coordinates, used in~\S\ref{s:proof-of-porosity},\S\ref{s:end-proof} below, which
straighten out at one point the decomposition~\eqref{e:decor}:
\begin{lemm}
  \label{l:straighten-out}
Fix $q^0\in T^*\mathbb{CH}^n\setminus 0$.
Then there exists a neighborhood
$U_0$ of $q^0$ in $T^*\mathbb{CH}^n$ and a symplectomorphism onto its image
$\varkappa_0:U_0\to T^*\mathbb R^{2n}$, such that, denoting by
$(y_1,\dots,y_{2n})$ the coordinates on $\mathbb R^{2n}$ and by $(\eta_1,\dots,\eta_{2n})$ the corresponding coordinates on the fibers of~$T^*\mathbb R^{2n}$, we have
\begin{align}
  \label{e:straighten-out-1}
\varkappa_0(q^0)&=0,\\
  \label{e:straight-out-2}
d\varkappa_0(q^0)(V^+(q^0))&\in \mathbb R\partial_{y_1},\\
  \label{e:straight-out-3}
d\varkappa_0(q^0)(V^-(q^0))&\in \mathbb R\partial_{\eta_1},\\
  \label{e:straight-out-4}
d\varkappa_0(q^0)(E^+(q^0))&= \Span(\partial_{y_2},\dots,\partial_{y_{2n-1}}),\\
  \label{e:straight-out-5}
d\varkappa_0(q^0)(E^-(q^0))&= \Span(\partial_{\eta_2},\dots,\partial_{\eta_{2n-1}}),\\
  \label{e:straight-out-6}
d\varkappa_0(q^0)(X(q^0))&\in \mathbb R\partial_{y_{2n}},\\
  \label{e:straight-out-7}
d\varkappa_0(q^0)(\zeta\cdot\partial_\zeta(q^0))&\in \mathbb R\partial_{\eta_{2n}}.
\end{align}
\end{lemm}
\begin{proof}
Put $\mathbf e_1:=V^+(q^0),\mathbf e_{2n}:=X(q^0)$, and let
$\mathbf e_2,\dots,\mathbf e_{2n-1}$ be a basis of $E^+(q^0)$.
By Lemma~\ref{l:E-su-symplectic}, the symplectic complement
of $V^+$ is given by $\Span(\zeta\cdot\partial_\zeta,X,V^+)
\oplus E^+\oplus E^-$, which has trivial intersection with $V^-$. Therefore,
there exists $\mathbf f_1\in \mathbb R V^-(q^0)$ such that $\omega(\mathbf f_1,\mathbf e_1)=1$.
The symplectic complement of $E^+$ is given by $\Span(\zeta\cdot\partial_\zeta,X,V^+,V^-) \oplus E^+$, thus the symplectic form $\omega$ is nondegenerate
when restricted to $E^+\times E^-$. It follows that there exists
a basis $\mathbf f_2,\dots,\mathbf f_{2n-1}$ of $E^-(q_0)$ such that
$\omega(\mathbf f_j,\mathbf e_k)=\delta_{jk}$.
Finally, the symplectic complement of $\mathbb RX$ is given by
$\Span(X,V^+,V^-)\oplus E^+\oplus E^-$, thus there exists $\mathbf f_{2n}\in\mathbb R(\zeta\cdot\partial_\zeta)(q^0)$ such that $\omega(\mathbf f_{2n},\mathbf e_{2n})=1$.

It follows from the construction above and Lemma~\ref{l:E-su-symplectic} that
$\mathbf e_1,\dots,\mathbf e_{2n},\mathbf f_1,\dots,\mathbf f_{2n}$
forms a symplectic basis of~$T_{q^0}(T^*\mathbb{CH}^n)$ with respect to $\omega$. 
It remains to take
a symplectomorphism $\varkappa_0$ such that $d\varkappa_0(q^0)$
maps $\mathbf e_j$ to $\partial_{y_j}$ and $\mathbf f_k$ to $\partial_{\eta_k}$.
\end{proof}
Define the following complements of the fast unstable/stable 
spaces $\mathbb RV^\pm$:
\begin{equation}
  \label{e:V-perp-def}
V_\perp^\pm:=\mathbb R(\zeta\cdot\partial_\zeta)\oplus\mathbb RX\oplus \mathbb RV^\mp
\oplus E^+\oplus E^-.
\end{equation}
Then Lemma~\ref{l:straighten-out} implies that
\begin{align}
  \label{e:straighten-out-2}
d\varkappa_0(q^0)V_\perp^+(q^0)&=\ker dy_1,\\
  \label{e:straighten-out-3}
d\varkappa_0(q^0)V_\perp^-(q^0)&=\ker d\eta_1.
\end{align}

\subsection{Complex hyperbolic quotients}
\label{s:quotients-def}

Assume now that $M$ is a compact complex hyperbolic quotient, that is a compact
Riemannian manifold of the form
$$
M=\Gamma\backslash \mathbb{CH}^n
$$
where $\Gamma\subset G$ is a co-compact discrete subgroup acting freely
and the metric on~$M$ is descended from the complex hyperbolic metric on $\mathbb{CH}^n$.
For a discussion of known constructions of such $\Gamma$, see~\S\ref{sec:geodesicsubman}.
Using~\eqref{e:CH-quotient} we can write $M$ and its sphere bundle $SM$ as double quotients
of the group~$G$:
\begin{equation}
  \label{e:M-quotient}
M\simeq \Gamma\backslash G / K,\qquad
SM\simeq \Gamma\backslash G / R\simeq \Gamma\backslash S\mathbb{CH}^n.
\end{equation}
We have the following commutative diagram of quotient maps:
\begin{equation}\label{eqn:commtriangle}
\begin{gathered}
\xymatrix{
& & \mathbb{CH}^n \ar[ddd]^{\pi_M} \\
G \ar[r]_{\widetilde\pi_R} \ar[urr]^{\widetilde\pi_K} \ar[d]_{\pi_\Gamma} & S\mathbb{CH}^n \ar[ur]_{\widetilde\pi_S} \ar[d]^{\pi_{SM}} & \\
\Gamma\backslash G \ar[r]^{\pi_R} \ar[drr]_{\pi_K} & SM \ar[dr]^{\pi_S} &  \\
& & M
}
\end{gathered}
\end{equation}
The vector fields $X,V^\pm$ and the spaces $E^\pm$ defined in~\S\ref{s:complex-stun}
are invariant under the left action of $G$ on $S\mathbb{CH}^n$ and thus descend to~$SM$ via
the projection $\pi_{SM}$. In particular,
the unstable/stable decomposition~\eqref{e:stun-CH} and the expansion/contraction property~\eqref{e:stun-exp}
still hold on~$SM$.

\subsubsection{Slow unstable/stable rectangles}
\label{s:slow-stun}

We finally state a result about the propagation of `rectangles' which have
size $\alpha\ll 1$ in the direction of the space $V_\perp^\pm$ defined in~\eqref{e:V-perp-def} and size $\alpha^2$ in the transversal direction of $V^\pm$. This statement is used in the proof of Lemma~\ref{l:Omega-approx} below. This is an important step in the proof of the porosity property
needed to apply the Fractal Uncertainty Principle: this is where we use
that the expansion rate along the slow unstable/stable directions is less than
along the fast directions.

We remark that the subbundles $V_\perp^\pm\subset T(T^*\mathbb{CH}^n\setminus 0)$
are not Frobenius integrable, as can be seen by following
the proof of Lemma~\ref{l:E-u-integrable} and using~\eqref{e:N+comm}:
the Lie bracket of two vector fields tangent to $E^\pm$ can be a nonzero multiple
of $V^\pm$. Nevertheless, the rectangles used below are canonically defined up to multiplying $\alpha$ by a constant.
\begin{lemm}
  \label{l:rectangle-propagate}
Assume that $q^0\in T^*M\setminus 0$, $U_0$ is an open set containing $q^0$, and
$\varkappa_0:U_0\to T^*\mathbb R^{2n}$ is a diffeomorphism onto its image satisfying
the properties~\eqref{e:straighten-out-1} and~\eqref{e:straighten-out-2}--\eqref{e:straighten-out-3}.
Take small $\alpha>0$ and two numbers $y_1^0,\eta_1^0\in [-\alpha,\alpha]$, and define the slow unstable/stable rectangles (which are subsets of~$T^*M\setminus 0$)
\begin{equation}
  \label{e:rect-def}
\begin{aligned}
\mathcal R^-_{q_0,\eta_1^0,\alpha}&\,:=\varkappa_0^{-1}\big(\{(y,\eta)\colon |y|+|\eta|\leq \alpha,\ |\eta_1-\eta_1^0|\leq\alpha^2\}\big),\\
\mathcal R^+_{q_0,y_1^0,\alpha}&\,:=\varkappa_0^{-1}\big(\{(y,\eta)\colon |y|+|\eta|\leq\alpha,\
|y_1-y_1^0|\leq \alpha^2\}\big).
\end{aligned}
\end{equation}
Then there exists a constant $C$ independent of $\alpha,y_1^0,\eta_1^0$ such that, denoting by $\diam$ the diameter of a subset of $T^*M$, we have for all $t\geq 0$
\begin{align}
  \label{e:rect-prop-1}
\diam \varphi^t(\mathcal R^-_{q_0,\eta_1^0,\alpha})&\leq C\alpha e^t,\\
  \label{e:rect-prop-2}
\diam \varphi^{-t}(\mathcal R^+_{q_0,y_1^0,\alpha})&\leq C\alpha e^t.
\end{align}
\end{lemm}
\begin{proof}
1. We show~\eqref{e:rect-prop-1}, with~\eqref{e:rect-prop-2} proved similarly.
Take arbitrary $q$ such that $\varkappa_0(q)\in \{|y|+|\eta|\leq\alpha\}$.
We will estimate the images of the coordinate vector fields
by the map $d\varphi^t(q)d\varkappa_0(q)^{-1}:\mathbb R^{4n}\to T_{\varphi^t(q)}(T^*M)$. We first have
\begin{equation}
  \label{e:rect-int-1}
|d\varphi^t(q)d\varkappa_0(q)^{-1}\partial_{\eta_1}|\leq Ce^{2t}. 
\end{equation}
This follows from the general bound $\|d\varphi^t(q)\|\leq Ce^{2t}$,
which in turn follows from~\eqref{e:stun-exp}
and the fact that $d\varphi^t$ preserves the vector fields
$\zeta\cdot\partial_\zeta$ and~$X$.

We next have
\begin{equation}
  \label{e:rect-int-2}
W\in \{\partial_{y_1},\dots,\partial_{y_{2n}},\partial_{\eta_2},\dots,\partial_{\eta_{2n}}\}
\ \Rightarrow\
|d\varphi^t(q)d\varkappa_0(q)^{-1}W|\leq C\alpha e^{2t}+Ce^t.
\end{equation}
Indeed, since $d\varkappa_0(q^0)^{-1}W\in V_\perp^-(q^0)$
and $d(q,q^0)\leq C\alpha$, we can write
$$
d\varkappa_0(q)^{-1}W=cV^-(q)+W_\perp\quad\text{where }
W_\perp\in V_\perp^-(q),\
|c|\leq C\alpha.
$$
Using~\eqref{e:stun-exp} again, we see that
$$
|d\varphi^t(q)V^-(q)|\leq Ce^{2t},\quad
|d\varphi^t(q)W_\perp|\leq Ce^t,
$$
which gives~\eqref{e:rect-int-2}.

\noindent 2. Take arbitrary $q^1,q^2\in \mathcal R^-_{q^0,\eta_1^0,\alpha}$. Define the path
$q(s)\in T^* M$, $0\leq s\leq 1$, by the formula
$$
\varkappa_0(q(s))=(1-s)\varkappa_0(q^1)+s\varkappa_0(q^2).
$$
Then
\begin{equation}
  \label{e:rect-int-3}
\begin{aligned}
d(\varphi^t(q^1),\varphi^t(q^2))&=d\big(\varphi^t(q(0)),\varphi^t(q(1))\big)
\leq \max_{0\leq s\leq 1}|\partial_s\varphi^t(q(s))|\\
&= \max_{0\leq s\leq 1}
\big|d\varphi^t(q(s))d\varkappa_0(q(s))^{-1} (\varkappa_0(q^2)-\varkappa_0(q^1))\big|.
\end{aligned}
\end{equation}
From the definition of $\mathcal R^-_{q^0,\eta_1^0,\alpha}$ we see that
$$
\varkappa_0(q^2)-\varkappa_0(q^1)=\sum_{j=1}^{2n}(a_j \partial_{y_j}+b_j\partial_{\eta_j})\quad\text{with }|a_j|,|b_j|\leq 2\alpha,\ |b_1|\leq 2\alpha^2.
$$
We can now estimate the right-hand side of~\eqref{e:rect-int-3}
using~\eqref{e:rect-int-1}
and~\eqref{e:rect-int-2},
which gives
$$
\diam\varphi^t(\mathcal R^-_{q_0,\eta_1^0,\alpha})\leq C\alpha e^t+C\alpha^2e^{2t}.
$$
Since the diameter on the left-hand side is also bounded above by a fixed constant independent of $\alpha,\eta_1^0$
(as $S^*M$ is compact), we obtain~\eqref{e:rect-prop-1}.
\end{proof}

\section{Classifying orbit closures in $SM$}
\label{s:orbits-total}

In this section we assume that $M$ is a compact complex hyperbolic quotient
and study the closure of the orbit of a point on $SM$ under the fast unstable/stable flow
$e^{sV^\pm}$ together with the geodesic flow $e^{tX}$.
Using Ratner's theorem, we
show that each such orbit closure is algebraic and coincides with the unit sphere bundle of a
compact immersed totally geodesic complex submanifold on $M$; this is the content of Theorem~\ref{theo:orbitclosure}
stated in~\S\ref{s:orbit-closure-stmt} and proved in the rest of this section. In \S\ref{sec:geodesicsubman}, we discuss examples of complex hyperbolic manifolds which have differing behaviors with respect to their complex totally geodesic submanifolds.
Before embarking upon this, we give a preliminary section,
on orbits of vector fields.

\subsection{Orbits and segments}
\label{s:orbits}

Let $\mathcal M$ be a compact manifold and $V\in C^\infty(\mathcal M;T\mathcal M)$ be a
nonvanishing vector field. Let $e^{tV}:\mathcal M\to\mathcal M$ be the flow of~$V$. We first
make a few definitions:
\begin{itemize}
\item for $T\geq 0$, a \emph{$V$\!-segment} of length $T$ is a set of the form
$\{e^{tV}(q)\mid 0\leq t\leq T\}$ where $q\in \mathcal M$;
\item a \emph{$V$\!-orbit} is a set of the form
$\{e^{tV}(q)\mid t\in\mathbb R\}$ where $q\in \mathcal M$;
\item a set $\mathcal U\subset \mathcal M$ is called \emph{$V$\!-dense} if it
intersects every $V$\!-orbit.
\end{itemize}
Note that if $\mathcal U$ is open, then it is $V$\!-dense if and only if it intersects the closure of every $V$\!-orbit in $\mathcal M$.

The next lemma establishes basic properties of $V$\!-dense sets:
\begin{lemm}
\label{l:dense-basic}
Assume that $\mathcal U$ is a $V$\!-dense open set. Then:
\begin{enumerate}
\item
there exists a $V$\!-dense compact set $K\subset \mathcal U$;
\item
there exists $T>0$ such that each $V$\!-segment of length $T$ intersects $\mathcal U$.
\end{enumerate}
\end{lemm}
\begin{proof}
The set $\mathcal U$ is $V$\!-dense if and only if
\begin{equation}
  \label{e:dense-basic}
\mathcal M=\bigcup_{t\in\mathbb R}e^{tV}(\mathcal U).
\end{equation}
Take a nested sequence of open sets $\mathcal U_1\subset \mathcal U_2\subset\dots$ such that
$$
\overline{\mathcal U_j}\subset \mathcal U,\quad
\mathcal U=\bigcup_{j\geq 1} \mathcal U_j.
$$
Since $\mathcal U$ is $V$\!-dense, we have
$$
\mathcal M=\bigcup_{j\geq 1}\widehat{\mathcal U}_j\quad\text{where}\quad
\widehat{\mathcal U}_j:=\bigcup_{t\in\mathbb R}e^{tV}(\mathcal U_j).
$$
Since $\widehat{\mathcal U}_j$ is a nested sequence of open sets and $\mathcal M$ is compact,
there exists $j$ such that $\widehat{\mathcal U}_j=\mathcal M$. Putting $K:=\overline{\mathcal U_j}$, we obtain property~(1).

To show property~(2), we rewrite~\eqref{e:dense-basic} as
$$
\mathcal M=\bigcup_{T\geq 0} \widetilde{\mathcal U}_T\quad\text{where}\quad
\widetilde{\mathcal U}_T:=\bigcup_{|t|\leq T/2}e^{tV}(\mathcal U).
$$
Since $\widetilde{\mathcal U}_T$ is a nested family of open sets and $\mathcal M$ is compact, there exists
$T$ such that $\mathcal M=\widetilde{\mathcal U}_T$. Then each $V$\!-segment of length $T$ intersects $\mathcal U$.
\end{proof}
We also give an analog of~\cite[Lemma~3.5]{highcat}, using partitions of unity.
\begin{lemm}
\label{l:partition-construction}
Let $\mathcal U\subset\mathcal M$ be a $V$\!-dense open set. Then there exist
$\chi_1,\chi_2\in C^\infty(\mathcal M)$ such that
\begin{equation}
\chi_1,\chi_2\geq 0,\quad
\chi_1+\chi_2=1,\quad
\supp \chi_1\subset \mathcal U,
\end{equation}
and the complements $\mathcal M\setminus \supp \chi_1$, $\mathcal M\setminus\supp \chi_2$
are both $V$\!-dense.
\end{lemm}
\begin{proof}
Let $D\subset \mathcal M$ be a Poincar\'e section for $V$, that is a finite union of compact
embedded disks of codimension 1 which are transverse to $V$ and such that
$D$ is $V$\!-dense. To construct $D$, one can for example take a covering of
$\mathcal M$ by finitely many coordinate charts in each of which $V=\partial_{x_1}$.

The set $\mathcal U\setminus D$ is $V$\!-dense: indeed, for each $q\in\mathcal M$ the set
$\{t\in\mathbb R\mid e^{tV}(q)\in \mathcal U\}$ is open and nonempty,
while the set $\{t\in\mathbb R\mid e^{tV}(q)\in D\}$ is discrete since
$V$ is transverse to $D$. Since $\mathcal U\setminus D$ is also open, by Lemma~\ref{l:dense-basic}(1)
there exists a compact $V$\!-dense set $K\subset \mathcal U\setminus D$.

The sets $\mathcal U\setminus D$, $\mathcal M\setminus K$
form an open cover of $\mathcal M$. Using a partition of unity, we construct
$\chi_1,\chi_2\in C^\infty(\mathcal M)$ such that $\chi_1,\chi_2\geq 0$, $\chi_1+\chi_2=1$, and
$$
\supp \chi_1\subset \mathcal U\setminus D,\quad
\supp \chi_2\subset \mathcal M\setminus K.
$$
The complements $\mathcal M\setminus \supp \chi_1$, $\mathcal M\setminus \supp \chi_2$ contain the sets $D,K$ and thus are $V$\!-dense.
\end{proof}
\Remark We can instead consider a finite collection $V_1,\dots,V_q$ of nonvanishing vector fields
on $\mathcal M$. Lemma~\ref{l:partition-construction}
still holds if we replace the property of being $V$\!-dense by the property of being $V_\ell$-dense for all~$\ell=1,\dots,q$.
The only adjustment needed is in the construction of~$D$, which can still be done since
a collection of codimension 1 embedded disks in generic directions centered at a sufficiently large finite
set of points will be $V_\ell$-dense and transverse to $V_\ell$ for all $\ell$. This is the version
of Lemma~\ref{l:partition-construction} that we use in the proof of Lemma~\ref{l:our-partition} below.

\subsection{Statement of the orbit closure result}
\label{s:orbit-closure-stmt}

Let $M$ be a compact complex hyperbolic quotient (see~\S\ref{s:quotients-def}).
Recall the vector fields $X,V^+,V^-$ on the sphere bundle~$SM$ (see~\S\ref{s:complex-stun}),
generating the geodesic flow, the fast stable horocyclic flow, and
the fast unstable horocyclic flow respectively. The main result of this section is
\begin{theo}\label{theo:orbitclosure}
Let $(z_0,v_0)\in SM$. Then there exists a compact immersed totally geodesic complex submanifold $\Sigma\subset M$
such that $(z_0,v_0)\in S\Sigma$ and the closure of the orbit $\{e^{tX}e^{sV^+}(z_0,v_0)\mid t,s\in\mathbb R\}$
in~$SM$ is equal to $S\Sigma$. The same holds when $V^+$ is replaced by $V^-$.
\end{theo}
\Remark If $\Sigma\subset M$ is a compact immersed totally geodesic complex submanifold,
then the vector fields $X,V^+,V^-$ are tangent to $S\Sigma$ (see~\S\ref{s:tot-geod-sub} below).
Therefore, any compact immersed totally geodesic complex submanifold whose unit sphere bundle contains $(z_0,v_0)$ also contains the closure of $\{e^{tX}e^{sV^\pm}(z_0,v_0)\mid t,s\in\mathbb R\}$.
Consequently, the manifold~$\Sigma$ in Theorem~\ref{theo:orbitclosure} is characterized as the \emph{minimal} compact immersed totally geodesic complex submanifold of $M$
such that $(z_0,v_0)\in S\Sigma$.
Note that we allow for the possibility that $\Sigma=M$.

In this paper (specifically in~\S\ref{s:reduction-control} below) we will use the following corollary of Theorem~\ref{theo:orbitclosure}:
\begin{corr}
\label{c:orbitclosure}
Assume that $\mathcal U\subset SM$ is an open set invariant under the geodesic flow $\varphi^t=e^{tX}$. Then either
$\mathcal U$ is both $V^+$\!-dense and $V^-$\!-dense (in the sense of~\S\ref{s:orbits}), or
there exists a compact immersed totally geodesic complex submanifold $\Sigma\subset M$ such that
$\mathcal U\cap S\Sigma=\emptyset$.
\end{corr}
\begin{proof}
Assume for example that $\mathcal U$ is not $V^+$\!-dense (the case when $\mathcal U$
is not $V^-$\!-dense is handled in the same way). Then there exists $(z_0,v_0)\in SM$
such that $\mathcal U$ does not intersect the orbit $\{e^{sV^+}(z_0,v_0)\mid s\in\mathbb R\}$.
Since $\mathcal U$ is $e^{tX}$-invariant, we see that it does not intersect
the set $\{e^{tX}e^{sV^+}(z_0,v_0)\mid t,s\in\mathbb R\}$ and, as $\mathcal U$ is open,
it does not intersect the closure of this set in~$SM$. By Theorem~\ref{theo:orbitclosure}
we see that there exists a compact immersed totally geodesic complex submanifold $\Sigma\subset M$
such that $\mathcal U\cap S\Sigma=\emptyset$.
\end{proof}

\subsection{Orbit closures in \texorpdfstring{$\Gamma\backslash G$}{\unichar{"0393} \unichar{"005C} G}}
\label{s:orbits-frame}

In this section, we reduce Theorem~\ref{theo:orbitclosure} to a statement
about orbit closures on the quotient $\Gamma\backslash G$,
where $M=\Gamma\backslash \mathbb{CH}^n$ as in~\S\ref{s:quotients-def}
and $G=\SU(n,1)$ as in~\S\ref{s:isometry-group}. Note that $\Gamma\backslash G$
is a quotient of a Lie group by a lattice; this is the setting of Ratner theory, which will be crucially used in our proofs.

\subsubsection{Subgroups of $G$}
\label{s:subgroups-of-G}

We first introduce some subgroups of $G$ used throughout the rest of this section.
Let $U^\pm,A\subset G$ be the one-parameter subgroups
generated by the elements $V^\pm,X\in \mathfrak g$ defined in~\eqref{e:lie-basis} so that
\begin{equation}\label{eqn:unipotent}
\begin{aligned}
U^\pm&=\left\{\begin{pmatrix}
1+is &\mp is& 0    \\
\pm is & 1-is  & 0    \\
0 & 0  &  I_{n-1}   \\
\end{pmatrix}\colon s\in\R\right\},\\
A&=\left\{\begin{pmatrix}
\cosh t & \sinh t & 0 \\
\sinh t & \cosh t & 0 \\
0 & 0 & I_{n-1} 
\end{pmatrix}\colon t\in\mathbb R\right\}.
\end{aligned}
\end{equation}
Then $U^\pm$ and $A$ commute with the group $R$ defined in~\eqref{e:R-iso-def}. The right actions of $U^\pm$ and $A$ on $S\mathbb{CH}^n$ and~$SM$ define the flows of the vector fields $V^\pm$ and $X$ descended to these quotients. We note that $U^\pm$ are \emph{unipotent subgroups},
more precisely $(I-B)^2=0$ for all $B\in U^\pm$. Moreover, as follows from the commutation relations~\eqref{e:comm-rel}, $A$ normalizes $U^\pm$ and thus $AU^\pm$ are subgroups of~$G$.

We now introduce the \emph{standard subgroups} of $G$. For each $1\le k\le n$, let $W_k$ denote an isomorphic copy of $\SU(k,1)$ embedded in $G=\SU(n,1)$ in the upper left corner, so that
\begin{equation}\label{eqn:complexW}
W_k=\left\{\left.\begin{pmatrix}
B&0\\
0&I_{n-k}\end{pmatrix}\,\right|\,B\in\SU(k,1)\right\}.
\end{equation}
Note that $W_k^{\R}:=W_k\cap\mathop{\mathrm{GL}}_{n+1}(\R)$ is isomorphic to a copy of $\SO(k,1)$ embedded in the upper left corner.
Let $W$ be a subgroup of $G$, then we call $W$ \emph{standard} if $W$ is either equal to $W^{\R}_k$ for some $2\le k\le n$ or equal to $W_k$ for some $1\le k\le n$.
In the latter case, we call $W$ a \emph{complex standard subgroup} of $G$.
Note that the subgroups $U^\pm,A$ defined above all lie inside $W_1\simeq \SU(1,1)$.

The normalizer of the complex standard subgroup $W_k$ in $G$ is given by
\begin{equation}
  \label{e:W-k-normalizer}
N_G(W_k)=\left\{
\left.\begin{pmatrix} B & 0 \\ 0 & C \end{pmatrix}\,\right|\,
B\in \Un(k,1),\quad
C\in \Un(n-k),\quad
\det B\det C=1
\right\}.
\end{equation}
Note that $N_G(W_k)=W_kC_G(W_k)$ where the centralizer of $W_k$ in $G$ is given by
$$
C_G(W_k)=\left\{
\left.\begin{pmatrix} e^{i\theta}I_{k+1} & 0 \\ 0 & C \end{pmatrix}\,\right|\,
C\in \Un(n-k),\quad
e^{i(k+1)\theta}\det C=1
\right\}.
$$

\subsubsection{Totally geodesic submanifolds}
\label{s:tot-geod-sub}

Any totally geodesic subspace of $\mathbb{CH}^n$ of real dimension at least $2$ is either isometric to real hyperbolic space $\mathbb{H}^k$ for $2\le k\le n$ or complex hyperbolic space $\mathbb{CH}^k$ for $1\le k\le n$, see~\cite[\S\S 3.1.11]{Goldman-Book}.
Identifying $\mathbb{CH}^n$ with $G/K$, we now recall the dictionary between these geodesic planes and certain orbits of the form $\widetilde\pi_K(g_0W)$ where $g_0\in G$, $W\subset G$ is a standard subgroup, and
$\widetilde\pi_K:G\to \mathbb{CH}^n$ is the projection map from~\eqref{eqn:commtriangle}. 
For more details, see the discussion in \cite[\S 2]{BFMS}
and~\cite[Lemma 8.2(1)]{BFMS}.

Given a standard subgroup $W\subset G$ and any $g_0\in G$, the coset projection $\widetilde\pi_K(g_0W)\subset \mathbb{CH}^n$ is either a totally geodesic copy of real hyperbolic space $\mathbb{H}^k$ when $W=W^{\R}_k$ or a totally geodesic copy of complex hyperbolic space $\mathbb{CH}^k$ when $W=W_k$.
Conversely, any totally geodesic copy of one of these planes is of the form $\widetilde\pi_K(g_0W)$ for some standard subgroup $W$.
Note that we allow for the case that $W=W_n=G$, when $\widetilde\pi_K(g_0W)=\mathbb{CH}^n$.

Let $M=\Gamma\backslash\mathbb{CH}^n$ be a compact complex hyperbolic quotient as in~\S\ref{s:quotients-def}
and the maps $\pi_K,\pi_R$ be the projections from~\eqref{eqn:commtriangle}. 
If $\Sigma\subset M$ is a connected compact immersed totally geodesic submanifold of real dimension at least $2$, then
$$
\Sigma=\pi_K(x_0W),
$$
for some $x_0\in \Gamma\backslash G$ and some standard subgroup $W$.
Moreover, $\Sigma$ is a complex submanifold if and only if $W=W_k$ for some $1\le k\le n$ and otherwise $W=W_k^{\R}$.

Given a standard subgroup $W\subset G$ and $g_0\in G$,  the inclusion $\iota:\widetilde\pi_K(g_0W)\hookrightarrow \mathbb{CH}^n$ induces an embedding of tangent bundles $d\iota: T(\widetilde\pi_K(g_0W))\hookrightarrow T\mathbb{CH}^n$.
Since this embedding preserves the norm of vectors, $d\iota$ induces an embedding of unit tangent bundles $d\iota^1:S(\widetilde\pi_K(g_0W))\to S\mathbb{CH}^n$.
The image of this embedding is
\begin{equation}
\label{eqn:diota}
S(\widetilde\pi_K(g_0W))=\widetilde\pi_R(g_0W)\ \subset\ S\mathbb{CH}^n.
\end{equation}
These maps are natural with respect to the covering projections $\pi_M,\pi_{SM}$.
In particular, if $\Sigma=\pi_K(x_0W)$ is a compact immersed totally geodesic submanifold of~$M$ then we have an immersion
\begin{equation}\label{eqn:sphereinclusion}
S\Sigma=\pi_R(x_0W)\ \subset\ SM,
\end{equation}
induced from the inclusion $\Sigma\subset M$.

As a consequence of~\eqref{eqn:sphereinclusion}, we see that
the vector fields $X,V^\pm$ on $SM$ are tangent to $S\Sigma$, since
they lie in the Lie algebra
of the groups $W_k$ for all $k\geq 1$.

\subsubsection{Results on $\Gamma\backslash G$ and proof of Theorem~\ref{theo:orbitclosure}}

We now state two propositions regarding orbit closures on~$\Gamma\backslash G$, whose proofs are given in~\S\ref{sec:uniporbit}--\ref{s:AU-orbits} below.
The first one gives a description of orbit closures of the standard group $W_1\simeq \SU(1,1)$
introduced in~\eqref{eqn:complexW}:
\begin{prop}
  \label{l:W-1-orbits}
Let $x_0\in \Gamma\backslash G$. Then the orbit closure
$\overline{x_0W_1}$ in $\Gamma\backslash G$ is given by
\begin{equation}
  \label{e:W-1-orbits}
\overline{x_0W_1}=x_0H
\end{equation}
for some closed connected reductive subgroup $H\subset G$ containing $W_1$ and such that for some $1\leq k\leq n$
and $r_H\in R$ we have
\begin{equation}
  \label{e:H-squeezed}
W_k\ \subset\ r_H Hr_H^{-1}\ \subset\ N_G(W_k).
\end{equation}
\end{prop}
The second proposition states that the orbit closures for the groups $AU^+,AU^-\subset W_1$
coincide with the whole $W_1$-orbit closure
(in particular, the $AU^\pm$-closure is invariant under $U^\mp$):
\begin{prop}
  \label{l:AU-orbits-same}
Let $x_0\in \Gamma\backslash G$. Then we have the equality of closures in $\Gamma\backslash G$
\begin{equation}
  \label{e:AU-orbits-same}
\overline{x_0AU^+}=\overline{x_0AU^-}=\overline{x_0W_1}.
\end{equation} 
\end{prop}
Using the above two propositions, we give
\begin{proof}[Proof of Theorem~\ref{theo:orbitclosure}]
1. We give the proof in the case of~$V^+$; the case of $V^-$ is handled similarly.
We use the notation from~\eqref{eqn:commtriangle}.

Fix some $x_0\in \Gamma\backslash G$ such that $\pi_R(x_0)=(z_0,v_0)$. Since $\Gamma\backslash G$ is compact (as $M$ is compact),
the closure of the orbit of $(z_0,v_0)$ under $X,V^+$ in $SM$ is equal to the image
under $\pi_R$ of the closure of the $AU^+$-orbit of $x_0$ in~$\Gamma\backslash G$:
\begin{equation}
  \label{e:proveclos-1}
\overline{\{e^{tX}e^{sV^+}(z_0,v_0)\mid t,s\in\mathbb R\}}
=\overline{\pi_R(x_0AU^+)}
=\pi_R(\overline{x_0AU^+}).
\end{equation}
By Propositions~\ref{l:W-1-orbits} and~\ref{l:AU-orbits-same} this set is equal to
$$
\pi_R(\overline{x_0AU^+})=
\pi_R(\overline{x_0W_1})
=\pi_R(x_0H)
$$
for some closed subgroup $H\subset G$ such that $x_0H\subset\Gamma\backslash G$ is closed and
there exist some $1\leq k\leq n$ and $r_H\in R$ for which $W_k\subset r_H Hr_H^{-1}\subset N_G(W_k)$. 
We then have
\begin{equation}
  \label{e:closures-geom-1}
x_0 r_H^{-1}W_k r_H\ \subset\ x_0H\ \subset\ 
x_0 r_H^{-1}N_G(W_k)r_H.
\end{equation}
By~\eqref{e:W-k-normalizer} we have $N_G(W_k)\subset W_kR$, therefore
the images under $\pi_R$ of the first and the last sets in~\eqref{e:closures-geom-1}
are equal to each other. It follows that
\begin{equation}
  \label{e:closures-geom-2}
\pi_R(x_0H)=\pi_R(x_0 r_H^{-1}W_k).
\end{equation}

\noindent 2. Define
$$
\Sigma:=\pi_K(x_0 H)=\pi_K(x_0r_H^{-1}W_k).
$$
Then $\Sigma$ is a compact immersed totally geodesic complex submanifold of $M$
as explained in~\S\ref{s:tot-geod-sub}. From this and Equations~\eqref{e:closures-geom-2} and~\eqref{eqn:sphereinclusion}, one readily concludes that the closure~\eqref{e:proveclos-1} is equal to $S\Sigma$ as needed.
\end{proof}

\subsection{Unipotent orbit closures and proof of Proposition~\ref{l:W-1-orbits}}
\label{sec:uniporbit}

In this section, we review preliminaries from Lie theory and Ratner theory and
apply these to prove Proposition~\ref{l:W-1-orbits}. We also give a description of the $U^\pm$-orbits in Lemma~\ref{lemma:unipotentorbit} below. Using this description and an argument involving
Zariski density, we show in Lemma~\ref{lemma:smallunipotent} that if the closure $\overline{x_0 W_1}$ is as small as possible, that is, if it projects to a complex totally geodesic submanifold of complex dimension $1$ in $M$,  then the orbit closures $\overline{x_0 U^\pm}$ are equal to that of $\overline{x_0W_1}$.
This special case is the simplest setting for Proposition \ref{l:AU-orbits-same}, in that one does not need the additional $A$-action to obtain the required result.
The case where the orbit closure is bigger will be handled in \S\ref{s:AU-orbits}, where the $A$-invariance needs to be invoked.


\subsubsection{Preliminaries}

We first review some concepts from Lie theory:
\begin{itemize}
\item If $G'$ is a Lie group, then a discrete subgroup $\Gamma'\subset G'$
is called a \emph{lattice} in~$G'$ if there exists a probability measure on the quotient
$\Gamma'\backslash G'$ which is invariant under right multiplication by elements
of $G'$. If $\Gamma'\backslash G'$ is compact, then $\Gamma'$ is called a \emph{uniform lattice}.
We are studying a compact hyperbolic quotient $M=\Gamma\backslash\mathbb{CH}^n$,
thus $\Gamma$ is a uniform lattice in $G=\SU(n,1)$.
\item For a subgroup $J\subset G=\SU(n,1)$, we use the notation $J^\dagger$ to denote the subgroup of $J$ generated by unipotent elements.
Note that $J^\dagger$ is connected and
\begin{equation}
  \label{e:J-dagger-inclusions}
J^\dagger\ \subset\ J\ \subset\ N_G(J^\dagger).
\end{equation}
For our choice of $G$, $J^\dagger$ is either unipotent or a non-compact, almost simple subgroup of~$G$.
In the latter case, $J^\dagger$ will always be conjugate to a standard subgroup $W$ of~$G$ 
as defined in~\S\ref{s:subgroups-of-G}, see~\cite[Proposition 2.4]{BFMS}.
\item We have Iwasawa decompositions $G=KAN^\pm$ where $K$ is as in \eqref{e:K-iso-def}, $A$ is defined in~\eqref{eqn:unipotent}, and $N^\pm$ is the unique maximal unipotent subgroup containing~$U^\pm$. In fact, $N^\pm$ is the connected Lie group with the Lie algebra
$$
\mathfrak n^\pm:=\mathbb R V^\pm\oplus \widetilde E^\pm,
$$
where $V^\pm$ is defined in~\eqref{e:lie-basis} and $\widetilde E^\pm$ is defined in~\eqref{e:tilde-E-pm-def}. Note that $U^\pm$ is central in~$N^\pm$ by~\eqref{e:N+comm}.
\item We use $P^\pm$ to denote the unique proper parabolic subgroup of $G$ containing~$U^\pm$.
In particular, $P^\pm=N_G(U^\pm)=N_G(N^\pm)$ and $N^\pm$ is the unipotent radical of~$P^\pm$.
In terms of the action of $G$ on $\mathbb C^{n,1}$ we have by~\eqref{e:stun-matrix-action}
$$
P^\pm=\{B\in G\mid B(e_0\pm e_1)\in \mathbb C(e_0\pm e_1)\}.
$$
The Lie algebra of~$P^\pm$ is given by
$$
\mathfrak p^\pm:=\mathfrak n^\pm\oplus\mathbb RX\oplus \mathfrak r,
$$
where $\mathfrak r$ is the Lie algebra of~$R$.
\end{itemize}
We also have the following technical lemma.
\begin{lemm}
\label{l:U+unique}
Assume that $g\in G$ and $gU^+g^{-1}\subset N^+$. Then
$gU^+g^{-1}=U^+$. 
\end{lemm} 
\begin{proof}
Recall that $N^+$ is a maximal unipotent subgroup of $G$ and, since $G$ has rank~$1$, any other distinct maximal unipotent subgroup of $G$ intersects $N^+$ in the identity \cite[Lemma 12.15]{Rag}.
Therefore $gN^+g^{-1}\cap N^+=\{e\}$ or $gN^+g^{-1}=N^+$.
Since $gU^+ g^{-1}\subset N^+$ it follows that $gN^+g^{-1}=N^+$ and thus $g\in P^+=N_G(N^+)$.
However $U^+$ is normal in $P^+$ and so we conclude that $gU^+g^{-1}=U^+$ as required.

As an alternative proof, using~\eqref{e:N+product} one can characterize $U^+$
in terms of matrix powers as
$$
U^+=\{B\in N^+\mid (I-B)^2=0\},
$$
and $gU^+g^{-1}$ satisfies the same characterization.
\end{proof}

\subsubsection{Ratner theory}
\label{s:Ratner}

We will make heavy use of Ratner's Orbit Closure Theorem, which describes the closures of unipotent orbits on homogeneous spaces, tailored to our setting, via the following statement.
As in~\eqref{eqn:commtriangle}, denote by $\pi_\Gamma:G\to\Gamma\backslash G$ the projection map.
\begin{theo}\cite[Theorem~A, Corollary~A]{Ratner}
\label{t:Ratner}
Fix $g_0\in G$, let $x_0=\pi_\Gamma(g_0)$, and let $D$ be a subgroup of $G$ generated by unipotent elements.
Then there exists a closed subgroup $J\subset G$ containing $D$ such that the orbit closure $\overline{x_0 D}$ in $\Gamma\backslash G$ is equal to $x_0J$ and $D$ acts ergodically on $x_0J$.
Moreover, $g_0Jg_0^{-1}\cap\Gamma$ is a Zariski dense lattice in $g_0Jg_0^{-1}$.
\end{theo}
\noindent Note that the final statement in Theorem~\ref{t:Ratner} is not listed in \cite{Ratner} but can readily be deduced from ergodicity of the action, such as in~\cite[Corollary 2.13]{Shah}.
We also point out that when $D$ is connected, which will always be the case for us, the $J$ that appears in Theorem~\ref{t:Ratner} is connected as well.

\subsubsection{Closures of $U^\pm$-orbits}

The following lemma classifies $U^\pm$-orbit closures in $\Gamma\backslash G$.
It is stated for $U^+$ but a similar statement holds for $U^-$ as well.
However, the resulting groups $L$ for $U^+$ and $U^-$-orbits may be different.
Moreover, the presence of the element~$u\in N^+$ in~\eqref{e:L-contained}
means that we cannot use Lemma~\ref{lemma:unipotentorbit} in the proof of Theorem~\ref{theo:orbitclosure} directly and we cannot show that the closures
of $U^\pm$-orbits project to totally geodesic submanifolds.
This explains the need for the additional $A$ action in Theorem~\ref{theo:orbitclosure}. 
\begin{lemm}\label{lemma:unipotentorbit}
Let $x_0\in\Gamma\backslash G$.
Then the orbit closure $\overline{x_0U^+}$ in $\Gamma\backslash G$ is equal to $x_0L$ for some closed connected subgroup $L\subset G$ such that $U^+\subset L$.
Moreover $L$ is reductive, $L^\dagger$ is conjugate to a complex standard subgroup $W_\ell$ for some $\ell$, and there exists $u\in N^+$ for which $\overline{x_0U^+}$ is $uAu^{-1}$-invariant, that is, $uAu^{-1}\subset L$.
To be more precise, $L^\dagger=urW_\ell (ur)^{-1}$ for some $r\in R$ and therefore
\begin{equation}
\label{e:L-contained}
W_\ell\ \subset\ (ur)^{-1}Lur\ \subset\ N_G(W_\ell).
\end{equation}
\end{lemm}
\begin{proof}
The first statement is simply an application of Ratner's Theorem (Theorem~\ref{t:Ratner}) so it suffices to exhibit the others.
Fix $g_0\in G$ such that $x_0=\pi_\Gamma(g_0)$.

\noindent 1. We first claim that $L$ is reductive.
Indeed in \cite[Proposition 3.1]{Shah}, Shah shows that $L$ must either be unipotent or reductive with compact center under the additional assumption that $G$ is center free.
In our setting, where $G$ has center, it is straightforward to deduce from this that $L$ either has a finite index unipotent subgroup or is reductive with compact center in the following way.
By projecting to the adjoint group $\overline{G}=\PU(n,1)$, the argument in \cite[Corollaries 1.3, 1.4]{Shah} shows that either $L$ is reductive or $L=CU$, where $U$ is unipotent and $C$ is contained in the center of $G$.
In the latter case, $L$ contains a finite index subgroup, say $L'$, which is unipotent.
As this is a finite index subgroup, $g_0L' g_0^{-1}\cap\Gamma$ is also a lattice in $g_0L' g_0^{-1}$.
However this implies that $(g_0L' g_0^{-1}\cap\Gamma)\backslash g_0L' g_0^{-1}$ is compact \cite[Theorem 2.1]{Rag}and, in particular, $g_0L' g_0^{-1}\cap\Gamma$ is infinite.
This would force $\Gamma$ to contain a non-trivial unipotent element.
However $\Gamma\backslash G$ is compact, and hence $\Gamma$ cannot contain any nontrivial unipotent elements (see e.g.~\cite[Lemma~1]{Kazdan-Margulis}), a contradiction.

\noindent 2. To see the second claim, note that $L$ is reductive and contains the non-compact group $U^+$, therefore it must be of real rank $1$.
Since $U^+$ is unipotent, it also must be the case that $U^+\subset L^\dagger$.
As $L$ is reductive, $L^\dagger$ is a connected almost simple subgroup of $G$ and therefore is conjugate to a standard subgroup $W$; that is,
\begin{equation}
  \label{e:L-dagger-form}
L^\dagger=b^{-1}Wb\quad\text{for some }b\in G.
\end{equation}
As proper parabolic subgroups of $W$ are minimal parabolics, $W$ acts transitively on them by conjugation.
Therefore we may assume that $b$ is such that $bU^+b^{-1}\subset W\cap P^+$, as the latter is a proper parabolic subgroup of $W$.
Since $W$ is real rank $1$, all of the unipotent elements of $W\cap P^+$ are contained in its unipotent radical $W\cap N^+$, therefore it moreover follows that $b$ is such that $b U^+b^{-1}\subset W\cap N^+$.
By Lemma~\ref{l:U+unique}, we have
\begin{equation}
  \label{e:B-U+comm}
bU^+b^{-1}=U^+,
\end{equation}
and thus $U^+\subset W$, from which it follows that $W=W_\ell$ for some complex standard subgroup $W_\ell$ and some $\ell\in\{1,\dots,n\}$.

\noindent 3. Continuing to the final claim, by~\eqref{e:B-U+comm}
we have $b\in P^+=N_G(U^+)$.
Since $P^+$ has Langlands decomposition $P^+=RAN^+$,\footnote{In the literature, typically one writes the Langlands decomposition using the letter $M$ instead of~$R$, however we want to avoid the notational conflict with $M$ as our manifold.}
and since $RA=AR$, we may write
$$
bur=a\quad\text{for some }
a\in A,\ r\in R,\ u\in N^+.
$$
Since $a$ and $r$ commute with $A$, we have
$$
uAu^{-1}=b^{-1}Ab\ \subset\ b^{-1}W_\ell b=L^\dagger\ \subset\ L.
$$
Moreover, since $a\in W_\ell$, we have from~\eqref{e:L-dagger-form}
$$
L^\dagger=ur W_\ell (ur)^{-1}.
$$
Now the containment~\eqref{e:L-contained} follows from~\eqref{e:J-dagger-inclusions}.
\end{proof}

\subsubsection{Closures of $W_1$-orbits and proof of Proposition~\ref{l:W-1-orbits}}

We now give the proof of Proposition~\ref{l:W-1-orbits} on the closures of $W_1$-orbits
in $\Gamma\backslash G$.
\begin{proof}[Proof of Proposition~\ref{l:W-1-orbits}]
As $W_1$ is generated by unipotents, by Ratner's Theorem (Theorem~\ref{t:Ratner}) we have $\overline{x_0W_1}=x_0H$ for some closed subgroup $H\subset G$ containing~$W_1$.
Similar to the proof of Lemma~\ref{lemma:unipotentorbit}, since $U^+\subset W_1\subset H$ it follows that $H^\dagger$ is conjugate to a complex standard subgroup $W_k$. We now show that this conjugation can be achieved by an element of $R$.

Let $b\in G$ be such that $W_k= bH^\dagger b^{-1}$, then $bW_1b^{-1}\subset W_k$.
Since all copies of $W_1$ contained in $W_k$ are conjugate in $W_k$, there exists some $w\in W_k$ for which $wbW_1(wb)^{-1}=W_1$. See for instance \cite[Proposition 2.4]{BFMS}, applied when $G=W_k$.
It follows that $wb\in N_G(W_1)=RW_1$.
Therefore there exists $r_H\in R$ and $w'\in W_1\subset H^\dagger$ such that $wb=r_Hw'$.
Hence
$$r_HH^\dagger r_H^{-1}=r_Hw'H^\dagger (r_Hw')^{-1}=wbH^\dagger (wb)^{-1}=wW_k w^{-1}=W_k,$$
as required.
\end{proof}

Note that this argument appears in \cite[Lemma 2.7(4)]{BFMS}. There it is only claimed that $r_H\in K$, however the proof gives the stronger results that $r_H\in R$.


\subsubsection{More on orbit closures}

We now give two lemmas which show that if two orbit closures have the same almost simple component, then they are equal.
We briefly remark that Lemmas \ref{lemma:zariskiclosure} and \ref{lemma:orbitclosureequality} hold for any simple real rank $1$ Lie group, however they fail in higher rank.
\begin{lemm}\label{lemma:zariskiclosure}
Suppose that $J_1$, $J_2$ are connected non-compact reductive subgroups of~$G$ for which $\pi_\Gamma(J_1)$, $\pi_\Gamma(J_2)$ are closed subsets of $\Gamma\backslash G$ and $J_1\cap\Gamma$, $J_2\cap\Gamma$ are Zariski dense in $J_1$, $J_2$ (respectively).
Then
$$
N_G(J_1^\dagger)=N_G(J_2^\dagger)\quad\Longrightarrow\quad
J_1=J_2.
$$
In particular, any closed subset $\pi_\Gamma(J)$ of $\Gamma\backslash G$ for which $J\cap\Gamma$ is Zariski dense in $J$ is uniquely determined by $N_G(J^\dagger)$. 
\end{lemm}
\begin{proof}
Since $J_i^\dagger$ is cocompact in $N_G(J_i^\dagger)$, it follows that $\pi_\Gamma(N_G(J_i^\dagger))$ is closed and therefore $ N_G(J_i^\dagger)\cap \Gamma$ is a lattice in $N_G(J_i^\dagger)$.
Moreover $ J_i\cap\Gamma$ is finite index in $ N_G(J_i^\dagger)\cap\Gamma$ and hence $J_i=\overline{ J_i\cap\Gamma}$ is finite index in $\overline{ N_G(J_i^\dagger)\cap\Gamma}$, where this closure is with respect to the Zariski topology.
Since $J_i$ is a connected and non-compact subgroup of $N_G(J_i^\dagger)$, it coincides with the identity component of $\overline{ N_G(J_i^\dagger)\cap\Gamma}$, and since $N_G(J_1^\dagger)=N_G(J_2^\dagger)$, we conclude from this description of $J_i$ that $J_1=J_2$.
\end{proof}
%
\begin{lemm}\label{lemma:orbitclosureequality}
Fix $g_0\in G$, let $x_0=\pi_\Gamma(g_0)$, and suppose that $J_1$, $J_2$  are connected non-compact reductive subgroups of $G$ for which $x_0J_1$, $x_0J_2$ are closed subsets of $\Gamma\backslash G$ and $g_0J_1g_0^{-1}\cap\Gamma$, $g_0J_2g_0^{-1}\cap\Gamma$ are Zariski dense in $g_0J_1g_0^{-1}$, $g_0J_2g_0^{-1}$ (respectively).
Then
$$
N_G(J^\dagger_1)=N_G(J^\dagger_2)\quad\Longrightarrow\quad
J_1=J_2.
$$
\end{lemm}
\begin{proof}
Writing $J'_1=g_0J_1g_0^{-1}$ and $J'_2=g_0J_2g_0^{-1}$, we conclude that $N_G((J'_1)^\dagger)=N_G((J'_2)^\dagger)$ and that $J'_1\cap\Gamma$, $J_2'\cap\Gamma$ are Zariski dense in $J'_1$, $J'_2$ (respectively).
Moreover, note that
$$
x_0J_i=\pi_\Gamma(g_0 J_i)=\pi_\Gamma(J'_i)g_0,
$$
and therefore $\pi_\Gamma(J'_i) g_0$ and hence $\pi_\Gamma( J'_i)$ are closed subsets of $\Gamma\backslash G$ for each $i\in\{1,2\}$.
Applying Lemma~\ref{lemma:zariskiclosure} to the latter, we find that $J'_1=J'_2$ and hence $J_1=J_2$.
\end{proof}
We point out to the reader that the previous lemma will apply to subgroups in the class $\mathcal{H}_{g_0H}$ defined in~\S\ref{s:singular-set} below.

As a consequence of Lemmas~\ref{l:W-1-orbits}, \ref{lemma:unipotentorbit}, and~\ref{lemma:orbitclosureequality}, we will now show that if the orbit closure of $\overline{x_0W_1}$ is as small as possible, then the orbit closures of $\overline{x_0 U^\pm}$ and $\overline{x_0W_1}$ coincide.
\begin{lemm}\label{lemma:smallunipotent}
Fix $x_0\in\Gamma\backslash G$ and write $\overline{x_0W_1}=x_0H$ with $W_k\subset r_HHr_H^{-1}\subset N_G(W_k)$ as in Proposition~\ref{l:W-1-orbits}.
If $k=1$, then $\overline{x_0U^\pm}=\overline{x_0W_1}$.
\end{lemm}
\begin{proof}
Fix $g_0\in G$ such that $\pi_\Gamma(g_0)=x_0$. We consider the case of $U^+$, with
$U^-$ handled similarly. Since $k=1$ and $r_H\in R$ centralizes $W_1$,
we have $W_1\subset H\subset N_G(W_1)$ and thus $H^\dagger=W_1$.
Let $\overline{x_0U^+}=x_0L$ for $L$ as in Lemma \ref{lemma:unipotentorbit}.
Then, as in that lemma, $L^\dagger=urW_\ell (ur)^{-1}$ for some $r\in R$, some $u\in N^+$, and some complex standard subgroup $W_\ell$.
As $\overline{x_0U^+}\subset \overline{x_0 W_1}$ it follows that $L\subset H$ and therefore $L^\dagger\subset H^\dagger$.
Hence $urW_\ell (ur)^{-1}\subset W_1$ and by dimension considerations it follows that $\ell=1$ and $ur\in N_G(W_1)$.
In particular, we obtain the equalities $L^\dagger=H^\dagger=W_1$.
Therefore $L$ and $H$ fit all of the hypotheses of Lemma~\ref{lemma:orbitclosureequality}
(with Zariski density following from Ratner's Theorem~\ref{t:Ratner}).
Consequently $L=H$ and we conclude that indeed $\overline{x_0U}=\overline{x_0W_1}$.
\end{proof}

\subsection{\texorpdfstring{$AU^\pm$}{AU\unichar{"00B1}}-orbit closures and proof of Proposition~\ref{l:AU-orbits-same}}
\label{s:AU-orbits}

In this subsection, we show that the closures of $AU^\pm$-orbits coincide with the closures
of $W_1$-orbits, proving Proposition~\ref{l:AU-orbits-same}. We focus
on the case of $AU^+$, with the case of $AU^-$ handled in the same way.
Note that Ratner's Theorem (Theorem~\ref{t:Ratner}) does not apply to the group
$AU^+$ since it is not generated by unipotents (we have $(AU^+)^\dagger=U^+$).
Our proof uses the fact that $W_1/AU^+$ is compact to show
that for any $x_0\in \Gamma\backslash G$, the orbit closure
$\overline{x_0AU^+}$ contains a point $y$ such that
$\overline{yU^+}=\overline{x_0W_1}$.

\subsubsection{The singular set}
\label{s:singular-set}

Fix $x_0=\pi_\Gamma(g_0)\in \Gamma\backslash G$. By Proposition~\ref{l:W-1-orbits}
we have $\overline{x_0 W_1}=x_0H$ where $H^\dagger=r_H^{-1}W_k r_H$ for
some $1\leq k\leq n$ and $r_H\in R$. 
For any $y\in x_0H$,
we have $yU^+\subset x_0H$ (as $U^+\subset W_1\subset H$) and thus $\overline{yU^+}\subset x_0H$ as well. We say that $y$ is a \emph{regular} point if the closure
$\overline{yU^+}$ is equal to the whole $x_0H$ and a \emph{singular} point otherwise.

The aim of this section is to obtain a description of the set of singular points,
with obstructions to equidistribution of the orbit $yU^+$ in $x_0H$
given by certain intermediate subgroups~-- see~\eqref{e:singular-described} below. Our discussion is inspired by and follows closely~\cite[\S 5]{LeeOh} and mimics the proof of~\cite[Proposition 2.3]{DM}.

Following~\cite{DM}, for a subgroup $J\subset G$ define the set
\begin{equation}
  \label{e:X-set-def}
X(J,U^+):=\{g\in G\mid gU^+g^{-1}\subset J\}.
\end{equation}

Let us take $y=x_0h=\pi_\Gamma(g_0h)$ for some $h\in H$. By Lemma~\ref{lemma:unipotentorbit} (with
$x_0$ replaced by $y$), we have
$\overline{yU^+}=yL$ for some closed connected reductive subgroup $L\subset G$ containing $U^+$ such that $L^\dagger$ is conjugate to~$W_\ell$ for some $1\leq \ell\leq n$.
Define
\begin{equation}
  \label{e:JJ-def}
J:=g_0hL(g_0h)^{-1}.
\end{equation}
Then $\pi_\Gamma(J)=yL(g_0h)^{-1}$ is a closed subset of $\Gamma\backslash G$.
Moreover, since $U^+\subset L$, we see that $g_0h\in X(J,U^+)$.

We now study the relation between the groups $H$ and~$L$.
Since $yL=\overline{yU^+}\subset x_0H=yH$, we have $L\subset H$
and thus $L^\dagger\subset H^\dagger$, which by dimensional considerations
implies that $\ell\leq k$. (Recall that
$L^\dagger$ is conjugate to~$W_\ell$ and $H^\dagger$
is conjugate to~$W_k$.) Moreover, if $\ell=k$ then $L^\dagger=H^\dagger$.
By Theorem~\ref{t:Ratner} we know that $g_0Hg_0^{-1}\cap\Gamma$ is Zariski
dense in $g_0Hg_0^{-1}$ and $J\cap \Gamma$ is Zariski dense in~$J$.
Note also that $(g_0h)H(g_0h)^{-1}=g_0Hg_0^{-1}$. Thus Lemma~\ref{lemma:orbitclosureequality}
(with $g_0$ replaced by $g_0h$) applied to the groups $L,H$ implies that
$$
\ell=k\quad\Longrightarrow\quad
L^\dagger=H^\dagger\quad\Longrightarrow\quad L=H\quad\Longrightarrow\quad \overline{yU^+}=x_0H.
$$
That is, if $\ell=k$ then $y$ is a regular point; equivalently, if $y$ is a singular point,
then $\ell<k$.

To extract a description of the set of singular points from the above discussion,
define $\mathcal H_{g_0H}$ to be the set of $J$ such that:
\begin{enumerate}
\item $J\subset G$ is a closed connected reductive subgroup;
\item $J$ contains a conjugate of~$U^+$;
\item $\pi_\Gamma(J)$ is a closed subset of $\Gamma\backslash G$;
\item $J\cap \Gamma$ is Zariski dense in~$J$;
\item $g_0^{-1}Jg_0\subset H$;
\item $J^\dagger$ is conjugate to $W_\ell$ for some $1\leq \ell<k$ (where $H^\dagger$ is conjugate to $W_k$).
\end{enumerate}
The set $\mathcal H_{g_0H}$ is countable by~\cite[Theorem 2]{Ratner2}, see also~\cite[Proposition~2.1]{DM} (for this we only need the properties (1)--(4) above). Now, define the set
\begin{equation}
  \label{e:singular-defined}
\mathcal S_{g_0H}=g_0H\cap\bigcup_{J\in \mathcal H_{g_0H}}X(J,U^+).
\end{equation}
Then the above discussion shows that the set of singular points is contained in $\pi_\Gamma(\mathcal S_{g_0H})$:
\begin{equation}
  \label{e:singular-described}
y\in x_0H\setminus \pi_\Gamma(\mathcal S_{g_0H})\quad\Longrightarrow\quad
\overline{yU^+}=x_0H.
\end{equation}
Indeed, if $y=x_0h=\pi_\Gamma(g_0h)$ for some $h\in H$ and $\overline{yU^+}\neq x_0H$, then the group
$J$ defined in~\eqref{e:JJ-def} lies in $\mathcal H_{g_0H}$
and we have $g_0h\in X(J,U^+)$, thus $g_0h\in \mathcal S_{g_0H}$.

\subsubsection{Nowhere density of singular sets}

We now show that for any $J\in \mathcal H_{g_0H}$, the set $g_0H\cap X(J,U^+)$
is nowhere dense in~$g_0H$. Recalling~\eqref{e:singular-described} where
the set $\mathcal H_{g_0H}$ is countable, we see from here that the set of singular points $y\in x_0H$ is a countable union of nowhere dense sets in $x_0H$ and thus (by the Baire category theorem) there exists a regular point in $x_0H$. Alternatively one could use the
concept of Lebesgue measure zero sets instead of nowhere dense sets.

It fact, we show a stronger statement that
the $W_1$-saturation of $g_0H\cap X(J,U^+)$ is nowhere dense in~$g_0H$,
which is needed in the proof of Lemma~\ref{cor:nowheredensesaturation} below.
Our proof follows the strategy of Lee--Oh \cite[\S 5]{LeeOh}, where similar arguments are given in a different, albeit related, context and with different proofs.

Before continuing to the argument, we make a few remarks. 
First, note that one can straightforwardly compute that 
\begin{equation}
\label{e:funfact-1}
bX(J,U^+)=X(bJb^{-1},U^+),
\end{equation}
for any $b\in G$.
Second, if $b\in P^+=N_G(U^+)$ then one can see that
\begin{equation}
\label{e:funfact-2}
X(J,U^+)b=X(J,U^+).
\end{equation}
In particular, the latter applies to any element of $R$.
Finally, we have the relationship that
\begin{equation}
\label{e:funfact-3}
X(J,U^+)=X(J^\dagger,U^+).
\end{equation}
Indeed, by definition if $gU^+g^{-1}\subset J$ then it must be the case that $gU^+g^{-1}\subset J^\dagger$ since the latter is the subgroup generated by unipotent elements in $J$.

The main result of this section is
\begin{lemm}\label{lemma:nowheredensity}
Let  $H\subset G$ be a subgroup such that $H^\dagger=r_H^{-1}W_kr_H$ for some $1\leq k\leq n$ and $r_H\in R$. Let also $g_0\in G$ and $J\subset g_0Hg_0^{-1}$ be a subgroup such that $J^\dagger$ is conjugate to $W_\ell$ for some $1\le \ell<k$. Then $\left(g_0H\cap X(J,U^+)\right)W_1$ is nowhere dense in~$g_0H$.
\end{lemm}
\begin{proof}
1. To simplify the situation, we will first argue that we can reduce to the case that $g_0=I$,
$H=W_k$, and $J=W_\ell$.
We will then argue the nowhere density in that specific case.

First of all, by~\eqref{e:funfact-1} we see that
$(g_0H\cap X(J,U^+))W_1$ is nowhere dense in $g_0H$ if and only if $(H\cap X(g_0^{-1}Jg_0,U^+))W_1$
is nowhere dense in $H$. Thus (replacing $J$ by $g_0^{-1}Jg_0$) we reduce to the case when $g_0=I$, $J\subset H$, and we need to show that $(H\cap X(J,U^+))W_1$ is nowhere dense in $H$.

Next, let $H'=r_HHr_H^{-1}$ and $J'=r_H J r_H^{-1}$.
Then $\left(H\cap X(J,U^+)\right)W_1$ is nowhere dense in $H$ if and only if $r_H\left(H\cap X(J,U^+)\right)W_1r_H^{-1}$ is nowhere dense in $r_HHr_H^{-1}=H'$.
Since $r_H\in R\subset C_G(W_1)$, we see that~\eqref{e:funfact-1}--\eqref{e:funfact-2} imply that
\begin{align*}
r_H\left(H\cap X(J,U^+)\right)W_1r_H^{-1}&=\left(r_HHr_H^{-1}\cap r_HX(J,U^+)r_H^{-1}\right)W_1,\\
&=\left(H'\cap X(r_H J r_H^{-1},U^+)\right)W_1,\\
&=\left(H'\cap X(J',U^+)\right)W_1.
\end{align*}
We therefore conclude that the nowhere density of $\left(H\cap X(J,U^+)\right)W_1$ in $H$ is equivalent to that of $\left(H'\cap X(J',U^+)\right)W_1$ in $H'$.

Since $H'^\dagger=W_k$, by~\eqref{e:J-dagger-inclusions} and~\eqref{e:W-k-normalizer}
we have $H'=W_kC_k$ for some subgroup $C_k\subset C_G(W_k)$.
Additionally, since $J'^\dagger$ is a conjugate of $W_\ell$ lying in $H'^\dagger=W_k$ and since $W_\ell$ also lies in $W_k$, it follows that there exists $w\in W_k$ for which $W_\ell=wJ'^\dagger w^{-1}$. See for instance \cite[Proposition 2.4]{BFMS} applied when $G=W_k$.

Using~\eqref{e:funfact-3} and~\eqref{e:funfact-1}, we compute that 
\begin{align*}
H'\cap X(J',U^+)&=H'\cap X(J'^\dagger,U^+),\\
&=H'\cap X(w^{-1}W_\ell w,U^+),\\
&=H'\cap w^{-1}X(W_\ell ,U^+),\\
&=w^{-1}\left(H'\cap X(W_\ell ,U^+)\right),\\
&=w^{-1}\left(W_k\cap X(W_\ell,U^+)\right)C_k,
\end{align*}
where the final line follows from the inclusion $C_G(W_k)\subset R$ and from~\eqref{e:funfact-2}.
Recall that $w\in W_k\subset H'$ and therefore the $W_1$-saturation $w^{-1}\left(W_k\cap X(W_\ell,U^+)\right)C_kW_1$ is a nowhere dense subset of $H'$ if and only if $\left(W_k \cap X(W_\ell  ,U^+)\right)C_kW_1$ is a nowhere dense subset of $H'$.
Since the latter set is $C_k$-saturated and $C_k$ commutes with $W_1$, it is therefore equivalent to see that
\begin{equation}
  \label{e:dense-reduced}
\left(W_k \cap X(W_\ell  ,U^+)\right)W_1\quad\text{is a nowhere dense subset of }W_k.
\end{equation}

\noindent 2. We next describe the left-hand side of~\eqref{e:dense-reduced}
in terms of the action of $G=\SU(n,1)$ on $\mathbb C^{n,1}$. 
Let $B\in G$, then by~\eqref{e:X-set-def} we have
$B\in X(W_\ell,U^+)$ if and only if $\Ad_BV^+$ lies in the Lie algebra of $W_\ell$
where $V^+$ is a generator of the Lie algebra of~$U^+$.
Using~\eqref{e:stun-matrix-action} we see that
$$
(\Ad_BV^+)z=-i\langle z,B(e_0+e_1)\rangle_{\mathbb C^{n,1}}(B(e_0+e_1))\quad\text{for all }
z\in\mathbb C^{n,1}.
$$
Recalling~\eqref{eqn:complexW}, we then have
\begin{equation}
  \label{e:workred-1}
X(W_\ell,U^+)=\big\{B\in G\mid B(e_0+e_1)\in \mathbb C^{\ell,1}\oplus \{0\}\big\}.
\end{equation}
Next, take arbitrary $D\in (W_k\cap X(W_\ell,U^+))W_1$.
Then $D\in W_k$ and there exists $C\in W_1$ such that $DC\in X(W_\ell,U^+)$.
Since $\langle e_0+e_1,e_0+e_1\rangle_{\mathbb C^{n,1}}=0$, we see that
$C(e_0+e_1)$ has the form $\lambda (e_0+e^{i\theta}e_1)$ for some $\lambda\in\mathbb C\setminus\{0\}$
and $\theta\in\mathbb S^1=\mathbb R/\mathbb Z$. We now see from~\eqref{e:workred-1} that
\begin{equation}
  \label{e:workred-2}
\begin{gathered}
(W_k\cap X(W_\ell,U^+))W_1\subset
\bigcup_{\theta\in\mathbb S^1} Y_\theta,\\\text{where}\quad
Y_\theta:=\big\{D\in W_k\mid D(e_0+e^{i\theta}e_1)\in \mathbb C^{\ell,1}\oplus\{0\}\big\}.
\end{gathered}
\end{equation}

\noindent 3.
For any $D\in W_k$ and $\theta\in\mathbb S^1$, the vector $v:=D(e_0+e^{i\theta}e_1)$ lies in $\mathbb C^{k,1}\oplus\{0\}$
and satisfies $\langle v,v\rangle_{\mathbb C^{n,1}}=0$.
Thus we may write $v=\lambda' (1,\beta_\theta(D),0)$ for some
$\lambda'\in\mathbb C\setminus\{0\}$ and $\beta_\theta(D)\in \mathbb S^{2k-1}\subset\mathbb C^k$
depending smoothly on $\theta$ and~$D$. For each $\theta$, the map
$\beta_\theta:W_k\to\mathbb S^{2k-1}$ is a submersion.
Moreover, we have
$$
Y_\theta=\big\{D\in W_k\mid \beta_\theta(D)\in \mathbb S^{2k-1}\cap (\mathbb C^{\ell}\oplus\{0\})\big\},\quad
\mathbb S^{2k-1}\cap (\mathbb C^{\ell}\oplus \{0\})\simeq \mathbb S^{2\ell-1}.
$$
It follows that each $Y_\theta$ is a codimension $2(k-\ell)$ embedded submanifold of~$W_k$,
depending smoothly on~$\theta$. Thus the union $\bigcup_{\theta\in\mathbb S^1}Y_\theta$
has codimension at least $2(k-\ell)-1>0$ in $W_k$ and therefore is a nowhere dense subset of~$W_k$,
finishing the proof of~\eqref{e:dense-reduced}.
\end{proof}

\subsubsection{End of the proof of Proposition~\ref{l:AU-orbits-same}}

As a corollary of Lemma~\ref{lemma:nowheredensity} we show that the $W_1$-saturation of the set $\pi_\Gamma(\mathcal{S}_{g_0H})$ featured in~\eqref{e:singular-described} is proper in $x_0H$,
that is, there exists a $W_1$-orbit in $x_0H$ consisting entirely of regular points.
\begin{lemm}\label{cor:nowheredensesaturation}
Fix $x_0=\pi_\Gamma(g_0)\in \Gamma\backslash G$ and write $\overline{x_0W_1}=x_0H$ as in Proposition~\ref{l:W-1-orbits}. Then $\pi_\Gamma(\mathcal S_{g_0H})W_1$ is a proper subset of $x_0H$,
that is
there exists $y_0\in x_0H$ such that $y_0W_1\cap \pi_\Gamma(\mathcal S_{g_0H})=\emptyset$.
\end{lemm}
\Remark Note that in the special case when $H^\dagger$ is conjugate to $W_1$
(that is, $k=1$ in the notation of Proposition~\ref{l:W-1-orbits}), the set of singular points in $x_0H$
is empty by Lemma~\ref{lemma:smallunipotent}. Note that in this case the set $\mathcal S_{g_0H}$
is empty as well, since the set $\mathcal H_{g_0H}$ is empty (as the inequalities
$1\leq \ell<k$ cannot hold for $k=1$).
\begin{proof}
By~\eqref{e:singular-defined} we have
$$
\pi_\Gamma(\mathcal S_{g_0H})W_1=\bigcup_{J\in \mathcal H_{g_0H}} \pi_\Gamma\big((g_0H\cap X(J,U^+))W_1\big).
$$
By Lemma~\ref{lemma:nowheredensity}, recalling Proposition~\ref{l:W-1-orbits} and items~(5)--(6) in the definition
of $\mathcal H_{g_0H}$ in~\S\ref{s:singular-set}, we see that each set $(g_0H\cap X(J,U^+))W_1$ is nowhere dense in~$g_0H$.
Since both $\mathcal H_{g_0H}$ and $\Gamma$ are countable, $\pi_\Gamma(\mathcal S_{g_0H})W_1$ is contained in a countable union
of nowhere dense subsets of $x_0H$, which by the Baire category theorem implies
that cannot be equal to the whole~$x_0H$.
\end{proof}
We are finally ready to finish the proof of Proposition~\ref{l:AU-orbits-same}
and with it of Theorem~\ref{theo:orbitclosure}:
\begin{proof}[Proof of Proposition~\ref{l:AU-orbits-same}]
As before, we consider the case of $AU^+$, with $AU^-$ handled similarly.
We write $\overline{x_0W_1}=x_0H$ as in Proposition~\ref{l:W-1-orbits}.
Take $g_0\in G$ such that $\pi_\Gamma(g_0)=x_0$. By Lemma~\ref{cor:nowheredensesaturation} there exists
\begin{equation}
  \label{e:AUtor-1}
y_0\in x_0H,\quad
y_0W_1\cap \pi_\Gamma(\mathcal S_{g_0H})=\emptyset.
\end{equation}
Note that an Iwasawa decomposition of $W_1$ is given by $W_1=A U^+(K\cap W_1)$.
In particular, $W_1/AU^+=K\cap W_1$ is compact (more precisely, it is a circle) and thus
\begin{equation}\label{eqn:auorbitsaturate}
x_0H=\overline{x_0W_1}=\overline{x_0AU^+}(K\cap W_1).
\end{equation}
Therefore, we can write $y_0=yw$ for some $y\in\overline{x_0AU^+}$
and $w\in K\cap W_1$. By~\eqref{e:AUtor-1} we then see that
$y\notin\pi_\Gamma(\mathcal S_{g_0H})$. Therefore, by~\eqref{e:singular-described}
the closure $\overline{yU^+}$ is equal to the entire $x_0H$.
Thus
$$
\overline{yU^+}\subset \overline{x_0AU^+}\subset \overline{x_0W_1}=xH=\overline{yU^+}
$$
which shows that $\overline{x_0AU^+}=\overline{x_0W_1}$ as needed.
\end{proof}

\subsection{Known examples of complex hyperbolic manifolds and their geodesic submanifolds}
\label{sec:geodesicsubman}

In this subsection, we discuss known examples of closed complex hyperbolic manifolds $M$ in arbitrary dimensions and the behavior of their geodesic submanifolds.
In particular, we give examples in all dimensions of $M$ for which $M$ contains a proper complex geodesic submanifold $\Sigma$ and examples in infinitely many dimensions for which $M$ contains no proper geodesic complex submanifold.
In the latter case Theorem \ref{theo:orbitclosure} shows that every $AU^\pm$-orbit closure equidistributes in $SM$, which we will state formally in Corollary \ref{cor:orbitclosure}.

At present, the only known constructions of finite volume complex hyperbolic manifolds in $\mathbb{CH}^n$ in all dimensions are via arithmetic constructions.
Indeed, since the non-arithmetic constructions of Deligne--Mostow \cite{DeMo} it remains a major open problem whether finite-volume non-arithmetic complex hyperbolic manifolds exist in complex dimension at least $4$, see \cite[Problem 9]{Margulis} or \cite[Conjecture 10.8]{Kapovich}.
For non-arithmetic complex hyperbolic manifolds, all known constructions contain finitely many (and at least one) complex totally geodesic submanifold of complex codimension $1$. 
Indeed, at present all known examples are commensurable with reflection groups and hence contain at least one such submanifold (see the argument in \cite[Theorem 1.3]{Stover} for instance).
That there are then finitely many is the main theorem of \cite{BFMS}.

For complex hyperbolic manifolds, arithmetic manifolds always arise as certain unitary groups of Hermitian elements in central simple algebras.
Such constructions are heavily number theoretic in nature and so rather than describing how to produce such manifolds, we refer the interested reader to \cite[\S9]{BFMS} or \cite[\S2]{Stover2} for a more detailed exposition (see also \cite{Meyer} or \cite{BenSurvey}). 
Importantly, there are two constructions of arithmetic manifolds that have radically different behavior with respect to geodesic submanifolds:  
\begin{itemize}
\item For every $n\ge 2$, there exist closed complex hyperbolic manifolds $M$ of complex dimension $n$ such that for each $1\le k\le n-1$, $M$ contains infinitely many totally geodesic complex submanifolds of complex dimension $k$. $M$ also contains infinitely many totally geodesic real submanifolds in all possible real dimensions.
\item For every $n\ge 2$ such that $n+1$ is prime, there exist closed complex hyperbolic manifolds with no proper totally geodesic complex submanifolds of any dimension.
\end{itemize}
See \cite[Example 9.1]{BFMS} for examples of the former and \cite[Example 9.2]{BFMS} for examples of the latter.
In particular, the latter manifolds allow us to produce the following immediate corollary of Theorem \ref{theo:orbitclosure}.

\begin{corr}\label{cor:orbitclosure}
If $n+1$ is prime, then there exists a closed arithmetic complex hyperbolic manifold $M$ for which $M$ has no proper geodesic complex submanifolds.
In particular, any orbit closure of the $AU^+$-action or $AU^-$-action on $SM$ is all of $SM$.
\end{corr}

\Remark For the reader well versed in arithmetic constructions, the example in \cite[Example 9.1]{BFMS} is actually not closed. However, as is well known to experts, one can easily modify it to get a closed example.
Specifically, one has to require that the requisite Hermitian form is a signature $(n,1)$ form defined over a CM field which is not an imaginary quadratic extension of $\Q$ and such that all of its non-trivial Galois conjugates have signature $(n+1,0)$.
Note also that \cite[Example 9.1]{BFMS} shows how to produce at least one geodesic submanifold but, as explained in the introduction of that paper, for arithmetic manifolds the existence of one geodesic submanifold implies infinitely many.

\section{From decay for long words to Theorem~\ref{t:meas}}
\label{s:usual-proof}

In this section we prove Theorem~\ref{t:meas} modulo the key estimate,
Proposition~\ref{l:key-estimate} below.

\subsection{Semiclassical analysis}
\label{s:review-semi}

We first give a brief review of semiclassical analysis, sending the reader
to~\cite[\S14.2.2]{Zworski-Book}, \cite[\S E.1.5]{DZ-Book}, and~\cite[\S2.1]{hgap} for details.

Let $M$ be a manifold and denote by $T^*M$ its cotangent bundle.
We write elements of $T^*M$ as $(x,\xi)$ where $x\in M$, $\xi\in T_x^*M$.
Denote by $|\xi|$ the norm of $\xi$ with respect
to some Riemannian metric, and denote $\langle\xi\rangle:=\sqrt{1+|\xi|^2}$.
We use the Kohn--Nirenberg symbol class $S_{1,0}^m(T^*M)$ of order $m\in\mathbb R$,
consisting of functions $a\in C^\infty(T^*M)$ such that
for any compact set $K\subset M$ and multiindices $\alpha,\beta$ we have
$|\partial^\alpha_x\partial^\beta_\xi a(x,\xi)|\leq C_{\alpha\beta K}\langle\xi\rangle^{m-|\beta|}$
for some constant $C_{\alpha\beta K}$ and all $x\in K$, $\xi\in T_x^*M$.

We use a semiclassical quantization procedure, mapping each $a\in S_{1,0}^m(T^*M)$
to a family of operators
$$
\Op_h(a)=a(x,hD_x):\CIc(M)\to C^\infty(M),\quad
\mathcal E'(M)\to\mathcal D'(M).
$$
Here $D_x:=-i\partial_x$ denotes the differentiation operator and $0<h<1$ is called
the semiclassical parameter; we are interested in the limit $h\to 0$. The symbol $a$
can depend on~$h$ but for now we require that its $S_{1,0}^m$-seminorms be bounded
uniformly in~$h$. The quantization procedure depends on choices of local charts on~$M$
but a different choice of those produces the same class of operators and
symbols in different quantizations differ by $\mathcal O(h)_{S_{1,0}^{m-1}}$.

We will mostly work with symbols which are compactly supported. Denote by $S^{\comp}_h(T^*M)$
the set of $h$-dependent functions in $\CIc(T^*M)$ which are bounded with all derivatives uniformly in~$h$ and whose support is contained in some $h$-independent compact subset of~$T^*M$. We also introduce here the residual classes $\mathcal O(h^\infty)_{L^2\to L^2}$ consisting of $h$-dependent operators on~$L^2(M)$ whose operator norm is bounded by $\mathcal O(h^N)$ for each $N$,
and $\mathcal O(h^\infty)_{\Psi^{-\infty}}$, consisting of $h$-dependent smoothing
operators whose Schwartz kernels have every $C^\infty(M\times M)$-seminorm bounded
by $\mathcal O(h^N)$ for every~$N$.

We now state some standard properties of semiclassical quantization. To avoid technical details, we focus on the case when $M$ is compact and all the symbols are in $S^{\comp}_h(T^*M)$. First of all, if $a\in S^{\comp}_h(T^*M)$ then the operator $\Op_h(a)$ is bounded
on~$L^2(M)$ uniformly in~$h$.
Next, we have the general composition formula
$$
\Op_h(a)\Op_h(b)=\Op_h(a\# b)+\mathcal O(h^\infty)_{L^2\to L^2}
$$
where the symbol $a\#b\in S^{\comp}_h(T^*M)$ has an asymptotic expansion in the powers of~$h$ featuring the derivatives of~$a,b$. Consequences of this formula include:
\begin{itemize}
\item the Product Rule
\begin{equation}
  \label{e:semi-product}
\Op_h(a)\Op_h(b)=\Op_h(ab)+\mathcal O(h)_{L^2\to L^2},
\end{equation}
\item the Commutator Rule (where $\{\bullet,\bullet\}$ denotes the Poisson bracket on~$T^*M$)
\begin{equation}
  \label{e:semi-commutator}
[\Op_h(a),\Op_h(b)]=-ih\Op_h(\{a,b\})+\mathcal O(h^2)_{L^2\to L^2},
\end{equation}
\item and the Nonintersecting Support Property:
\begin{equation}
  \label{e:semi-nonint-supp}
\supp a\cap\supp b=\emptyset\quad\Longrightarrow\quad
\Op_h(a)\Op_h(b)=\mathcal O(h^\infty)_{L^2\to L^2}.
\end{equation}
\end{itemize}
We also have the Adjoint Rule
\begin{equation}
  \label{e:semi-adjoint}
\Op_h(a)^*=\Op_h(\overline a)+\mathcal O(h)_{L^2\to L^2}.
\end{equation}
Denote by $\Psi^{\comp}_h(M)$ the class of compactly supported operators
of the form $\Op_h(a)+\mathcal O(h^\infty)_{\Psi^{-\infty}}$ with $a\in S^{\comp}_h(T^*M)$ and by $\Psi^m_h(M)$ the class of operators
$\Op_h(a)+\mathcal O(h^\infty)_{\Psi^{-\infty}}$ with $a\in S^m_{1,0}(T^*M)$.
Note that in~\cite{meassupp} we used the
more restrictive polyhomogeneous symbol classes, which have an asymptotic expansion
in powers of~$h$ and~$\xi$, however the difference between the two classes
will not matter in this paper.

For $A\in \Psi^{\comp}_h(M)$, denote its semiclassical wavefront set by
\begin{equation}
\label{e:WF-A}
\WFh(A)\ \subset\ T^*M. 
\end{equation}
One definition of $\WFh(A)$ is as follows: a point $(x,\xi)$ does not lie in $\WFh(A)$
if and only if we can write $A=\Op_h(a)+\mathcal O(h^\infty)_{\Psi^{-\infty}}$
for some symbol $a$ which vanishes on an $h$-independent neighborhood
of $(x,\xi)$. We have $\WFh(A)=\emptyset$
if and only if $A=\mathcal O(h^\infty)_{\Psi^{-\infty}}$ and $\WFh(AB)\subset\WFh(A)\cap \WFh(B)$
for $A,B\in\Psi^{\comp}_h(M)$.

We will occasionally use the more general classes (which are in between the class $S^{\comp}_h$ and the classes introduced in~\S\ref{s:calculus-lagrangian}),
\begin{equation}
  \label{e:S-rho}
S^{\comp}_\rho(T^*M)\quad\text{where }\rho\in [0,\tfrac12),
\end{equation}
consisting of $h$-dependent functions $a(x,\xi;h)\in\CIc(T^*M)$ with support contained
in some $h$-independent compact subset and satisfying the derivative bounds
for all multiindices $\alpha$
$$
\sup|\partial^\alpha a|\leq C_\alpha h^{-\rho|\alpha|}.
$$
Note that $S^{\comp}_h$ is the special case $\rho=0$. Operators with symbols
in $S^{\comp}_\rho$ satisfy analogs of properties~\eqref{e:semi-product}--\eqref{e:semi-adjoint}
with weaker remainders depending on~$\rho$, see e.g.~\cite[Theorem~4.18]{Zworski-Book}.

\subsection{Long time propagation}
\label{s:longtime-prop}

Similarly to~\cite{meassupp} (and \cite{highcat}, which used a different version of this calculus) our argument relies on an anisotropic
semiclassical calculus originating in~\cite{hgap}. We use the version described in~\cite[Appendix~A]{meassupp}.

\subsubsection{Calculus associated to a Lagrangian foliation}
\label{s:calculus-lagrangian}

Let $(M,g)$ be a compact complex hyperbolic quotient. 
Let $L\in \{L_u,L_s\}$ where the weak unstable/stable foliations $L_u,L_s\subset T(T^*M\setminus 0)$ are defined in~\eqref{e:L-def}
and~\S\ref{s:quotients-def}.
As shown in Lemma~\ref{l:E-u-integrable} and Corollary~\ref{l:l-su-lagr}, $L$ is a Lagrangian foliation in the sense of~\cite[\S A.1]{meassupp},
namely each fiber of $L$ is a Lagrangian subspace of $T(T^*M\setminus 0)$ and the foliation $L$ is integrable in the sense of Frobenius.

Fix two parameters
\begin{equation}
  \label{e:rho-properties}
0\leq\rho <1,\quad
0\leq \rho'\leq \tfrac12 \rho,\quad
\rho+\rho'<1.
\end{equation}
As in~\cite[\S A.1]{meassupp}, we say that an $h$-dependent family of smooth functions $a(x,\xi;h)$ on~$T^*M$ lies in the symbol class
$$
S^{\comp}_{L,\rho,\rho'}(T^*M\setminus 0)
$$
if
$a$ is supported in an $h$-independent compact subset of $T^*M\setminus 0$ and satisfies the derivative bounds
\begin{equation}
  \label{e:symb-def}
\sup_{x,\xi}|Y_1\dots Y_m Q_1 \dots Q_k a(x,\xi;h)|\leq Ch^{-\rho k-\rho'm},\quad
0<h\leq 1
\end{equation}
for all vector fields $Y_1,\dots,Y_m,Q_1,\dots,Q_k$ on $T^*M\setminus 0$ such that
$Y_1,\dots,Y_m$ are tangent to~$L$; here the constant $C$ depends on the choice of the vector fields
but not on~$h$. Roughly speaking, the estimates \eqref{e:symb-def} mean that $a$ grows by at most $h^{-\rho'}$ when
differentiated along~$L$ and by at most $h^{-\rho}$ when differentiated in other directions.

We now use the quantization procedure for symbols in the class $S^{\comp}_{L,\rho,\rho'}$ constructed in~\cite[\S A.4]{meassupp},
which maps each symbol $a$ to an $h$-dependent family of smoothing operators on $M$:
\begin{equation}
  \label{e:Op-L}
a\in S^{\comp}_{L,\rho,\rho'}(T^*M\setminus 0)\ \mapsto\ \Op^L_h(a):\mathcal D'(M)\to C^\infty(M).  
\end{equation}
Such operators satisfy the properties of semiclassical quantization described in~\cite[\S A.4]{meassupp}, in particular
their operator norms on $L^2$ are bounded uniformly in~$h$ and we have the following versions
of the Product Rule, Nonintersecting Support Property, and Adjoint Rule from~\S\ref{s:review-semi}:
for all $a,b\in S^{\comp}_{L,\rho,\rho'}(T^*M\setminus 0)$
\begin{align}
  \label{e:extra-product}
\Op_h^L(a)\Op_h^L(b)&=\Op_h^L(ab)+\mathcal O(h^{1-\rho-\rho'})_{L^2\to L^2},\\
  \label{e:extra-nonint-supp}
\supp a\cap\supp b=\emptyset\quad&\Longrightarrow\quad \Op_h^L(a)\Op_h^L(b)=\mathcal O(h^\infty)_{L^2\to L^2},\\
  \label{e:extra-adjoint}
\Op_h^L(a)^*&=\Op_h^L(\overline a)+\mathcal O(h^{1-\rho-\rho'})_{L^2\to L^2}.
\end{align}

Note that for $\rho=\rho'=0$ the symbol class $S^{\comp}_{L,0,0}(T^*M\setminus 0)$
is independent of~$L$ and is the same as the class $S^{\comp}(T^*M\setminus 0)$
of symbols which are in~$C^\infty(T^*M\setminus 0)$ uniformly in~$h$. If
$a\in S^{\comp}(T^*M\setminus 0)$, then the special quantization $\Op_h^L(a)$
is equivalent to the usual quantization $\Op_h(a)$ used in~\S\ref{s:review-semi} above, in particular
$$
\Op_h^L(a)=\Op_h(a)+\mathcal O(h)_{L^2\to L^2}.
$$
More generally, if $0\leq \rho'<\frac13$ then the symbol class $S^{\comp}_{\rho'}(T^*M\setminus 0)$
defined in~\eqref{e:S-rho} (where we require the support to be in an $h$-independent
compact subset of $T^*M\setminus 0$) is contained in the class $S^{\comp}_{L,2\rho',\rho'}(T^*M\setminus 0)$ and we have for all $a\in S^{\comp}_{\rho'}(T^*M\setminus 0)$
\begin{equation}
  \label{e:quant-convert-S-rho}
\Op_h^L(a)=\Op_h(a)+\mathcal O(h^{1-2\rho'})_{L^2\to L^2}.
\end{equation}

\subsubsection{Propagation of classical observables}

Symbols in the classes $S^{\comp}_{L,\rho,\rho'}$ appear in our argument as the results of propagating $h$-independent symbols along the geodesic flow for times logarithmic in~$h$. Here the geodesic flow
\begin{equation}
  \label{e:geod-flow}
\varphi^t=e^{tX}:T^*M\setminus 0\to T^*M\setminus 0
\end{equation}
is the projection of the flow~\eqref{e:Ham-flow-CH-2} under the map $T^*\mathbb{CH}^n\to T^* M$, and it is the Hamiltonian flow of the symbol
\begin{equation}
  \label{e:p-def}
p\in C^\infty(T^*M\setminus 0),\quad
p(x,\xi)=|\xi|_{g(x)}.
\end{equation}
\begin{lemm}
\label{l:propag-classical}
Fix $0\leq\rho<1$ and an $h$-independent function $a\in \CIc(T^*M\setminus 0)$.
Assume that $0\leq t\leq \frac 12\rho\log(1/h)$. Then we have
\begin{align}
\label{e:propag-1}
a\circ\varphi^t&\in S^{\comp}_{L_s,\rho,0}(T^*M\setminus 0),
\\
\label{e:propag-2}
a\circ\varphi^{-t}&\in S^{\comp}_{L_u,\rho,0}(T^*M\setminus 0)
\end{align}
with $S^{\comp}_{\bullet,\rho,0}$-seminorms bounded uniformly in $t,h$.
\end{lemm}
\begin{proof}
We show~\eqref{e:propag-1}, with~\eqref{e:propag-2} proved similarly.
We argue similarly to the proof of~\cite[Lemma~4.2]{hgap}.
As in that proof, we see that it suffices to show the bound
\begin{equation}
  \label{e:propagb-1}
\sup_{S^*M} |Y_1\dots Y_m Q_1\dots Q_k (a\circ \varphi^t)|\leq Ch^{-\rho k}
\end{equation}
for all vector fields $Y_1,\dots,Y_m,Q_1,\dots,Q_k$ on~$S^*M$ such that
$Y_1,\dots,Y_m$ are tangent to~$E_s$ and $Q_1,\dots,Q_k$ are tangent to~$E_u$.
Here the constant $C$ depends on $a$ and the choice of vector fields but not on $t$ or~$h$.

Using the projection $\pi_R:\Gamma\backslash G\to SM\simeq S^*M$ from~\eqref{eqn:commtriangle},
we lift the function $a|_{S^*M}$ to $\Gamma\backslash G$. Recalling the construction of the
spaces $E_u,E_s$ in~\S\ref{s:complex-stun}, we see that the bound~\eqref{e:propagb-1} reduces to
\begin{equation}
  \label{e:propagb-2}
\begin{gathered}
\sup_{\Gamma\backslash G} |\widetilde Y_1\dots\widetilde Y_m\widetilde Q_1\dots\widetilde Q_k((\pi_R^*a)\circ e^{tX})|
\leq Ch^{-\rho k}\qquad\text{for all}\\
\widetilde Y_1,\dots,\widetilde Y_m\in \{V^+,W_2^+,\dots,W_n^+,Z_2^+,\dots,Z_n^+\},\\
\widetilde Q_1,\dots,\widetilde Q_k\in \{V^-,W_2^-,\dots,W_n^-,Z_2^-,\dots,Z_n^-\}.
\end{gathered}
\end{equation}
We write $m=m_1+m_2$ where $m_1$ is the number of vector fields $\widetilde Y_1,\dots,\widetilde Y_m$ equal to~$V^+$
and similarly write $k=k_1+k_2$. By the commutation relations~\eqref{e:comm-rel} we see that
the left-hand side of~\eqref{e:propagb-2} is equal to
$$
e^{(-2m_1-m_2+2k_1+k_2)t}\sup_{\Gamma\backslash G} |\widetilde Y_1\dots\widetilde Y_m\widetilde Q_1\dots\widetilde Q_k(\pi_R^*a)|.
$$
Now, since $0\leq t\leq \tfrac12 \rho\log(1/h)$ and $k_1+k_2=k$, we see that
$e^{(-2m_1-m_2+2k_1+k_2)t}\leq e^{2kt}\leq h^{-\rho k}$. This gives the estimate~\eqref{e:propagb-2} and finishes the proof.
\end{proof}

\subsubsection{Propagation of quantum observables}

We next discuss a version of long time Egorov's Theorem corresponding to Lemma~\ref{l:propag-classical}.
Following~\cite[\S2.2]{meassupp}, we fix a cutoff function
$$
\psi_P\in\CIc((0,\infty);\mathbb R),\quad \psi_P(\lambda)=\sqrt{\lambda}\quad\text{for }
\tfrac1{16}\leq \lambda\leq 16
$$
and define the bounded self-adjoint operator on $L^2(M)$
$$
P:=\psi_P(-h^2\Delta_g)
$$
and the corresponding unitary group
\begin{equation}
  \label{e:U-t-def}
U(t):=\exp\Big(-{itP\over h}\Big):L^2(M)\to L^2(M).
\end{equation}
For a bounded operator $A:L^2(M)\to L^2(M)$, define its conjugation by the unitary group
\begin{equation}
  \label{e:A-t-def}
A(t):=U(-t)AU(t).
\end{equation}
Then our version of Egorov's Theorem is given by
\begin{lemm}
  \label{l:egorov-long}
Fix $0\leq\rho<1$ and an $h$-independent function $a\in \CIc(T^*M)$ such that $\supp a\subset \{\tfrac14<|\xi|_g<4\}$.
Put $A:=\Op_h(a)$.
Then we have for all $t\in [0,\tfrac12\rho\log(1/h)]$
\begin{align}
  \label{e:egorov-long-1}
A(t)&=\Op_h^{L_s}(a\circ\varphi^t)+\mathcal O\big(h^{1-\rho}\log(1/h)\big)_{L^2\to L^2},\\
  \label{e:egorov-long-2}
A(-t)&=\Op_h^{L_u}(a\circ\varphi^{-t})+\mathcal O\big(h^{1-\rho}\log(1/h)\big)_{L^2\to L^2}
\end{align}
where the constants in $\mathcal O(\bullet)$ are independent of $t$ and $h$.
\end{lemm}
The proof of Lemma~\ref{l:egorov-long} is identical to that of~\cite[Lemma~A.8]{meassupp}
using Lemma~\ref{l:propag-classical} for bounds on the symbols $a\circ\varphi^{\pm t}$.

\subsection{Reduction to a control estimate}
\label{s:reduction-control}

We next reduce Theorem~\ref{t:meas} to a more general control estimate.
As before, we identify
the cotangent bundle $T^*M$ with the tangent bundle $TM$ via the metric,
which in particular identifies the cosphere bundle $S^*M$ with the sphere bundle $SM$.

Recall the fast unstable/stable vector fields $V^\pm$ on $S^*M$ introduced in~\S\ref{s:complex-stun}
and the notion of $V^\pm$-density from~\S\ref{s:orbits}. Our control estimate is given by
\begin{theo}
  \label{t:esti}
Let $(M,g)$ be a compact complex hyperbolic quotient.
Assume that $a\in S^0_{1,0}(T^*M)$ is $h$-independent and the set $\{a\neq 0\}\cap S^*M$
is both $V^+$\!-dense and $V^-$\!-dense in $S^*M$.
Then there exist constants $C,h_0>0$ depending only on $M,a$ such that
for all $u\in H^2(M)$ and all $h\in (0,h_0]$
\begin{equation}
  \label{e:esti}
\|u\|_{L^2(M)}\leq C\|\Op_h(a)u\|_{L^2(M)}+{C\log(1/h)\over h}\|(-h^2\Delta_g-I)u\|_{L^2(M)}.
\end{equation}
\end{theo}
Before giving the proof of Theorem~\ref{t:esti}, we show that together with the results
on orbit closures in~\S\ref{s:orbits-total} it implies Theorem~\ref{t:meas}:
\begin{proof}[Proof of Theorem~\ref{t:meas}]
We argue by contradiction. Assume that $\mu$ is a semiclassical measure and
$\supp \mu$ does not contain $S^*\Sigma$ for any compact immersed totally geodesic complex submanifold
$\Sigma\subset S^*M$. The complement $\mathscr U:=S^*M\setminus \supp\mu$ is an open subset of~$S^*M$ invariant under the geodesic flow $\varphi^t$, since $\mu$ is $\varphi^t$-invariant. By Corollary~\ref{c:orbitclosure} the set $\mathscr U$ is both $V^+$\!-dense and $V^-$\!-dense. By Lemma~\ref{l:dense-basic}, there exists a compact set $\mathscr K\subset\mathscr U$ which is both $V^+$\!-dense and $V^-$\!-dense. Fix a cutoff function
$$
a\in\CIc(T^*M),\quad
\supp a\cap S^*M\subset \mathscr U,\quad
\mathscr K\subset \{a\neq 0\}.
$$
Since $\mu$ is a semiclassical measure, there exists a sequence
of eigenfunctions $u_j$ satisfying~\eqref{e:laplace-eig} and converging to $\mu$
in the sense of~\eqref{e:semi-measure-conv}. (Here as before,
we have $h_j:=\lambda_j^{-1}\to 0$.) By the Product Rule~\eqref{e:semi-product}
and the Adjoint Rule~\eqref{e:semi-adjoint} we have
\begin{equation}
  \label{e:meas-int-1}
\begin{aligned}
\|\Op_{h_j}(a)u_j\|_{L^2}^2&=\langle\Op_{h_j}(a)^*\Op_{h_j}(a)u_j,u_j\rangle_{L^2}\\
&=\langle\Op_{h_j}(|a|^2)u_j,u_j\rangle_{L^2}+\mathcal O(h_j)\to\int_{T^*M} |a|^2\,d\mu=0.
\end{aligned}
\end{equation}
Here the last equality follows from the fact that $\mu$ is supported on $S^*M\setminus\mathscr U$
and thus $\supp a\cap\supp \mu=\emptyset$.

Applying Theorem~\ref{t:esti} with $u:=u_j$, $h:=h_j$ and using that $(-h_j^2\Delta_g-I)u_j=0$
by~\eqref{e:laplace-eig}, we get for $j$ large enough
$$
1=\|u_j\|_{L^2}\leq C\|\Op_{h_j}(a)u_j\|_{L^2}.
$$
This gives a contradiction with~\eqref{e:meas-int-1} and finishes the proof.
\end{proof}

\subsection{Partitions and words}
\label{s:partitions}

In~\S\S\ref{s:partitions}--\ref{s:esti-proof} we give the proof of Theorem~\ref{t:esti}, modulo the key estimate (Proposition~\ref{l:key-estimate}).
We largely follow~\cite[\S\S3--4]{meassupp}. For an expository presentation of this part of the argument, see~\cite[\S2]{Dyatlov-JEDP}.

\subsubsection{Microlocal partition of unity}

Let $a\in S^0_{1,0}(T^*M)$ be the symbol given in Theorem~\ref{t:esti}. 
Similarly to~\cite[\S3.1]{meassupp}, we construct a microlocal partition of unity:
\begin{lemm}
\label{l:our-partition}
There exists a decomposition
\begin{equation}
  \label{e:our-partition}
I=A_0+A_1+A_2,\quad
A_0\in\Psi_h^0(M),\quad
A_1,A_2\in\Psi_h^{\comp}(M)
\end{equation}
such that:
\begin{enumerate}
\item $A_0$ is microlocalized away from the cosphere bundle $S^*M$ and is a function
of the Laplacian,
more precisely $A_0=\psi_0(-h^2\Delta_g)$ for some function $\psi_0\in C^\infty(\mathbb R;[0,1])$
satisfying
\begin{equation}
  \label{e:psi-0}
\supp \psi_0\cap [\tfrac 14,4]=\emptyset,\quad
\supp(1-\psi_0)\subset (\tfrac 1{16},16);
\end{equation}
\item there exist $h$-independent functions $a_1,a_2\in \CIc(T^*M;[0,1])$ (called the
principal symbols of $A_1,A_2$) such that for $\ell=1,2$
\begin{align}
\label{e:a-ell-principal}
A_\ell&=\Op_h(a_\ell)+\mathcal O(h)_{\Psi^{\comp}_h},\\
\label{e:a-ell-supp}
\supp a_\ell\ &\subset\ \mathcal V_\ell\cap\{\tfrac 14<|\xi|_g<4\}
\end{align}
for some closed conic subsets $\mathcal V_\ell\subset T^*M\setminus 0$
such that $S^*M\setminus \mathcal V_\ell$ are both $V^+$\!-dense and $V^-$\!-dense;
\item $a_1$ is controlled by the symbol $a$ on the cosphere bundle, more precisely
\begin{equation}
  \label{e:a-1-supp}
\supp a_1\cap S^*M\ \subset\ \{a\neq 0\}.
\end{equation}
\end{enumerate}
\end{lemm}
\begin{proof}
Define the set $\mathcal U:=\{a\neq 0\}\cap S^*M$. By the assumption in Theorem~\ref{t:esti},
$\mathcal U$ is both $V^+$\!-dense and $V^-$\!-dense.
Applying Lemma~\ref{l:partition-construction} with $\mathcal M=S^*M$
and $V=V^\pm$
(see the remark after this lemma regarding the condition of being simultaneously
$V^+$-dense and $V^-$-dense), we construct a partition of unity
$$
\chi_1,\chi_2\in C^\infty(S^*M;[0,1]),\quad
\chi_1+\chi_2=1,\quad
\supp \chi_1\subset \{a\neq 0\}
$$
such that for $\ell=1,2$ the complements $S^*M\setminus \supp\chi_\ell$ are both $V^+$\!-dense
and $V^-$\!-dense.

Next, fix a function $\psi_0$ satisfying~\eqref{e:psi-0} and define $A_0:=\psi_0(-h^2\Delta_g)$.
By the functional calculus of pseudodifferential operators (see~\cite[Theorem~14.9]{Zworski-Book}
or~\cite[\S8]{Dimassi-Sjostrand}), we have
$$
I-A_0=\Op_h(b^\flat)+R,\quad
R=\mathcal O(h^\infty)_{\Psi^{\comp}_h}
$$
where $b^\flat\in S^{\comp}_h(T^*M)$ is an $h$-dependent symbol satisfying
$$
b^\flat=1-\psi_0(|\xi|_g^2)+\mathcal O(h)_{S^{\comp}_h},\quad
\supp b^\flat\subset \{\tfrac 14<|\xi|_g<4\}.
$$
Now, we extend $\chi_\ell$ to homogeneous functions of order~0 on $T^*M\setminus 0$ and define
$$
a_\ell^\flat:=\chi_\ell b^\flat,\quad
A_1:=\Op_h(a_1^\flat)+R,\quad
A_2:=\Op_h(a_2^\flat).
$$
Then~\eqref{e:our-partition} and~\eqref{e:a-ell-principal} hold with the functions
$a_\ell:=\chi_\ell (1-\psi_0(|\xi|^2))$ and the sets $\mathcal V_\ell:=\supp \chi_\ell$, which satisfy~\eqref{e:a-ell-supp}
and~\eqref{e:a-1-supp}.
\end{proof}

\subsubsection{Refined microlocal partition}

Still following~\cite[\S3.1]{meassupp}, we now dynamically refine the microlocal partition~\eqref{e:our-partition}. We only consider the partition elements
$A_1,A_2$, with $A_0$ handled by~\eqref{e:A0-rid} below.
This may look similar to the
refined microlocal partition introduced by Anantharaman~\cite{Anantharaman-Entropy}.
However, in~\cite{Anantharaman-Entropy} the supports of the symbols $a_1,a_2,\dots$ were
small enough so that each element of the refined partition was microlocalized
on a single unstable/stable rectangle; in the present paper the elements of the refined partition
are instead microlocalized on fractal sets.

For each $n\in\mathbb N_0$, consider the set of words of length~$n$
$$
\mathcal W(n)=\{1,2\}^n=\{\mathbf w=w_0\dots w_{n-1}\mid w_0,\dots,w_{n-1}\in \{1,2\}\}.
$$
For each word $\mathbf w=w_0\dots w_{n-1}$, using the notation~\eqref{e:A-t-def} we define the operator
\begin{equation}
  \label{e:A-w-def}
A_{\mathbf w}:=A_{w_{n-1}}(n-1)A_{w_{n-2}}(n-2)\cdots A_{w_1}(1)A_{w_0}(0).
\end{equation}
We will work with words of length $n\sim\log(1/h)$, for which the operators
$A_{\mathbf w}$ are bounded uniformly on~$L^2$:
\begin{lemm}
\label{l:A-w-bdd}
Assume that $n\leq C_0\log(1/h)$. Then there exists a constant $C$ depending on~$C_0$
but not on~$n,h$ such that for all $\mathbf w\in\mathcal W(n)$
we have $\|A_{\mathbf w}\|_{L^2\to L^2}\leq C$.
\end{lemm}
\begin{proof}
Using~\eqref{e:a-ell-principal}, the fact that $|a_\ell|\leq 1$, and
a standard bound on the norm of a pseudodifferential operator (see e.g.~\cite[Lemma~A.5]{meassupp}
with $\rho=\rho'=0$),
we see that there exists an $h$-independent constant $C_1$ such that
$$
\|A_\ell\|_{L^2\to L^2}\leq 1+C_1h\quad\text{for }\ell=1,2.
$$
It remains to recall the definition~\eqref{e:A-w-def} and use that the operator
$U(t)$ is unitary on~$L^2$ to get $\|A_{\mathbf w}\|_{L^2\to L^2}\leq (1+C_1h)^n\leq C$.
(We see from here that the argument in fact works until $n\leq C_0h^{-1}$ but
in this paper we only need logarithmically large times.)
\end{proof}
We also define linear combinations of operators $A_{\mathbf w}$.
If $c:\mathcal W(n)\to\mathbb C$ is a function, then we put
\begin{equation}
\label{e:A-c-def}
A_c:=\sum_{\mathbf w\in\mathcal W(n)}c(\mathbf w)A_{\mathbf w}.
\end{equation}
A special case is when $c$ is an indicator function:
for  a set $\mathcal E\subset \mathcal W(n)$ we define
\begin{equation}
\label{e:A-E-def}
A_{\mathcal E}:=\sum_{\mathbf w\in\mathcal E}A_{\mathbf w}.
\end{equation}

\subsubsection{Quantum/classical correspondence for the refined partition}

Using the functions $a_1,a_2$ featured in~\eqref{e:a-ell-principal}, we define
the symbols formally corresponding to $A_{\mathbf w},A_c,A_{\mathcal E}$:
\begin{equation}
  \label{e:a-w-def}
a_{\mathbf w}:=\prod_{j=0}^{n-1} (a_{w_j}\circ\varphi^j),\quad
a_c:=\sum_{\mathbf w\in \mathcal W(n)}c(\mathbf w)a_{\mathbf w},\quad
a_{\mathcal E}:=\sum_{\mathbf w\in\mathcal E}a_{\mathbf w}.
\end{equation}
We now establish a `quantum/classical correspondence' between the operators $A_{\mathbf w},A_c$
and the corresponding symbols. For fixed $n$ (bounded independently of~$h$),
combining the basic, bounded time, Egorov's Theorem (see e.g.~\cite[(2.15)]{meassupp})
with the Product Rule~\eqref{e:semi-product} we get
\begin{equation}
  \label{e:qc-basic}
A_{\mathbf w}=\Op_h(a_{\mathbf w})+\mathcal O(h)_{L^2\to L^2},\quad
A_c=\Op_h(a_c)+\mathcal O(h)_{L^2\to L^2}.
\end{equation}
However, in the argument we need to take $n$ which grows logarithmically with $h$.

We first give quantum/classical correspondence for the individual operators $A_{\mathbf w}$
when the length $n$ of the word $\mathbf w$ is less than $\tfrac 12\log(1/h)$,
which corresponds to the \emph{Ehrenfest time}\footnote{In general the Ehrenfest time is defined as the time at which the classical/quantum correspondence breaks down and it may depend on the quantization used. For the
more common quantizations using the classes $S^{\comp}_\rho$
defined in~\eqref{e:S-rho} with $\rho<\frac12$, the Ehrenfest time in the present setting would
be $\tfrac14\log(1/h)$. However, our choice of the quantization
$\Op_h^{L_s}$ allows us to prove classical/quantum correspondence
until time $\tfrac12\log(1/h)$.}:
the time at which the differential $d\varphi^t$ of the geodesic flow has norm~$h^{-1}$.
The $\varepsilon$ losses below are caused by the fact that $a_{\mathbf w}$
is the product of $\sim \log(1/h)$ many symbols, so each its derivative
is the sum of $\sim \log(1/h)$ many terms.
\begin{lemm}
\label{l:qc-Aw}
Fix $0\leq \rho<1$. Then for any $n\leq \frac12\rho\log(1/h)$, $\mathbf w\in\mathcal W(n)$, and 
small~$\varepsilon>0$, we have
\begin{align}
  \label{e:qc-aw}
a_{\mathbf w}&\in S^{\comp}_{L_s,\rho+\varepsilon,\varepsilon}(T^*M\setminus 0),\\
  \label{e:qc-Aw}
A_{\mathbf w}&=\Op_h^{L_s}(a_{\mathbf w})+\mathcal O(h^{1-\rho-\varepsilon})_{L^2\to L^2}.
\end{align}
The implied constants do not depend on $n,\mathbf w,h$.
\end{lemm}
\begin{proof}
This is deduced from Lemmas~\ref{l:propag-classical} and~\ref{l:egorov-long}
in the same way as~\cite[Lemma~3.2]{meassupp}.
\end{proof}
Next, we make the stronger assumption that $n$ is less than $\tfrac 16\log(1/h)$
and give quantum/classical correspondence for the linear combinations $A_c$
(and thus $A_{\mathcal E}$ as a special case).
Note that with more effort, one might be able to prove a version
of Lemma~\ref{l:qc-Ac} for all $\rho<1$, however this is not needed in our application;
see in particular~\S\ref{s:log-times}.
\begin{lemm}
  \label{l:qc-Ac}
Fix $0.01\leq \rho<\tfrac 13$. Then for any $n\leq\frac12\rho\log(1/h)$ and $c:\mathcal W(n)\to\mathbb C$
such that $\max|c|\leq 1$, we have
\begin{align}
  \label{e:qc-ac}
a_c&\in S^{\comp}_{L_s,2\rho,\rho}(T^*M\setminus 0),\\
  \label{e:qc-Ac}
A_c&=\Op_h^{L_s}(a_c)+\mathcal O(h^{1-2\rho})_{L^2\to L^2}.
\end{align}
The implied constants do not depend on $n,\mathbf w,h$.
\end{lemm}
\begin{proof}
We follow the proof of~\cite[Lemma~4.4]{meassupp} (which considered
the special case $\rho=\tfrac 14$). To show~\eqref{e:qc-ac} we first note
that $\sup |a_c|\leq 1$. It remains to estimate the derivatives of~$a_c$:
more precisely, we need to show that for $m+k>0$ and all vector fields
$Y_1,\dots,Y_m,Q_1\dots,Q_k$ on~$T^*M\setminus 0$ such that $Y_1,\dots,Y_m$
are tangent to~$L_s$, we have
\begin{equation}
  \label{e:qc-Ac-int}
\sup |Y_1\dots Y_mQ_1\dots Q_k a_c|\leq Ch^{-2\rho k-\rho m}.
\end{equation}
Using the triangle inequality, we see that the left-hand side of~\eqref{e:qc-Ac-int}
is bounded by
$$
\sum_{\mathbf w\in\mathcal W(n)} \sup|Y_1\dots Y_m Q_1\dots Q_k a_{\mathbf w}|.
$$
By~\eqref{e:qc-aw} with $\varepsilon:=0.001/(m+k)$, each summand is bounded by $Ch^{-\rho k-0.001}$
where $C$ is independent of~$\mathbf w$.
The number of summands is equal to
$2^n\leq h^{-\rho+0.001}$. Together these two statements give~\eqref{e:qc-Ac-int},
finishing the proof of~\eqref{e:qc-ac}.
A similar argument using the triangle inequality and~\eqref{e:qc-Aw} gives~\eqref{e:qc-Ac}.
\end{proof}
As an application of Lemma~\ref{l:qc-Ac}, we give the following inequality
used in the proof of~\eqref{e:A-Z-estimate} below:
\begin{lemm}
  \label{l:special-Garding}
Fix $0.01\leq \rho<\tfrac 13$. Then for any $n\leq \frac12\rho \log(1/h)$ and
functions $c,d:\mathcal W(n)\to\mathbb C$ such that
$$
|c(\mathbf w)|\leq d(\mathbf w)\leq 1\quad\text{for all }\mathbf w\in\mathcal W(n)
$$
and all $u\in L^2(M)$ we have
\begin{equation}
  \label{e:special-Garding}
\|A_cu\|_{L^2}\leq \|A_du\|_{L^2}+Ch^{1-3\rho\over 2}\|u\|_{L^2}
\end{equation}
where the constant $C$ is independent of $c,d,n$.
\end{lemm}
\begin{proof}
We follow the proof of~\cite[Lemma~4.5]{meassupp}. By~\eqref{e:qc-Ac} we may replace $A_c$, $A_d$
by $\Op_h^{L_s}(a_c)$, $\Op_h^{L_s}(a_d)$. Define the operator
$$
B:=\Op_h^{L_s}(a_d)^*\Op_h^{L_s}(a_d)-\Op_h^{L_s}(a_c)^*\Op_h^{L_s}(a_c).
$$
By~\eqref{e:qc-ac} and the Product and Adjoint Rules~\eqref{e:extra-product}, \eqref{e:extra-adjoint} for the $S^{\comp}_{L_s,2\rho,\rho}$-calculus we have
$$
B=\Op_h^{L_s}(a_d^2-|a_c|^2)+\mathcal O(h^{1-3\rho})_{L^2\to L^2}.
$$
Since $|a_c|^2\leq a_d^2$,
by the sharp G\r arding inequality for the $S^{\comp}_{L_s,2\rho,\rho}$-calculus~\cite[Lemma~A.4]{meassupp} we then have for all $u\in L^2(M)$
$$
\langle Bu,u\rangle_{L^2}\geq -Ch^{1-3\rho}\|u\|_{L^2}^2
$$
which gives $\|\Op_h^{L_s}(a_c)u\|_{L^2}^2\leq \|\Op_h^{L_s}(a_d)u\|_{L^2}^2+Ch^{1-3\rho}\|u\|_{L^2}^2$. It remains to take the square roots to arrive to~\eqref{e:special-Garding}.
\end{proof}

\subsection{Controlled and uncontrolled words and the proof of Theorem~\ref{t:esti}}
\label{s:esti-proof}

In this section we finish the proof of Theorem~\ref{t:esti}, modulo the key estimate (Proposition~\ref{l:key-estimate}). This part of the argument is similar
to~\cite{meassupp} and we refer to that paper for most of the details.

\subsubsection{Logarithmic propagation times}
\label{s:log-times}

We first fix the propagation times used in the argument. Our choice differs
from~\cite[\S3.2]{meassupp}, instead it is similar to the times fixed in~\cite[\S3.1.1]{highcat}
(in the special case $\log|\lambda_+|=2$, $\log\gamma=1$, $\rho=\tfrac23(1-\varepsilon_0)$, $\rho'=\tfrac13(1-\tfrac12\varepsilon_0)$, $J=2$,  and with $\mathbf N:=h^{-1}$),
taking advantage of the presence of fast and slow unstable/stable directions.

Let $\varepsilon_0>0$ be small. An examination of the arguments below shows that we can take any $\varepsilon_0\in (0,\tfrac14)$ (which is most crucially used to ensure that $\rho>\frac12$ in Proposition~\ref{l:fup-specific} below) so for example we could fix $\varepsilon_0:=\frac18$.
However we choose to not fix $\varepsilon_0$ to make the exponents below easier to understand.
Define
\begin{equation}
  \label{e:propagation-times}
N_0:=\Big\lceil {1-\varepsilon_0\over 6}\log(1/h)\Big\rceil,\quad
N_1:=2N_0\approx {1-\varepsilon_0\over 3}\log(1/h).
\end{equation}
We call $N_0$ the \emph{short propagation time} and $2N_1$ the \emph{long propagation time}. What matters for the argument
is the value of $N_1$ (as explained at the beginning of~\S\ref{s:key-proof} below) and the fact that $N_0\approx N_1/J$ for some sufficiently large integer $J$;
in our version of the argument we can already take $J=2$, and our choice of $N_0$ is most prominently used in the fact that Lemmas~\ref{l:qc-Ac} and~\ref{l:special-Garding} above apply with $n=N_0$.

\subsubsection{Statement of the key estimate}

We now formulate the key estimate needed in the proof of Theorem~\ref{t:esti}. Its statement
is similar to~\cite[Proposition~3.5]{meassupp} but its proof, given in~\S\ref{s:key-proof} below,
is a key difference between the present paper and~\cite{meassupp} (though both rely on the same fractal uncertainty principle of~\cite{fullgap}).
\begin{prop}
\label{l:key-estimate}
Assume that $0<\varepsilon_0<\frac14$.
Let $N_1$ be fixed in~\eqref{e:propagation-times}. Then there exist constants $\beta>0,C$
such that for all $\mathbf w\in\mathcal W(2N_1)$ we have
\begin{equation}
  \label{e:key-estimate}
\|A_{\mathbf w}\|_{L^2\to L^2}\leq Ch^\beta.
\end{equation}
\end{prop}
\Remark The value of $\beta$ depends only on the manifold $(M,g)$, the sets $\mathcal V_\ell$
in~\eqref{e:a-ell-supp}, and $\varepsilon_0$ (as mentioned
above we can put $\varepsilon_0:=\frac18$ in the argument). 

\subsubsection{Controlled and uncontrolled words}

Similarly to~\cite[Lemma~3.1 and~(3.8)]{meassupp},
using the properties of the operator $A_0$ in Lemma~\ref{l:our-partition}
we have for any $u\in H^2(M)$, uniformly in $n\in\mathbb N_0$
\begin{equation}
  \label{e:A0-rid}
\|u-A_{\mathcal W(n)}u\|_{L^2}\leq C\|(-h^2\Delta_g-I)u\|_{L^2}.  
\end{equation}
Here $A_{\mathcal W(n)}=\sum_{\mathbf w\in\mathcal W(n)}A_{\mathbf w}$
is defined in~\eqref{e:A-E-def}.
In fact, since $A_1+A_2=I-A_0$ and $A_0$ commutes with $U(t)$ by Lemma~\ref{l:our-partition},
we have
\begin{equation}
  \label{e:A-W-n}
A_{\mathcal W(n)}=(A_1+A_2)^n.
\end{equation}
We use~\eqref{e:A0-rid} in particular with $n=2N_1$ where $N_1$ is fixed in~\eqref{e:propagation-times}.

We now follow~\cite[\S3.2]{meassupp} and write $A_{\mathcal W(2N_1)}$ as the sum of two
operators, $A_{\mathcal X}$ and~$A_{\mathcal Y}$, where
\begin{equation}
  \label{e:X-Y}
\mathcal W(2N_1)=\mathcal X\sqcup \mathcal Y.
\end{equation}
We call $\mathcal X$ the set of \emph{uncontrolled words} and~$\mathcal Y$ the set of \emph{controlled} words. Roughly speaking, $\mathcal X$ consists of words $\mathbf w\in\mathcal W(2N_1)$ which have a small proportion of the digits equal to~1,
and $\mathcal Y$ consists of words where a positive proportion of the digits is equal to~1.
Later in the argument we estimate $\|A_{\mathcal X}u\|_{L^2}$ using Proposition~\ref{l:key-estimate}
and estimate $\|\mathcal A_{\mathcal Y}u\|_{L^2}$ using the property~\eqref{e:a-1-supp}.

To define the sets $\mathcal X,\mathcal Y$,
we recall that $N_1=2N_0$ and write words $\mathbf w\in\mathcal W(2N_1) = \mathcal W (4N_0)$ as concatenations $\mathbf w^{(1)}\mathbf w^{(2)}\mathbf w^{(3)}\mathbf w^{(4)}$ where $\mathbf w^{(\ell)}\in \mathcal W(N_0)$.
Define the \emph{density function}
\begin{equation}
  \label{e:density-F-def}
F:\mathcal W(N_0)\to [0,1],\quad
F(w_0\dots w_{N_0-1})={1\over N_0}\#\{j\mid w_j=1\}.
\end{equation}
Let $\alpha>0$ be a small enough constant depending
on the value of~$\beta$ in Proposition~\ref{l:key-estimate}, fixed in Lemma~\ref{l:X-count} below,
and define the set of \emph{controlled short words}
$$
\mathcal Z:=\{\mathbf w\in\mathcal W(N_0)\mid F(\mathbf w)\geq \alpha\}.
$$
We now define the sets $\mathcal X,\mathcal Y$ in~\eqref{e:X-Y} as follows:
\begin{equation}
  \label{e:X-Y-def}
\begin{aligned}
\mathcal X&\,:=\{\mathbf w^{(1)}\dots\mathbf w^{(4)}\in \mathcal W(2N_1)\mid \mathbf w^{(\ell)}\notin\mathcal Z\quad\text{for all }\ell\},\\
\mathcal Y&\,:=\{\mathbf w^{(1)}\dots\mathbf w^{(4)}\in\mathcal W(2N_1)\mid \mathbf w^{(\ell)}\in\mathcal Z\quad\text{for some }\ell\}.
\end{aligned}
\end{equation}

Using~\eqref{e:A0-rid} with $n=2N_1$ we have
\begin{equation}
  \label{e:A-XY-collected}
\|u\|_{L^2}\leq \|\mathcal A_{\mathcal X}u\|_{L^2}+\|\mathcal A_{\mathcal Y}u\|_{L^2}+C\|(-h^2\Delta_g-I)u\|_{L^2}
\end{equation}
and we will estimate the terms $\|A_{\mathcal X}u\|_{L^2}$, $\|A_{\mathcal Y}u\|_{L^2}$ separately.

\subsubsection{Estimating uncontrolled words}

We first estimate $\|\mathcal A_{\mathcal X}u\|_{L^2}$. In fact, we will bound the
operator norm of~$\mathcal A_{\mathcal X}$; in particular, this part of the argument
does not use the fact that $u$ is close to a Laplacian eigenfunction. We use that
the number of words in the set $\mathcal X$ grows like a small negative power of $h$
for small~$\alpha$, proved in the same way as~\cite[Lemma~3.3]{meassupp}
(which is a simple counting argument combined with Stirling's Formula): 
\begin{lemm}
\label{l:X-count}
Fix $\beta>0$. Then for $\alpha>0$ small enough depending on~$\beta$, there exists a constant $C$ such that
\begin{equation}
  \label{e:X-count}
\#(\mathcal X)\leq Ch^{-\beta/2}.  
\end{equation}
\end{lemm}
Combining the key estimate, Proposition~\ref{l:key-estimate}, with Lemma~\ref{l:X-count}, we get the bound
\begin{equation}
  \label{e:A-X-bound}
\|A_{\mathcal X}u\|_{L^2}\leq Ch^{\beta/2}\|u\|_{L^2}.
\end{equation}

\subsubsection{Estimating controlled words and end of the proof}

It remains to estimate~$\|A_{\mathcal Y}u\|_{L^2}$, which is done in the same way as the proof of~\cite[Proposition~3.4]{meassupp}.
We review the argument briefly, referring the reader to~\cite[\S4.3]{meassupp} for details.

We first give two basic estimates. The first one~\cite[Lemma~4.1]{meassupp},
uses a semiclassical elliptic estimate together with the property~\eqref{e:a-1-supp}
that $\supp a_1\cap S^*M\subset\{a\neq 0\}$ to conclude that
\begin{equation}
  \label{e:controller-1}
\|A_1u\|_{L^2}\leq C\|\Op_h(a)u\|_{L^2}+C\|(-h^2\Delta_g-I)u\|_{L^2}+Ch\|u\|_{L^2}.
\end{equation}
The second one has to do with propagation by the group $U(t)$
introduced in~\eqref{e:U-t-def}. If $u$ is an eigenfunction of $\Delta_g$,
then it is also an eigenfunction of $U(t)$; since the latter is unitary, 
for any operator $A$ on $L^2(M)$ we have for all $t\in\mathbb R$
$$
\|A(t)u\|_{L^2}=\|U(-t)AU(t)u\|_{L^2}=\|Au\|_{L^2}
$$
where $A(t) = U(-t) A U(t)$ is as defined in~\eqref{e:A-t-def}. More generally, for any $u\in H^2(M)$ we have~\cite[Lemma~4.2]{meassupp}
\begin{equation}
  \label{e:controller-2}
\|A(t)u\|_{L^2}\leq \|Au\|_{L^2}+{C|t|\over h}\|(-h^2\Delta_g-I)u\|_{L^2}
\end{equation}
for any $h$-dependent family of operators $A:L^2(M)\to L^2(M)$ bounded in norm uniformly in~$h$.

Coming back to estimating $\|A_{\mathcal Y}u\|_{L^2}$, we let
$\mathcal Z^\complement:=\mathcal W(N_0)\setminus\mathcal Z$ be the complement of~$\mathcal Z$
and decompose
$$
A_{\mathcal Y}=\sum_{\ell=1}^4 A_{\mathcal Z^\complement}(3N_0)\cdots A_{\mathcal Z^\complement}(\ell N_0)A_{\mathcal Z}((\ell-1)N_0)A_{\mathcal W((\ell-1)N_0)}.
$$
By Lemma~\ref{l:qc-Ac} with $\rho:=\frac 13(1-\varepsilon_0)$ the norms $\|A_{\mathcal Z}\|_{L^2\to L^2}$,
$\|A_{\mathcal Z^\complement}\|_{L^2\to L^2}$ are bounded uniformly in~$h$.
Together with~\eqref{e:A0-rid} and~\eqref{e:controller-2} this shows that
$\|A_{\mathcal Y}u\|_{L^2}$ is estimated in terms of~$\|A_{\mathcal Z}u\|_{L^2}$
(this is similar to the submultiplicativity argument in~\cite[\S2.2]{Anantharaman-Entropy}):
\begin{equation}
  \label{e:A-YZ-estimate}
\|A_{\mathcal Y}u\|_{L^2}\leq C\|A_{\mathcal Z}u\|_{L^2}+{C\log(1/h)\over h}\|(-h^2\Delta_g-I)u\|_{L^2}.
\end{equation}
Next, let $A_F$ be the operator defined in~\eqref{e:A-c-def}, corresponding
to the density function~$F$ defined in~\eqref{e:density-F-def}.
By the definition of the set $\mathcal Z$, we have
$$
0\leq \alpha \indic_{\mathcal Z}(\mathbf w)\leq F(\mathbf w)\leq 1\quad\text{for all }
\mathbf w\in\mathcal W(N_0).
$$
Applying Lemma~\ref{l:special-Garding} with $\rho:=\frac 13(1-\varepsilon_0)$, we then get (with the constants $C$ below depending on~$\alpha$)
\begin{equation}
  \label{e:A-Z-estimate}
\|A_{\mathcal Z}u\|_{L^2}\leq \alpha^{-1}\|A_F u\|_{L^2}+Ch^{\varepsilon_0\over 2} \|u\|_{L^2}.
\end{equation}
Finally, we write (in a way reminiscent of~\cite[\S2.5]{Anantharaman-Entropy})
$$
A_F={1\over N_0}\sum_{j=0}^{N_0-1}A_{\mathcal W(N_0-1-j)}A_1(j)A_{\mathcal W(j)}.
$$
Since $A_1+A_2=I-A_0$, we have $\|A_{\mathcal W(N_0-1-j)}\|_{L^2\to L^2}\leq 1$ by Lemma~\ref{l:our-partition} and~\eqref{e:A-W-n}. Using~\eqref{e:A0-rid} and~\eqref{e:controller-2} again, we see that
$$
\|A_Fu\|_{L^2}\leq \|A_1u\|_{L^2}+{C\log(1/h)\over h}\|(-h^2\Delta_g-I)u\|_{L^2}.
$$
Together with~\eqref{e:controller-1} this gives
\begin{equation}
  \label{e:A-F-estimate}
\|A_Fu\|_{L^2}\leq C\|\Op_h(a)u\|_{L^2}+{C\log(1/h)\over h}\|(-h^2\Delta_g-I)u\|_{L^2}+Ch\|u\|_{L^2}.
\end{equation}
Combining~\eqref{e:A-YZ-estimate}--\eqref{e:A-F-estimate}, we finally get the bound
on~$\|A_{\mathcal Y}u\|_{L^2}$:
\begin{equation}
  \label{e:A-Y-bound}
\|A_{\mathcal Y}u\|_{L^2}\leq C\|\Op_h(a)u\|_{L^2}+{C\log(1/h)\over h}\|(-h^2\Delta_g-I)u\|_{L^2}
+Ch^{\varepsilon_0\over 2}\|u\|_{L^2}.
\end{equation}
Together with~\eqref{e:A-XY-collected} and~\eqref{e:A-X-bound}, this gives
\begin{equation}
  \label{e:A-final-bound}
\|u\|_{L^2}\leq C\|\Op_h(a)u\|_{L^2}+{C\log(1/h)\over h}\|(-h^2\Delta_g-I)u\|_{L^2}
+Ch^{\min(\beta,\varepsilon_0)\over 2}\|u\|_{L^2}.
\end{equation}
Since $\beta$ and $\varepsilon_0$ are positive, for $h$ small enough
we can remove the last term on the right-hand side. This implies~\eqref{e:esti} and finishes the proof of Theorem~\ref{t:esti}.

\section{Decay for long words}
\label{s:key-proof}

In this section we prove Proposition~\ref{l:key-estimate}. Here is an outline of the proof:
\begin{itemize}
\item
The estimate~\eqref{e:key-estimate} is reduced to a norm bound on the product
of two operators, $\Op_h^{L_s}(a^-_{\mathbf w_-})$ and $\Op_h^{L_u}(a^+_{\mathbf w_+})$,
where $\Op_h^\bullet$ denotes the quantization reviewed in~\S\ref{s:calculus-lagrangian},
$L_s,L_u$ are the weak stable/unstable bundles, and the symbols $a^\pm_{\mathbf w_\pm}$
are constructed from the fixed symbols $a_1,a_2$ by the time evolution in forward ($-$)
or backward ($+$) time direction for time $N_1\approx {\rho\over 2}\log(1/h)$
defined in~\eqref{e:propagation-times}; this is half of the propagation time~$2N_1$
in Proposition~\ref{l:key-estimate} because we are propagating in both time directions.
Here we fix $\rho:=\tfrac23(1-\varepsilon_0)$.
\item We decompose the product above into a sum of pieces
$\Op_h^{L_s}(a^-_{\mathbf w_-}\psi_k)\Op_h^{L_u}(a^+_{\mathbf w_+}\psi_k)$,
where the $\psi_k^2$ form a partition of unity and each $\psi_k$ is supported in
the ball $B(q_k,2h^{\rho\over 2})$ centered at some point $q_k\in T^*M\setminus 0$.
The symbols $a^-_{\mathbf w_-}\psi_k$
belong to the $S^{\comp}_{L_s,\rho+\varepsilon,\rho/2}$ calculus,
and they can be quantized because $\tfrac 32\rho<1$; the
same is true for the symbols $a^+_{\mathbf w_+}\psi_k$ with $L_s$ replaced by $L_u$.
Then the decomposition above is almost orthogonal owing to the limited
overlap in the supports of $\psi_k$,
and thus by the Cotlar--Stein Theorem~\cite[Theorem~C.5]{Zworski-Book} it suffices to prove an estimate
on the norm of each piece, stated in~\eqref{e:key-estimate-4} below. 
\item For each individual piece, we conjugate the operators
$\Op_h^{L_s}(a^-_{\mathbf w_-}\psi_k)$ and $\Op_h^{L_u}(a^+_{\mathbf w_+}\psi_k)$
by some Fourier integral operators $\mathcal B,\mathcal B'$
quantizing a local symplectomorphism $\varkappa_k:T^*M\to T^*\mathbb R^{2n}$. This symplectomorphism
is chosen to straighten out the stable/unstable spaces,
and the decomposition of these into slow and fast parts, at the point $q_k$.
\item We study the images of the supports of the symbols
$a^\pm_{\mathbf w_\pm}\psi_k$ under the symplectomorphism $\varkappa_k$.
We show that they have projections onto the $y_1$ and $\eta_1$ variables
which are porous up to scale $\sim h^{\rho}$~-- see Lemma~\ref{l:porosity}.
This part of the proof uses that the the symbols $a^\pm_{\mathbf w_\pm}$
were defined using propagation for time $N_1\approx {\rho\over 2}\log(1/h)$
in two ways:
\begin{itemize}
\item In the slow stable/unstable directions, the symbols $a^\pm_{\mathbf w_\pm}$
vary on scales $e^{-N_1}\sim h^{\rho\over 2}$. Since we are intersecting with
$\supp\psi_k\subset B(q_k,2h^{\rho\over 2})$, we can essentially assume that
the symbols of interest are constant in the slow directions.
See in particular Lemmas~\ref{l:rectangle-propagate} and~\ref{l:Omega-approx}.
\item In the fast stable (for $a^+_{\mathbf w_+}$) and fast unstable (for $a^-_{\mathbf w_-}$)
directions, the symbols $a^\pm_{\mathbf w_\pm}$ vary on scales $e^{-2N_1}\sim h^\rho$.
This and the $V^\pm$\!-density of the complements of the supports of the symbols $a_1,a_2$
(see~\eqref{e:a-ell-supp}) imply the porosity property by a change of scale argument.
\end{itemize}
One also has to take care in the proof since $\varkappa_k$
straightens out the stable/unstable spaces only at one point $q_k$.
\item We next show that after conjugation by $\mathcal B,\mathcal B'$, the operators
$\Op_h^{L_s}(a^-_{\mathbf w_-}\psi_k)$ and $\Op_h^{L_u}(a^+_{\mathbf w_+}\psi_k)$
localize to porous sets in position ($y_1$) and in frequency ($\eta_1$), see Lemma~\ref{l:tilde-A-localized}.
This uses the information about the supports of the symbols described in the previous item
and some fairly technical analysis of the oscillatory integral forms of the operators in question.
\item The above arguments reduce Proposition~\ref{l:key-estimate} to an operator norm estimate on the product
of operators localizing in position and frequency,
$\indic_{\Omega_-}(hD_{y_1})\indic_{\Omega_+}(y_1)$,
where the sets $\Omega_\pm\subset\mathbb R$ are porous up to scale $\sim h^{\rho}$.
Since $\rho>\frac12$, the fractal uncertainty principle of~\cite{fullgap}
(or rather its extension from~\cite{varfup}) can be applied
to yield the desired estimate. Note that the above arguments
used that $\frac12<\rho<\frac23$, where the constant $\rho$
is related to the propagation time $N_1$. \end{itemize}

\subsection{Reduction to a localized estimate}

We first reduce to a localized estimate arguing similarly to~\cite[\S\S3.5,4.3.1--4.3.2]{highcat}.

\subsubsection{Writing $A_{\mathbf w}$ as a product of two operators}

Take arbitrary $\mathbf w\in \mathcal W(2N_1)$. We write $\mathbf w=\mathbf w_+\mathbf w_-$ as the concatenation of two words $\mathbf w_\pm\in \mathcal W(N_1)$, and denote
$$
\mathbf w_+=w^+_{N_1}\ldots w^+_1,\quad
\mathbf w_-=w^-_0\ldots w^-_{N_1-1}.
$$
Recalling the definition~\eqref{e:A-w-def} of~$A_{\mathbf w}$, we then write
$$
A_{\mathbf w}=U(-N_1)A^-_{\mathbf w_-}A^+_{\mathbf w_+}U(N_1)
$$
where
$$
\begin{aligned}
A^-_{\mathbf w_-}&\,:=A_{w^-_{N_1-1}}(N_1-1)\cdots A_{w^-_0}(0),\\
A^+_{\mathbf w_+}&\,:=A_{w^+_1}(-1)\cdots A_{w^+_{N_1}}(-N_1).
\end{aligned}
$$
Define the corresponding symbols
\begin{equation}
  \label{e:a-pm-s-def}
a^-_{\mathbf w_-}:= \prod_{j=0}^{N_1-1}(a_{w^-_j}\circ\varphi^j),\quad
a^+_{\mathbf w_+}:= \prod_{j=1}^{N_1}(a_{w^+_j}\circ\varphi^{-j}).
\end{equation}
Denote (where $\varepsilon_0$ is the constant in~\eqref{e:propagation-times})
\begin{equation}
  \label{e:our-rho}
\rho:=\tfrac23(1-\varepsilon_0).
\end{equation}
Then we have for all $\varepsilon>0$, with the implied constants independent of~$\mathbf w,h$,
\begin{equation}
  \label{e:a-pm-symb}
\begin{aligned}
a^-_{\mathbf w_-}\in S^{\comp}_{L_s,\rho+\varepsilon,\varepsilon}(T^*M\setminus 0),&\quad
A^-_{\mathbf w_-}=\Op_h^{L_s}(a^-_{\mathbf w_-})+\mathcal O(h^{\frac13})_{L^2\to L^2};\\
a^+_{\mathbf w_+}\in S^{\comp}_{L_u,\rho+\varepsilon,\varepsilon}(T^*M\setminus 0),&\quad
A^+_{\mathbf w_+}=\Op_h^{L_u}(a^+_{\mathbf w_+})+\mathcal O(h^{\frac13})_{L^2\to L^2}.
\end{aligned}
\end{equation}
Here the first line follows from Lemma~\ref{l:qc-Aw} and the second line is proved
in the same way, reversing the direction of propagation.

Since both $A^\pm_{\mathbf w_\pm}$ are bounded on $L^2$ uniformly in~$h$,
we see that Proposition~\ref{l:key-estimate} follows from the bound
\begin{equation}
  \label{e:key-estimate-2}
\|\Op_h^{L_s}(a^-_{\mathbf w_-})\Op_h^{L_u}(a^+_{\mathbf w_+})\|_{L^2\to L^2}\leq Ch^\beta.
\end{equation}

\subsubsection{Decomposing the operator}
\label{s:decomposer}

We next decompose the product of operators in~\eqref{e:key-estimate-2} as a sum of pieces.
Each piece corresponds to a ball of size $h^{\rho\over 2}> h^{\frac13}$ in the phase space $T^*M$.
The fact that the symbols $a^\pm_{\mathbf w_\pm}$ lie in Lagrangian calculi with parameters
$\rho+\varepsilon,\varepsilon$ where $\rho<\frac23$ make it possible to show that
the pieces are almost orthogonal and reduce~\eqref{e:key-estimate-2} to
a norm bound on each individual piece.

Let $q_1,\dots,q_L\in \{\tfrac15\leq |\xi|_g\leq 5\}\subset T^*M$
be a maximal $h^{\rho\over 2}$-separated set
(here $h^{\rho\over 2}$-separation means
that $d(q_k,q_{k'})\geq h^{\rho\over 2}$ for all $k\neq k'$). Since $T^*M$ is $4n$-dimensional,
we have for some $h$-independent constant $C$
\begin{equation}
  \label{e:L-limit}
L\leq Ch^{-2n\rho}.
\end{equation}
The balls $B(q_k,h^{\rho\over 2})$ cover $\{\tfrac 15\leq |\xi|_g\leq 5\}$.
Therefore we can construct an $h$-dependent partition of unity
\begin{equation}
  \label{e:psi-k}
\psi_k\in \CIc(T^*M),\quad
\supp\psi_k\subset B(q_k,2h^{\rho\over 2}),\quad
\sum_{k=1}^L\psi_k^2=1\quad\text{on }\{\tfrac 14\leq |\xi|_g\leq 4\}
\end{equation}
and the functions $\psi_k$ satisfy the derivative bounds for all multiindices $\alpha$
\begin{equation}
  \label{e:der-bounds}
\sup |\partial^\alpha \psi_k|\leq C_\alpha h^{-{\rho|\alpha|\over 2}}.
\end{equation}
For any fixed $k$, the balls $\{B(q_{k'},\frac 13h^{\rho\over 2})\mid \supp\psi_k\cap\supp\psi_{k'}\neq\emptyset\}$ are disjoint
and lie inside the ball $B(q_k,5h^{\rho\over 2})$. Comparing the volumes of these balls, we see
that there exists a constant $C$ independent of~$h$ such that
\begin{equation}
  \label{e:bounded-overlap}
\max_k\#\{k'\mid \supp \psi_k\cap \supp\psi_{k'}\neq\emptyset\}\leq C
\end{equation}
which implies that the sum $\sum_{k=1}^L\psi_k^2$ satisfies the derivative bounds~\eqref{e:der-bounds} as well.
Therefore, each $\psi_k$ and the sum $\sum_{k=1}^L\psi_k^2$ are bounded in the symbol
class $S^{\comp}_{\rho/2}(T^*M)$ introduced in~\eqref{e:S-rho},
and thus in the calculi $S^{\comp}_{L_s,\rho,\rho/2}$
and $S^{\comp}_{L_u,\rho,\rho/2}$.

By~\eqref{e:a-ell-supp}, we have $\supp a^-_{\mathbf w_-}\subset \{\tfrac14<|\xi|_g< 4\}$,
which shows that $a^-_{\mathbf w_-}=a^-_{\mathbf w_-}\sum_{k=1}^L\psi_k^2$. Then the Product Rule~\eqref{e:extra-product} for the $S^{\comp}_{L_s,\rho+\varepsilon,\rho/2}$ calculus together with~\eqref{e:quant-convert-S-rho} imply that
(here $\mathcal O(h^{\varepsilon_0-})$ denotes a function which is
$\mathcal O(h^{\varepsilon_0-\delta})$ for all $\delta>0$)
\begin{equation}
  \label{e:decomposer-1}
\Op_h^{L_s}(a^-_{\mathbf w_-})\Op_h^{L_u}(a^+_{\mathbf w_+})=
\bigg(\sum_{k=1}^L \Op_h^{L_s}(a^-_{\mathbf w_-})\Op_h(\psi_k^2)\Op_h^{L_u}(a^+_{\mathbf w_+})\bigg)+\mathcal O(h^{\varepsilon_0-})_{L^2\to L^2}.
\end{equation}
We now show that the summands in~\eqref{e:decomposer-1} form an almost orthogonal family:
\begin{lemm}
\label{l:almost-orth}
Denote $A^{(k)}:=\Op_h^{L_s}(a^-_{\mathbf w_-})\Op_h(\psi_k^2)\Op_h^{L_u}(a^+_{\mathbf w_+})$.
Then we have for some $h$-independent constant $C$
\begin{align}
  \label{e:almost-orth-1}
\max_k\sum_{k'=1}^L \|(A^{(k)})^* A^{(k')}\|_{L^2\to L^2}^{\frac12}\leq C\max_k \|A^{(k)}\|_{L^2\to L^2}+\mathcal O(h^\infty),\\
  \label{e:almost-orth-2}
\max_k\sum_{k'=1}^L \|A^{(k)}(A^{(k')})^*\|_{L^2\to L^2}^{\frac12}\leq C\max_k \|A^{(k)}\|_{L^2\to L^2}+\mathcal O(h^\infty).
\end{align}
\end{lemm}
\begin{proof}
We show~\eqref{e:almost-orth-1}, with~\eqref{e:almost-orth-2} proved similarly.
Assume first that $\supp \psi_k\cap \supp\psi_{k'}=\emptyset$. Then
$$
\|(A^{(k)})^* A^{(k')}\|_{L^2\to L^2}\leq C\|\Op_h(\psi_k^2)^*\Op_h^{L_s}(a^-_{\mathbf w_-})^*\Op_h^{L_s}(a^-_{\mathbf w_-})\Op_h(\psi_{k'}^2)\|_{L^2\to L^2}=\mathcal O(h^\infty).
$$
Here the last bound is similar to the Nonintersecting Support Property~\eqref{e:extra-nonint-supp}, following from the asymptotic expansions in the Product Rule~\eqref{e:extra-product} and the Adjoint Rule~\eqref{e:extra-adjoint} for the $S^{\comp}_{L_s,\rho+\varepsilon,\rho/2}$ calculus (see~\cite[(A.23)--(A.24)]{meassupp}) together with the asymptotic expansion for the change of quantization formula~\eqref{e:quant-convert-S-rho}. The fact that $\supp\psi_k\cap\supp\psi_{k'}=\emptyset$ implies that all the terms in the asymptotic
expansion for the full symbol of the product of four operators above are equal to~0.

Since the number of terms $L$ is bounded polynomially in $h$ by~\eqref{e:L-limit}, we see that the left-hand side
of~\eqref{e:almost-orth-1} is bounded above by
$$
\max_k\sum_{1\leq k'\leq L\atop \supp\psi_k\cap \supp\psi_{k'}\neq\emptyset}
\|(A^{(k)})^* A^{(k')}\|_{L^2\to L^2}^{\frac12}+\mathcal O(h^\infty)\leq
C\max_k\|A^{(k)}\|_{L^2\to L^2}+\mathcal O(h^\infty)
$$
where the last inequality follows from~\eqref{e:bounded-overlap}.
This gives~\eqref{e:almost-orth-1}.
\end{proof}
Using~\eqref{e:decomposer-1}, \eqref{e:almost-orth-1}--\eqref{e:almost-orth-2},
and the Cotlar--Stein Theorem~\cite[Theorem~C.5]{Zworski-Book}, we see that~\eqref{e:key-estimate-2} reduces to the following bound on the norm of each $A^{(k)}$:
\begin{equation}
  \label{e:key-estimate-3}
\max_k\|\Op_h^{L_s}(a^-_{\mathbf w_-})\Op_h(\psi_k^2)\Op_h^{L_u}(a^+_{\mathbf w_+})\|_{L^2\to L^2}\leq Ch^\beta.
\end{equation}
By~\eqref{e:quant-convert-S-rho} and the Product Rule~\eqref{e:extra-product}
for the $S^{\comp}_{L_s,\rho+\varepsilon,\rho/2}$ and $S^{\comp}_{L_u,\rho+\varepsilon,\rho/2}$
calculi, we have
$$
\begin{aligned}
\Op_h^{L_s}(a^-_{\mathbf w_-})\Op_h(\psi_k)&=\Op_h^{L_s}(a^-_{\mathbf w_-}\psi_k)+\mathcal O(h^{\varepsilon_0-})_{L^2\to L^2},\\
\Op_h(\psi_k)\Op_h^{L_u}(a^+_{\mathbf w_+})&=\Op_h^{L_u}(a^+_{\mathbf w_+}\psi_k)+\mathcal O(h^{\varepsilon_0-})_{L^2\to L^2}.
\end{aligned}
$$
We also have $\Op_h(\psi_k^2)=\Op_h(\psi_k)^2+\mathcal O(h^{\frac 13})$
by the properties of the $S^{\comp}_{\rho/2}$ calculus. Therefore~\eqref{e:key-estimate-3}
follows from the bound
\begin{equation}
\label{e:key-estimate-4}
\max_k\|\Op_h^{L_s}(a^-_{\mathbf w_-}\psi_k)\Op_h^{L_u}(a^+_{\mathbf w_+}\psi_k)\|_{L^2\to L^2}\leq Ch^\beta.
\end{equation}

\subsection{Fractal Uncertainty Principle}

We next review the Fractal Uncertainty Principle (FUP) of~\cite{fullgap}.
We use the slightly more general version from~\cite{varfup}.

To state FUP, we need the following
\begin{defi}
\label{d:porous}
Let $\nu\in (0,1)$ and $0<\alpha_0\leq\alpha_1$. We say that a subset
$\Omega\subset\mathbb R$ is $\nu$-porous on scales $\alpha_0$ to~$\alpha_1$
if for each interval $I\subset\mathbb R$ of length $|I|\in[\alpha_0,\alpha_1]$
there exists a subinterval $J\subset I$ of length $|J|=\nu|I|$ such that
$J\cap\Omega=\emptyset$.
\end{defi}
We also recall the semiclassical unitary Fourier transform $\mathcal F_h$
on $L^2(\mathbb R)$, defined by
\begin{equation}
  \label{e:F-h-def}
\mathcal F_hf(\xi)=(2\pi h)^{-\frac12}\int_{\mathbb R}e^{-{ix\xi\over h}}f(x)\,dx.
\end{equation}
We can now state a special case of the FUP from~\cite[Proposition~2.10]{varfup}:
\begin{prop}
\label{l:fup-1}
Fix numbers $\gamma_0,\gamma_1$ such that
$$
0\leq\gamma_1<\tfrac12<\gamma_0\leq 1.
$$
Then for each $\nu\in (0,1)$ there exist $\beta=\beta(\nu,\gamma_0,\gamma_1)>0$ and $C=C(\nu,\gamma_0,\gamma_1)$
such that the estimate
\begin{equation}
  \label{e:fup-1}
\|\indic_{\Omega_-}\mathcal F_h\indic_{\Omega_+}\|_{L^2(\mathbb R)\to L^2(\mathbb R)}\leq Ch^\beta
\end{equation}
holds for all $0<h<1$ and all sets $\Omega_\pm\subset\mathbb R$
which are $\nu$-porous on scales $h^{\gamma_0}$ to $h^{\gamma_1}$.
Here $\indic_{\Omega}$ denotes the multiplication operator
by the indicator function of~$\Omega$.
\end{prop}
In~\S\ref{s:putting-together} below, we will use the following corollary
of Proposition~\ref{l:fup-1} featuring operators on $L^2(\mathbb R^{2n})$. Here we recall that $D_{y_j}=-i\partial_{y_j}$ and for any bounded measurable function $\chi$ on~$\mathbb R$ the operator $\chi(D_{y_j})$ is a Fourier multiplier (here $\mathcal F$ denotes the Fourier transform, with $\mathcal Ff=\widehat f$\,):
\begin{equation}
  \label{e:fourier-mul}
\mathcal F(\chi(D_{y_j}) f)(\eta)=\chi(\eta_j)\widehat f(\eta)\quad\text{for all }
f\in L^2(\mathbb R^{2n}),\
\eta\in\mathbb R^{2n}.
\end{equation}
\begin{prop}
\label{l:fup-specific}
Assume that $0<\varepsilon_0<{1\over 4}$, $\rho={2\over 3}(1-\varepsilon_0)$ as in~\eqref{e:our-rho},
$\nu>0,C_0$ are constants, and $\Omega_-,\Omega_+\subset\mathbb R$
are $\nu$-porous on scales $C_0h^{\rho}$ to~1. Then for all $h\in (0,1)$
\begin{equation}
  \label{e:fup-specific}
\|\indic_{\Omega_-}(hD_{y_1})\indic_{\Omega_+}(y_1)\|_{L^2(\mathbb R^{2n})\to L^2(\mathbb R^{2n})}\leq Ch^\beta
\end{equation}
where $\beta>0$ depends only on $\nu,\varepsilon_0$ and $C$ depends only on $\nu,\varepsilon_0,C_0$.
\end{prop}
\begin{proof}
Fix $\gamma_1:=0$ and $\gamma_0:={1+2\rho\over 4}\in (\frac 12,\rho)$. If $h\leq c_1$
where $c_1>0$ is a small constant
depending only on~$C_0,\varepsilon_0$, then $C_0h^\rho\leq h^{\gamma_0}$
and thus $\Omega_\pm$ are $\nu$-porous on scales $h^{\gamma_0}$ to~$h^{\gamma_1}=1$. Take
$f\in L^2(\mathbb R^{2n})$. For almost every $y'\in \mathbb R^{2n-1}$, define the
function $f_{y'}\in L^2(\mathbb R)$ by $f_{y'}(y_1)=f(y_1,y')$. Then
$$
(\indic_{\Omega_-}(hD_{y_1})\indic_{\Omega_+}(y_1)f)(y_1,y')=g_{y'}(y_1)\quad\text{where }
g_{y'}:=\mathcal F_h^{-1}\indic_{\Omega_-}\mathcal F_h\indic_{\Omega_+}f_{y'}.
$$
Since $\mathcal F_h$ is unitary, Proposition~\ref{l:fup-1} implies that for almost every $y'$
$$
\|g_{y'}\|_{L^2(\mathbb R)}\leq Ch^\beta \|f_{y'}\|_{L^2(\mathbb R)}.
$$
Taking the squares of both sides and integrating in~$y'\in\mathbb R^{2n-1}$,
we get
$$
\|\indic_{\Omega_-}(hD_{y_1})\indic_{\Omega_+}(y_1)f\|_{L^2(\mathbb R^{2n})}\leq Ch^\beta \|f\|_{L^2(\mathbb R^{2n})}
$$
which gives~\eqref{e:fup-specific}.

On the other hand, if $c_1<h<1$ then~\eqref{e:fup-specific} follows from
the trivial bound $\|\indic_{\Omega_-}(hD_{y_1})\indic_{\Omega_+}(y_1)\|_{L^2(\mathbb R^{2n})\to L^2(\mathbb R^{2n})}\leq 1$.
\end{proof}

\subsection{Local normal coordinates and proof of porosity}
\label{s:proof-of-porosity}

We now start the proof of~\eqref{e:key-estimate-4}. Fix $k$ and let $q_k\in \{\frac15\leq |\xi|_g\leq 5\}$ be the corresponding point chosen at the beginning of~\S\ref{s:decomposer}.

Let $\varkappa_k:U_k\to T^*\mathbb R^{2n}$ be the symplectomorphism
constructed in Lemma~\ref{l:straighten-out} with $q^0:=q_k$. Recall that it satisfies the properties~\eqref{e:straighten-out-1} and~\eqref{e:straighten-out-2}--\eqref{e:straighten-out-3}:
$$
\varkappa_k(q_k)=0,\quad
d\varkappa_k(q_k)V^+_\perp(q_k)=\ker dy_1,\quad
d\varkappa_k(q_k)V^-_\perp(q_k)=\ker d\eta_1
$$
where the `slow' hyperplanes $V^\pm_\perp(q)\subset T_q(T^*M\setminus 0)$
were defined in~\eqref{e:V-perp-def}. It follows from the construction in Lemma~\ref{l:straighten-out} that we can make each derivative of $\varkappa_k$
bounded uniformly in~$k$.

The goal of this section is to show that the images
under the symplectomorphism~$\varkappa_k$
of the supports of the symbols $a^\pm_{\mathbf w_\pm}\psi_k$
featured in~\eqref{e:key-estimate-4} 
project to porous sets in $y_1$ and $\eta_1$ variables.
As in~\eqref{e:our-rho} we put $\rho:={2\over 3}(1-\varepsilon_0)$.
\begin{lemm}
\label{l:porosity}
There exist sets $\Omega_\pm\subset \mathbb R$ such that
\begin{align}
  \label{e:porosity-cont-1}
\varkappa_k(\supp(a^+_{\mathbf w_+}\psi_k))\ &\subset\ \{(y,\eta)\mid y_1\in\Omega_+\},\\
  \label{e:porosity-cont-2}
\varkappa_k(\supp(a^-_{\mathbf w_-}\psi_k))\ &\subset\ \{(y,\eta)\mid \eta_1\in\Omega_-\}
\end{align}
and the sets $\Omega_\pm$ are $\nu$-porous on scales $C_0h^{\rho}$ to~1,
for some constants $\nu>0$ and~$C_0$ which only depend on the manifold $(M,g)$, the (uniform in $k$) bounds on derivatives of the maps $\varkappa_k$, and the sets~$\mathcal V_\ell$ in~\eqref{e:a-ell-supp}, and in particular do not depend
on $h$ or~$k$.
\end{lemm}
We will only show~\eqref{e:porosity-cont-2}, with~\eqref{e:porosity-cont-1}
proved in the same way, reversing the direction of propagation.
From the definition~\eqref{e:a-pm-s-def} of $a^-_{\mathbf w_-}$ and the support
property~\eqref{e:a-ell-supp} of the symbols $a_1,a_2$ of the original partition,
we see that
\begin{equation}
\label{e:a-pm-supp}
\supp a^-_{\mathbf w_-}\ \subset\ \bigg(\bigcap_{j=0}^{N_1-1}\varphi^{-j}(\mathcal V_{w_j^-})\bigg)\cap \big\{\tfrac 14\leq |\xi|_g\leq 4\big\}
\end{equation}
where $\mathcal V_1,\mathcal V_2\subset T^*M\setminus 0$ are the closed conic sets featured in~\eqref{e:a-ell-supp}.
Recall that the complements $S^*M\setminus \mathcal V_1,S^*M\setminus\mathcal V_2$
are both $V^+$\!-dense and $V^-$\!-dense in the sense of~\S\ref{s:orbits}. Therefore
by Lemma~\ref{l:dense-basic}(1) there exist closed conic sets
$$
K_1,K_2\subset T^*M\setminus 0,\quad
\mathcal V_\ell\cap K_\ell=\emptyset
$$
such that $S^*M\cap K_\ell$ are both $V^+$\!-dense and $V^-$\!-dense.
Fix open conic sets
\begin{equation}
  \label{e:V-sharp}
\mathcal V_1^\sharp,\mathcal V_2^\sharp\subset T^*M\setminus 0,\quad\mathcal V_\ell\subset \mathcal V_\ell^\sharp,\quad
\overline{\mathcal V_\ell^\sharp}\cap K_\ell=\emptyset.
\end{equation}
To avoid wasting indices, we next choose a large constant $C_1$ depending only
on the manifold $(M,g)$, the (uniform in $k$) bounds on derivatives of the maps $\varkappa_k$, and the sets $\mathcal V_\ell,\mathcal V_\ell^\sharp$ such that:
\begin{enumerate}
\item we have
\begin{equation}
\label{e:psi-k-supp}
\supp \psi_k\ \subset\ \varkappa_k^{-1}\big(\{(y,\eta)\colon |y|+|\eta|\leq C_1h^{\rho\over 2}\}\big).
\end{equation}
This is possible by~\eqref{e:psi-k};
\item 
we have the upper bound on the derivatives of the trajectory $s\mapsto \varkappa_k(e^{sV^-}(q_k))$
\begin{equation}
  \label{e:V-regular}
\begin{aligned}
|\partial_s y(\varkappa_k(e^{sV^-}(q_k)))|+
|\partial_s \eta(\varkappa_k(e^{sV^-}(q_k)))|&\leq C_1\quad\text{for all }
s\in [-C_1^{-1},C_1^{-1}],\\
  |\partial_s^2 \eta_1(\varkappa_k(e^{sV^-}(q_k)))|&\leq C_1\quad\text{for all }
s\in [-C_1^{-1},C_1^{-1}];
\end{aligned}
\end{equation}
\item
we have the lower bound on the derivative of the $\eta_1$-component of the above trajectory:
\begin{equation}
  \label{e:V-lower}
|\partial_s \eta_1(\varkappa_k(e^{sV^-}(q_k)))|\geq C_1^{-1}\quad\text{for all }
s\in [-C_1^{-1},C_1^{-1}].
\end{equation}
This is possible since $V^-(\eta_1\circ\varkappa_k)(q_k)\neq 0$
by~\eqref{e:straighten-out-3} (as $V^-$ is transverse to $V^-_\perp$
by~\eqref{e:V-perp-def});
\item the distance between the set $\mathcal V_\ell\cap \{\tfrac1{4}\leq |\xi|_g\leq 4\}$
and the complement of the set $\mathcal V_\ell^\sharp$ is at least $C_1^{-1}$: 
\begin{equation}
\label{e:V-sharp-distance}
q\in \mathcal V_\ell\cap \{\tfrac 14\leq |\xi|_g\leq 4\},\quad
d(q,q')\leq C_1^{-1}\quad\Longrightarrow\quad
q'\in \mathcal V_\ell^\sharp.
\end{equation}
\end{enumerate}
We now define the set $\Omega_-$, which corresponds to the intersection of the $V^-$-trajectory 
$\{e^{sV^-}(q_k)\mid s\in\mathbb R\}$ and the set on the right-hand side of~\eqref{e:a-pm-supp}, with $\mathcal V_\ell$ replaced by the larger sets $\mathcal V_\ell^\sharp$
and the time of propagation reduced by an $h$-independent constant $C_2$
to be chosen later in~\eqref{e:C-2-choice}. We first define
the set $\widetilde\Omega_-$ which uses the parametrization of the trajectory
by~$s$:
\begin{equation}
  \label{e:Omega-tilde-def}
\widetilde\Omega_-:=\bigg\{s\in[-C_1^{-1},C_1^{-1}]\colon e^{sV^-}(q_k)\in\bigcap_{j=0}^{N_1-C_2}\varphi^{-j}(\mathcal V^\sharp_{w_j^-})\bigg\}.
\end{equation}
To obtain $\Omega_-$ from here, we instead parametrize by the variable $\eta_1\circ\varkappa_k$
and intersect with the set featured in~\eqref{e:psi-k-supp}:
\begin{equation}
  \label{e:Omega-def}
\Omega_-:=\eta_1(\varkappa_k(\{e^{sV^-}(q_k)\mid s\in\widetilde\Omega_-\}))\cap [-C_1h^{\rho\over 2},C_1h^{\rho\over 2}].
\end{equation}
Now Lemma~\ref{l:porosity} follows from the two lemmas below:
\begin{lemm}
  \label{l:Omega-approx}
For $C_2$ large enough depending only on the manifold $(M,g)$, the derivative bounds
on the maps $\varkappa_k$, and the constant $C_1$, the inclusion~\eqref{e:porosity-cont-2} holds.  
\end{lemm}
\begin{figure}
\includegraphics{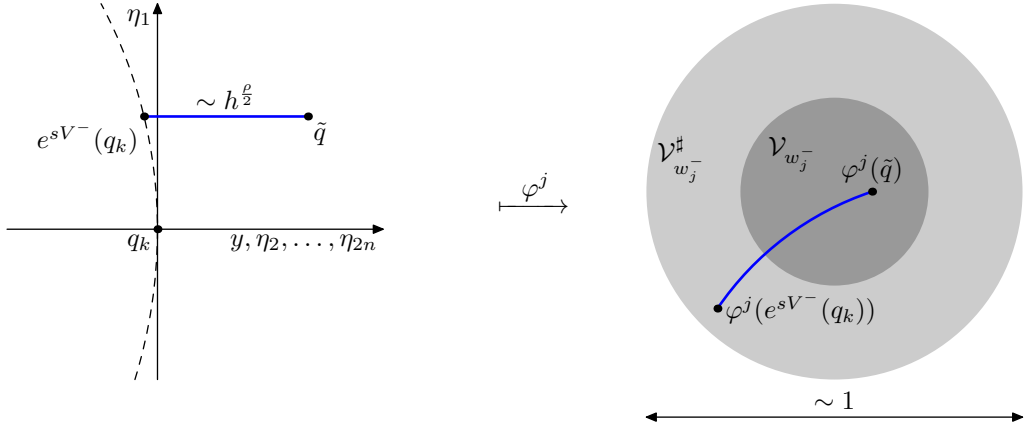}
\caption{An illustration of the proof of Lemma~\ref{l:Omega-approx}.
On the left is an $h^{\rho\over 2}$-sized neighborhood of the point $q_k$,
viewed in the coordinates $(y,\eta)$ given by the symplectomorphism~$\varkappa_k$.
The dashed curve is the flow line of $V^-$ passing through~$q_k$.
The blue line lies in the disk $\mathcal R^-$, which has diameter $\sim h^{\rho\over 2}$. On the right is the image of the left side by $\varphi^j$, with
the blue line contained in the image $\varphi^j(\mathcal R^-)$.
Even though $j$ can be as large as ${\rho\over 2}\log(1/h)$ and
the flow $\varphi^j$ can expand by $e^{2j}$, the diameter of $\varphi^j(\mathcal R^-)$
is still smaller than~1. This is proved in Lemma~\ref{l:rectangle-propagate}
and uses that the `slow unstable' space $V_\perp^-(q_k)$
is horizontal on the left side of the picture.
The shaded sets are $\mathcal V_{w_j^-}$ and $\mathcal V_{w_j^-}^\sharp$.}
\label{f:porosity-1}
\end{figure}
\begin{proof}
1. Take arbitrary $\tilde q\in \supp(a^-_{\mathbf w_-}\psi_k)$. We need to show that
$$
\tilde\eta_1\in\Omega_-\quad\text{where }\tilde\eta_1:=\eta_1(\varkappa_k(\tilde q)).
$$
Note that $|\tilde\eta_1|\leq C_1h^{\rho\over 2}$ by~\eqref{e:psi-k-supp}, so in particular $|\tilde\eta_1|\leq C_1^{-2}$
for $h$ small enough depending on~$C_1$.
Then it follows from~\eqref{e:V-lower} (and the fact that $\varkappa_k(q_k)=0$)
that there exists $s\in \mathbb R$ such that $|s|\leq C_1^2h^{\rho\over 2}\leq C_1^{-1}$ and
$$
\eta_1(\varkappa_k(e^{sV^-}(q_k)))=\tilde\eta_1.
$$
It suffices to show that $s\in \widetilde\Omega_-$. See Figure~\ref{f:porosity-1}.

\noindent 2. By~\eqref{e:psi-k-supp} and~\eqref{e:V-regular}, both $\tilde q$ and $e^{sV^-}(q_k)$ lie in the codimension 1 disk
$$
\mathcal R^-:=\varkappa_k^{-1}\big(\{(y,\eta)\colon |y|+|\eta|\leq C_1^3h^{\rho\over 2},\
\eta_1=\tilde\eta_1\}\big).
$$
By Lemma~\ref{l:rectangle-propagate} with $\alpha:=C_1^3h^{\rho\over 2}$ there exists a constant $C_3$ depending only
on the manifold~$(M,g)$ and the derivative bounds on the maps $\varkappa_k$ such that for all $j\geq 0$
$$
d\big(\varphi^j(\tilde q),\varphi^j(e^{sV^-}(q_k))\big)\leq C_3C_1^3 h^{\rho\over 2}e^j.
$$
We now choose $C_2$ large enough so that
\begin{equation}
  \label{e:C-2-choice}
e^{C_2}\geq 10C_3C_1^4.
\end{equation}
Take arbitrary $j\in \{0,1,\dots,N_1-C_2\}$.
Recalling the definition~\eqref{e:propagation-times} of~$N_1$ and the fact that $\rho={2\over 3}(1-\varepsilon_0)$, we see that 
\begin{equation}
  \label{e:Omega-approx-int-1}
d\big(\varphi^j(\tilde q),\varphi^j(e^{sV^-}(q_k))\big)\leq 10C_3C_1^3e^{-C_2}\leq C_1^{-1}.
\end{equation}
We have $\varphi^j(\tilde q)\in \mathcal V_{w^-_j}\cap \{\tfrac 14\leq |\xi|_g\leq 4\}$
by~\eqref{e:a-pm-supp}. Then by~\eqref{e:Omega-approx-int-1}
and~\eqref{e:V-sharp-distance}
we get $\varphi^j(e^{sV^-}(q_k))\in\mathcal V_{w^-_j}^\sharp$.
It follows that $s\in \widetilde\Omega_-$, finishing the proof.
\end{proof}
%
\begin{lemm}
  \label{l:Omega-porous}
The set $\Omega_-$ defined in~\eqref{e:Omega-def} is $\nu$-porous on scales $C_0h^\rho$ to~1,
for some constants $\nu>0$ and $C_0$ which only depend on the sets $\mathcal V_\ell^\sharp,K_\ell$ and the constants $C_1,C_2$.
\end{lemm}
\begin{proof}
1. We first make some preparatory arguments. By~\eqref{e:V-sharp}, we may fix open conic sets
for $\ell\in\{1,2\}$
$$
\mathcal U_\ell\subset T^*M\setminus 0,\quad
\overline{\mathcal U_\ell}\cap\overline{\mathcal V_\ell^\sharp}=\emptyset,\quad
K_\ell\subset \mathcal U_\ell.
$$
We use the notation of~\S\ref{s:orbits}. Since $S^*M\cap K_\ell$
is $V^-$\!-dense, $S^*M\cap \mathcal U_\ell$ is $V^-$\!-dense as well.
By Lemma~\ref{l:dense-basic}(2), there exists $T\geq 1$ such that
each $V^-$\!-segment of length $T$ in~$S^*M$ intersects $\mathcal U_\ell$.
Since $\overline{\mathcal U_\ell}\cap\overline{\mathcal V_\ell^\sharp}=\emptyset$,
there exists $\delta>0$ such that each $V^-$\!-segment of length~$T$
in~$S^*M$ has a subsegment of length $\delta$ which does not intersect
$\mathcal V_\ell^\sharp$.
 Since the vector field $V^-$ is extended homogeneously from $S^*M$
to $T^*M\setminus 0$ and $\mathcal V_\ell^\sharp$ is a conic set,
we see that the previous statement extends to all $V^-$\!-segments of length~$T$
in~$T^*M\setminus 0$.

We define constants 
\begin{equation}
  \label{e:nu-chosen-ish}
\nu':=e^{-2}T^{-1}\delta,\quad
C_0':=e^{2(C_2+1)}T.
\end{equation}

\begin{figure}
\includegraphics{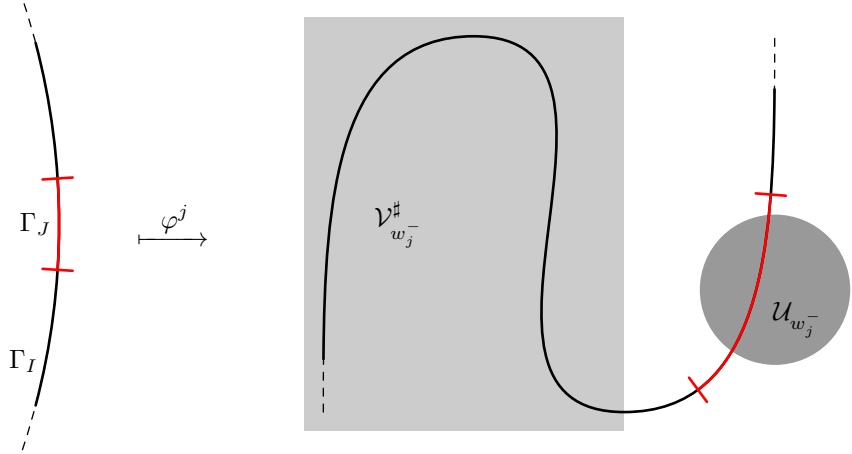}
\caption{An illustration of the proof of porosity of $\widetilde\Omega_-$
in Lemma~\ref{l:Omega-porous}. On the left, the dashed curve is the flow line of $V^-$
passing through $q_k$. The solid black curve is the segment $\Gamma_I$
and the red curve inside of it is the segment $\Gamma_J$. This segment
is obtained as follows: we propagate $\Gamma_I$ by $\varphi^j$ to yield
the picture on the right, where $j$ is chosen in~\eqref{e:j-choice}.
Then $\varphi^j(\Gamma_I)$ is a long enough $V^-$\!-segment
that it intersects the set $\mathcal U_{w_j^-}$
and thus contains a length $\delta$ subsegment which does
not intersect $\mathcal V_{w_j^-}^\sharp$. Now $\Gamma_J$
is the image of the latter subsegment by $\varphi^{-j}$.}
\label{f:porosity-2}
\end{figure}

\noindent 2. We now show that the set $\widetilde\Omega_-$ defined in~\eqref{e:Omega-tilde-def}
is $\nu'$-porous on scales $C'_0h^\rho$ to~1. We use the following corollary of~\eqref{e:stun-exp}:
for each $t\in\mathbb R$, the image under $\varphi^t$ of
a $V^-$-segment of length $\alpha$ is a $V^-$\!-segment of length $e^{2t}\alpha$.

Let $I\subset \mathbb R$ be an interval of length $|I|\in [C'_0h^\rho,1]$.
Choose $j\in \mathbb Z$ such that
\begin{equation}
  \label{e:j-choice}
T\leq e^{2j}|I|\leq e^2T.
\end{equation}
Since $|I|\leq 1\leq T$, we have $j\geq 0$. Moreover, we have
$C'_0h^\rho\leq |I|\leq e^{2-2j}T$. Recalling that $\rho=\tfrac 23(1-\varepsilon_0)$
and the definition~\eqref{e:propagation-times} of~$N_1$, we see that
$$
j\leq \tfrac12\rho\log(1/h)-C_2\leq N_1-C_2.
$$
Define $\Gamma_I:=\{e^{sV^-}(q_k)\mid s\in I\}$ which is a $V^-$\!-segment
in $T^*M\setminus 0$ of length $|I|$. Then $\varphi^j(\Gamma_I)$
is a $V^-$\!-segment of length $e^{2j}|I|\geq T$.
From Step~1 of this proof
we know that there exists a subsegment of $\varphi^j(\Gamma_I)$ of length $\delta$
which does not intersect $\mathcal V_{w_j^-}^\sharp$.
We can write this subsegment
as $\varphi^j(\Gamma_J)$ where $\Gamma_J=\{e^{sV^-}(q_k)\mid s\in J\}$
and $J\subset I$ is a subinterval of length
$$
|J|=e^{-2j}\delta\geq \nu'|I|.
$$
For each $s\in J$, we have $\varphi^j(e^{sV^-}(q_k))\notin \mathcal V_{w_j^-}^\sharp$.
Recalling~\eqref{e:Omega-tilde-def}, this shows that $J\cap\widetilde\Omega_-=\emptyset$.
This finishes the proof of porosity of~$\widetilde\Omega_-$. See Figure~\ref{f:porosity-2}.

\noindent 3. We finally show the porosity of the set $\Omega_-$.
Let $\psi(s)=\eta_1(\varkappa_k(e^{sV^-}(q_k)))$ for $|s|\leq C_1^{-1}$.
By~\eqref{e:V-regular} and~\eqref{e:V-lower}, we can extend $\psi$
to a diffeomorphism of~$\mathbb R$ (still denoted $\psi$) which satisfies the bounds
$$
\max(\sup|\psi'|,\sup|\psi'|^{-1},\sup|\psi''|)\leq 2C_1.
$$
By~\eqref{e:Omega-def} we have $\Omega_-\subset\psi(\widetilde\Omega_-)$.
Now the porosity property of $\widetilde\Omega_-$ established in Step~2 of this proof together with~\cite[Lemma~2.12]{varfup} show that $\Omega_-$ is $\nu$-porous on scales $C_0h^{\rho}$ to~$\alpha_1$, with
\begin{equation}
  \label{e:nu-chosen}
\nu:=\tfrac12 \nu',\quad
C_0:=2C_1C'_0,\quad
\alpha_1:=\tfrac12 C_1^{-3}.
\end{equation}
Since $\Omega_-\subset [-C_1h^{\rho\over 2},C_1h^{\rho\over 2}]$,
we see from the definition of porosity that $\Omega_-$ is also $\nu$-porous (in fact, $\tfrac13$-porous) on scales $\alpha_1$ to~1,
if $h$ is small enough depending on~$C_1$.
\end{proof}

\subsection{Fourier integral operators}

The proof of~\eqref{e:key-estimate-4} uses conjugation by Fourier integral
operators quantizing the symplectomorphism $\varkappa_k$.
This makes it possible to replace the operators
$\Op_h^{L_s}(a^-_{\mathbf w_-}\psi_k)$ and $\Op_h^{L_u}(a^+_{\mathbf w_+}\psi_k)$ 
by localization operators in $\eta_1$ and~$y_1$ to the porous sets $\Omega_\pm$
appearing in Lemma~\ref{l:porosity} and then apply the Fractal Uncertainty Principle
of Proposition~\ref{l:fup-specific}.
In this section we introduce parts
of the theory of semiclassical Fourier integral operators that will be needed in~\S\ref{s:end-proof}
below.

\subsubsection{Review of general theory}
\label{s:fios}

We first briefly review the general theory of Fourier integral operators,
following~\cite[\S2.2]{hgap}, \cite[\S A.3]{meassupp}, and~\cite[\S2.3]{varfup}.
We refer the reader to~\cite{Alexandrova-FIO},
\cite[Chapter~5]{Guillemin-Sternberg-old}, and \cite[Chapter~8]{Guillemin-Sternberg-new} for a more detailed treatment
and to~\cite[Chapter~25]{Hormander4} and~\cite[Chapters~10--11]{Grigis-Sjostrand} for the related nonsemiclassical case.

Let $M$ be a $d$-dimensional manifold and $\Lambda\subset T^*M$ be a Lagrangian submanifold,
that is $\dim\Lambda=d$ and the symplectic form $\omega$ vanishes when restricted
to the tangent spaces of $\Lambda$. Denote by $I^{\comp}_h(\Lambda)$ the space
of compactly microlocalized semiclassical Lagrangian distributions associated to~$\Lambda$.
Each element of $I^{\comp}_h(\Lambda)$ is an $h$-dependent family of compactly supported
functions in $\CIc(M)$.

An important special case is when $\Lambda$ projects diffeomorphically onto the $x$ variables,
which (given that $\Lambda$ is Lagrangian, and assuming that $\Lambda$ is simply connected) means it is the graph of a gradient:
\begin{equation}
  \label{e:Lambda-graph}
\Lambda=\{(x,\xi)\mid x\in U,\ \xi=\partial_x\Phi(x)\}
\end{equation}
where $U\subset M$ is an open set and $\Phi\in C^\infty(U;\mathbb R)$.
Then elements of $I^{\comp}_h(\Lambda)$ have the following form:
\begin{equation}
  \label{e:Lambda-graph-lag}
u(x;h)=e^{{i\over h}\Phi(x)}a(x;h)+\mathcal O(h^\infty)_{\CIc(M)}.
\end{equation}
Here the amplitude $a\in \CIc(U)$ is supported in an $h$-independent compact subset of~$U$
and has $x$-derivatives of all orders bounded uniformly in~$h$, and the residual class $\mathcal O(h^\infty)_{\CIc(M)}$ consists
of smooth functions supported in an $h$-independent compact subset of $M$
and with derivatives of all orders bounded by $\mathcal O(h^N)$ for each $N$.

In~\cite{hgap,meassupp,varfup} one made the additional assumption that $\Lambda$
is an \emph{exact} Lagrangian submanifold and fixed an antiderivative on $\Lambda$.
For the Lagrangian submanifold~\eqref{e:Lambda-graph} this has the effect of removing
the freedom of adding a constant to $\Phi$. We will be working with
the cases when $U$ is a simply connected set (typically a small ball centered at some point)
so all the Lagrangian submanifolds and symplectomorphisms used will
be exact, and we do not need to fix an antiderivative. 

Next, assume that $M_1,M_2$ are two manifolds of the same dimension~$d$
and $\varkappa:U_2\to U_1$ is a symplectomorphism, where $U_1\subset T^*M_1$,
$U_2\subset T^*M_2$ are open subsets of the cotangent bundles. The flipped graph of $\varkappa$
is the Lagrangian submanifold of the product of the cotangent bundle (or the cotangent bundle of the product)~$T^*M_1\times T^*M_2=T^*(M_1\times M_2)$ defined by
$$
\Graph(\varkappa):=\{(x,\xi,y,-\eta)\mid (y,\eta)\in U_2,\ \varkappa(y,\eta)=(x,\xi)\}.
$$
Denote by $I^{\comp}_h(\varkappa)$ the class of compactly microlocalized semiclassical Fourier integral operators
associated to $\varkappa$. Each element of $I^{\comp}_h(\varkappa)$ is an $h$-dependent
family of compactly supported smoothing operators $B=B(h):\mathcal D'(M_2)\to \CIc(M_1)$
such that the corresponding Schwartz kernels are Lagrangian distributions in $h^{-{d\over 2}}I^{\comp}_h(\Graph(\varkappa))$.

An important special case is when $M_2=\mathbb R^d$ and the graph of $\varkappa$ projects diffeomorphically onto the $(x,\eta)$ variables, which (given that $\varkappa$ is a symplectomorphism
and assuming that its domain is simply connected) means that $\varkappa$ is given by a generating function:
\begin{equation}
  \label{e:kappa-generating}
\varkappa(y,\eta)=(x,\xi)\ \iff\ (x,\eta)\in U,\ \xi=\partial_x S(x,\eta),\ y=\partial_\eta S(x,\eta),
\end{equation}
where $U\subset M_1\times\mathbb R^d$ is an open set and $S\in C^\infty(U;\mathbb R)$.
Then elements of $I^{\comp}_h(\varkappa)$ have the following form,
modulo the class $\mathcal O(h^\infty)_{\Psi^{-\infty}}$ introduced in~\S\ref{s:review-semi}:
\begin{equation}
  \label{e:FIO-generating}
B(h)f(x)=(2\pi h)^{-d}\int_{\mathbb R^{2d}}e^{{i\over h}(S(x,\eta)-\langle y,\eta\rangle)}
b(x,\eta,y;h)f(y)\,dyd\eta
\end{equation}
where the amplitude $b\in \CIc(U\times\mathbb R^d)$ is supported in an $h$-independent compact
set and has all the derivatives bounded uniformly in~$h$. Here $\langle y,\eta\rangle=\sum_{j=1}^d y_j\eta_j$ denotes the Euclidean inner product.

Here are some standard properties of Lagrangian distributions and Fourier integral operators:
\begin{enumerate}
\item every element of $I^{\comp}_h(\varkappa)$ is bounded in $L^2(M_2)\to L^2(M_1)$ norm
uniformly in~$h$;
\item if $B\in I^{\comp}_h(\varkappa)$, then the adjoint operator $B^*$ lies in $I^{\comp}_h(\varkappa^{-1})$;
\item if $\varkappa:T^*M\to T^*M$ is the identity map, then $I^{\comp}_h(\varkappa)$ equals
the pseudodifferential class $\Psi^{\comp}_h(M)$ introduced in~\S\ref{s:review-semi};
\item if $\Lambda\subset T^*M_2$ is a Lagrangian submanifold, $\varkappa:U_2\to U_1$ is a symplectomorphism with $U_j\subset T^*M_j$, and $u\in I^{\comp}_h(\Lambda)$, $B\in I^{\comp}_h(\varkappa)$, then $Bu\in I^{\comp}_h(\varkappa(\Lambda))$, where $\varkappa(\Lambda)\subset T^*M_1$
is a Lagrangian submanifold;
\item if $\varkappa_1:U_2\to U_1,\varkappa_2:U_3\to U_2$ are symplectomorphisms with $U_j\subset T^*M_j$,
and $B_1\in I^{\comp}_h(\varkappa_1)$, $B_2\in I^{\comp}_h(\varkappa_2)$, then the composition
$B_1B_2$ is a Fourier integral operator in $I^{\comp}_h(\varkappa_1\circ\varkappa_2)$.
\end{enumerate}

We finally discuss microlocal conjugation by Fourier integral operators.
Let $\varkappa:U_2\to U_1$ be a symplectomorphism and $K_1\subset U_1$, $K_2\subset U_2$
be two compact sets with $\varkappa(K_2)=K_1$. We say a pair of Fourier integral operators
$B\in I^{\comp}_h(\varkappa)$, $B'\in I^{\comp}_h(\varkappa^{-1})$ \emph{quantizes
$\varkappa$ near $K_1\times K_2$}, if the pseudodifferential operators $BB'\in\Psi^{\comp}_h(M_1)$ and $B'B\in\Psi^{\comp}_h(M_2)$ satisfy
(where $\WFh(\bullet)$ was defined in~\eqref{e:WF-A})
\begin{equation}
  \label{e:FIO-quantize}
\WFh(I-BB')\cap K_1=\emptyset,\quad
\WFh(I-B'B)\cap K_2=\emptyset.
\end{equation}
Such operators always exist locally: if $\varkappa(y_0,\eta_0)=(x_0,\xi_0)$, then
there exist $B,B'$ quantizing $\varkappa$ near $\{(x_0,\xi_0)\}\times \{(y_0,\eta_0)\}$.

\subsubsection{More on the calculus associated to a Lagrangian foliation}

We now revisit the calculus associated to a Lagrangian foliation introduced in~\S\ref{s:calculus-lagrangian}, showing some of its technical properties used later in the proof.
Recall from that section and~\cite[Appendix~A]{meassupp} that if $M$ is a manifold,
$L$ is a Lagrangian foliation on an open subset $U\subset T^*M$, and the constants $\rho,\rho'$ satisfy~\eqref{e:rho-properties},
then for each $a\in S^{\comp}_{L,\rho,\rho'}(U)$ we can define
the quantization $\Op_h^L(a):L^2(M)\to L^2(M)$.

We first consider the model cases when $M=\mathbb R^d$, $U=T^*\mathbb R^d$, and 
$L\in\{L_V,L_H\}$ where $L_V$ is the vertical and $L_H$ the horizontal foliation:
\begin{align}
  \label{e:L-V-def}
  L_V&=\Span(\partial_{\eta_1},\dots,\partial_{\eta_d})=\ker(dy),\\
    \label{e:L-H-def}
L_H&=\Span(\partial_{y_1},\dots,\partial_{y_d})=\ker(d\eta).
\end{align}
Symbols $a\in S^{\comp}_{L_V,\rho,\rho'}(T^*\mathbb R^d)$ satisfy the derivative bounds
\begin{equation}
  \label{e:derb-LV}
\sup_{y,\eta}|\partial^\alpha_y\partial^\beta_\eta a(y,\eta;h)|\leq C_{\alpha\beta}h^{-\rho|\alpha|-\rho'|\beta|}
\end{equation}
and symbols $a\in S^{\comp}_{L_H,\rho,\rho'}(T^*\mathbb R^d)$ satisfy the bounds
\begin{equation}
  \label{e:derb-LH}
\sup_{y,\eta}|\partial^\alpha_y\partial^\beta_\eta a(y,\eta;h)|\leq C_{\alpha\beta}h^{-\rho|\beta|-\rho'|\alpha|}.
\end{equation}
For $0\leq s\leq 1$, define the following quantization procedure on~$\mathbb R^d$
(see~\cite[\S4.1.1]{Zworski-Book}):
\begin{equation}
  \label{e:Op-h-def}
\Op_h^{(s)}(a)f(y)=(2\pi h)^{-d}\int_{\mathbb R^{2d}}e^{{i\over h}\langle y-y',\eta\rangle}a\big(sy+(1-s)y',\eta\big)f(y')\,dy'd\eta.
\end{equation}
The case $s=1$ is called the \emph{standard}, or \emph{left} quantization; the case $s=0$ is the \emph{right} quantization and the case $s=\frac12$ is the \emph{Weyl} quantization.

In~\cite{hgap,meassupp} one used symbols of the class $S^{\comp}_{L_V,\rho,\rho'}(T^*\mathbb R^d)$ and
the standard quantization $\Op_h^{(1)}$, because it was easier to prove invariance
of this quantization under Fourier integral operators preserving the foliation;
see~\cite[Lemmas~3.9--3.10]{hgap}. The next few lemmas will show that in fact
one could use either $L_H$ or $L_V$ and any of the quantizations $\Op_h^{(s)}$.
For our purposes it is enough to consider the principal part of the operators,
allowing an $\mathcal O(h^{1-\rho-\rho'})_{L^2(\mathbb R^d)\to L^2(\mathbb R^d)}$ remainder.

We start with a change of quantization statement:
\begin{lemm}
  \label{l:change-of-quantization}
Let $L\in \{L_V,L_H\}$.
Assume that $a\in S^{\comp}_{L,\rho,\rho'}(T^*\mathbb R^d)$ and fix $s,s'\in [0,1]$.
Then we have
\begin{equation}
  \label{e:coq}
\Op_h^{(s')}(a)=\Op_h^{(s)}(a)+\mathcal O(h^{1-\rho-\rho'})_{L^2(\mathbb R^d)\to L^2(\mathbb R^d)}.
\end{equation}
\end{lemm}
\begin{proof}
We first consider the case when $L=L_V$.
By the change of quantization formula~\cite[Theorem~4.13]{Zworski-Book} we have
$$
\Op_h^{(s')}(a)=\Op_h^{(s)}(\check a)\quad\text{where }
\check a:=e^{i(s'-s)h\langle \partial_y,\partial_\eta\rangle}a.
$$
The symbol $\check a$ has a semiclassical expansion (in a sense made
precise in a moment):
\begin{equation}
  \label{e:ch-semi-1}
\check a\sim \sum_{k=0}^\infty {h^k\over k!}i^k(s'-s)^k\langle\partial_y,\partial_\eta\rangle^ka
\end{equation}
where $\langle\partial_y,\partial_\eta\rangle:=\sum_{j=1}^d\partial_{y_j}\partial_{\eta_j}$
is a second order differential operator.

By~\eqref{e:derb-LV}
the $k$-th term in~\eqref{e:ch-semi-1} is $\mathcal O(h^{(1-\rho-\rho')k})_{S_{L_V,\rho,\rho'}(T^*\mathbb R^{d})}$.
Here $S_{L_V,\rho,\rho'}$ denotes symbols satisfying the estimates~\eqref{e:derb-LV}
which are not necessarily compactly supported.
The expansion~\eqref{e:ch-semi-1} holds in the following sense: for each $N$
\begin{equation}
  \label{e:ch-semi-2}
\check a-\sum_{k=0}^{N-1}{h^k\over k!}i^k(s'-s)^k\langle\partial_y,\partial_\eta\rangle^ka
=\mathcal O(h^{(1-\rho-\rho')N})_{S_{L_V,\rho,\rho'}(T^*\mathbb R^{d})}.
\end{equation}
To show~\eqref{e:ch-semi-2}, we follow~\cite[\S A.2]{meassupp} and consider the rescaling map
$$
\Lambda:T^*\mathbb R^d\to T^*\mathbb R^d,\quad
\Lambda(y,\eta)=(h^{\rho-\rho'\over 2}y,h^{\rho'-\rho\over 2}\eta).
$$
Then $a\in S_{L_V,\rho,\rho'}(T^*\mathbb R^d)$ if and only the pullback $b:=\Lambda^*a$ lies
in the class $S_\delta(T^*\mathbb R^d)$ of symbols satisfying
$$
\sup_{y,\eta}|\partial^\alpha_y\partial^\beta_\eta b(y,\eta;h)|\leq C_{\alpha\beta}h^{-\delta(|\alpha|+|\beta|)},
$$
with $\delta:=\frac12(\rho+\rho')\in [0,\frac12)$. We have $e^{i(s'-s)h\langle \partial_y,\partial_\eta\rangle}a=(\Lambda^*)^{-1}e^{i(s'-s)h\langle \partial_y,\partial_\eta\rangle}\Lambda^* a$,
so~\eqref{e:ch-semi-2} follows from the same expansion in the class $S_\delta$ given
in~\cite[Theorem~4.17]{Zworski-Book}.

Now, putting $N=1$ in~\eqref{e:ch-semi-2} we get $\check a=a+h^{1-\rho-\rho'}b$
where $b=\mathcal O(1)_{S_{L_V,\rho,\rho'}(T^*\mathbb R^d)}$;
then $\Op_h^{(s')}(a)=\Op_h^{(s)}(a)+h^{1-\rho-\rho'}\Op_h^{(s)}(b)$.
We have $\|\Op_h^{(s)}(b)\|_{L^2(\mathbb R^d)\to L^2(\mathbb R^d)}=\mathcal O(1)$
as follows from a rescaling argument and the $L^2$ boundedness
for symbols in $S_\delta$ similarly to~\cite[\S A.2]{meassupp}. This finishes the proof in the case $L=L_V$.

The case $L=L_H$ is handled exactly the same, except the rescaling map $\Lambda$
needs to be replaced by~$\Lambda^{-1}$.
\end{proof}
Now, consider the general calculus associated to a Lagrangian foliation
$L$ on $U\subset T^*M$. We show the following lemma regarding
operators of the form $\Op_h^L(a)$ conjugated by semiclassical Fourier integral
operators sending $L$ to $L_H$; it is used in Lemma~\ref{l:tilde-A-localized} below.
The proof relies on the version of this lemma with $L_H$ replaced by $L_V$ shown in~\cite{hgap,meassupp}, as well as on equivariance of the Weyl quantization
under the Fourier transform and on the previous lemma to change to the Weyl quantization.
\begin{lemm}
  \label{l:L-H-straighten}
Assume that $a\in S^{\comp}_{L,\rho,\rho'}(U)$ is supported inside some $h$-independent compact set $K\subset U$, $\varkappa:U\to T^*\mathbb R^d$ is a symplectomorphism satisfying $\varkappa_*L=L_H$, and $B\in I^{\comp}_h(\varkappa)$, $B'\in I^{\comp}_h(\varkappa^{-1})$ quantize $\varkappa$
near $\varkappa(K)\times K$ in the sense of~\eqref{e:FIO-quantize}. Fix $s\in [0,1]$. Then
\begin{equation}
  \label{e:LHS-0}
\Op_h^L(a)=B'\Op_h^{(s)}(a\circ \varkappa^{-1})B+\mathcal O(h^{1-\rho-\rho'})_{L^2(M)\to L^2(M)}.
\end{equation}
\end{lemm}
\begin{proof}
1. Denote by~$\mathcal F_h:L^2(\mathbb R^d)\to L^2(\mathbb R^d)$
the unitary semiclassical Fourier transform, defined similarly to~\eqref{e:F-h-def}:
\begin{equation}
  \label{e:F-h-def-2}
\mathcal F_hf(\eta)=(2\pi h)^{-\frac d2}\int_{\mathbb R^d} e^{-\frac ih\langle y,\eta\rangle}f(y)\,dy.
\end{equation}
By~\eqref{e:FIO-quantize}, the fact that $\supp a\subset K$, and the nonintersecting
support property~\eqref{e:extra-nonint-supp}, we have
\begin{equation}
  \label{e:LHSS-X}
\begin{aligned}
\Op_h^L(a)&=B'B\Op_h^L(a)B'B+\mathcal O(h^\infty)_{L^2(M)\to L^2(M)}\\
&=B'\mathcal F_h^{-1}\widetilde A\mathcal F_hB+\mathcal O(h^\infty)_{L^2(M)\to L^2(M)}\\
\quad\text{where }\widetilde A&=\mathcal F_hB\Op_h^L(a)B'\mathcal F_h^{-1}:L^2(\mathbb R^d)\to L^2(\mathbb R^d).
\end{aligned}
\end{equation}

\noindent 2. For any $Z\in\Psi^{\comp}_h(\mathbb R^d)$, the composition $\mathcal F_hZ$
lies in $I^{\comp}_h(\varkappa_F)+\mathcal O(h^\infty)_{L^2(\mathbb R^d)\to L^2(\mathbb R^d)}$, where
$$
\varkappa_F:T^*\mathbb R^d\to T^*\mathbb R^d,\quad
\varkappa_F(y,\eta)=(\eta,-y);
$$
a similar statement is true for $Z\mathcal F_h^{-1}$
and the map $\varkappa_F^{-1}$.
Therefore by the composition property~(5) in~\S\ref{s:fios}
$$
\begin{aligned}
\mathcal F_hB&\in I^{\comp}_h(\varkappa_F\circ\varkappa)+\mathcal O(h^\infty)_{L^2(M)\to L^2(\mathbb R^d)},\\
B'\mathcal F_h^{-1}&\in I^{\comp}_h(\varkappa^{-1}\circ\varkappa_F^{-1})+\mathcal O(h^\infty)_{L^2(\mathbb R^d)\to L^2(M)}.
\end{aligned}
$$
Since $\varkappa_F$ interchanges the foliations $L_H$ and $L_V$,
we have
$$
(\varkappa_F\circ\varkappa)_*L=L_V.
$$
Note that $a\circ\varkappa^{-1}\in S^{\comp}_{L_H,\rho,\rho'}(T^*\mathbb R^d)$ and
$a\circ\varkappa^{-1}\circ\varkappa_F^{-1}\in S^{\comp}_{L_V,\rho,\rho'}(T^*\mathbb R^d)$.

We now apply~\cite[(A.20)]{meassupp} with the symplectomorphism $\varkappa_F\circ\varkappa:U\to T^*\mathbb R^d$ and the operators $\mathcal F_hB$, $B'\mathcal F_h^{-1}$
to write the operator $\widetilde A$ in terms of the standard quantization
(here we use that the operator $\mathcal F_hBB'\mathcal F_h^{-1}\in\Psi^{\comp}_h(\mathbb R^d)$
has principal symbol equal to~1 near $\varkappa_F(\varkappa(\supp a))$):
\begin{equation}
  \label{e:LHSS-0}
\widetilde A=\Op_h^{(1)}(a\circ\varkappa^{-1}\circ\varkappa_F^{-1})+\mathcal O(h^{1-\rho-\rho'})_{L^2(\mathbb R^d)\to L^2(\mathbb R^d)}.
\end{equation}
By Lemma~\ref{l:change-of-quantization}, we can replace the standard
quantization by the Weyl quantization:
\begin{equation}
  \label{e:LHSS-1}
\widetilde A=\Op_h^{(1/2)}(a\circ\varkappa^{-1}\circ\varkappa_F^{-1})+\mathcal O(h^{1-\rho-\rho'})_{L^2(\mathbb R^d)\to L^2(\mathbb R^d)}.
\end{equation}

\noindent 3. By~\cite[Theorem~4.9]{Zworski-Book}, we next have
\begin{equation}
  \label{e:LHSS-2}
\mathcal F_h^{-1}\Op_h^{(1/2)}(a\circ\varkappa^{-1}\circ\varkappa_F^{-1})\mathcal F_h=\Op_h^{(1/2)}(a\circ\varkappa^{-1}).
\end{equation}
Applying Lemma~\ref{l:change-of-quantization} again, we also have
\begin{equation}
  \label{e:LHSS-3}
\Op_h^{(1/2)}(a\circ\varkappa^{-1})=\Op_h^{(s)}(a\circ\varkappa^{-1})+\mathcal O(h^{1-\rho-\rho'})_{L^2(\mathbb R^d)\to L^2(\mathbb R^d)}.
\end{equation}
Combining~\eqref{e:LHSS-X} and \eqref{e:LHSS-1}--\eqref{e:LHSS-3}, we get~\eqref{e:LHS-0},
which finishes the proof.
\end{proof}

\subsubsection{Localization of Lagrangian states}
\label{s:lag-loc}

We next show two technical lemmas. As in~\eqref{e:our-rho} before
we fix $\rho={2\over 3}(1-\varepsilon_0)$. Similarly to~\eqref{e:fourier-mul}
for any measurable set $X\subset\mathbb R^d$ we define the Fourier multiplier
$\indic_{X}(hD_x)$ on $L^2(\mathbb R^d)$ by the formula (where $\mathcal F$ denotes
the Fourier transform)
\begin{equation}
  \label{e:fourier-mul-2}
\mathcal F(\indic_X(hD_x)f)(\xi)=\indic_X(h\xi)\widehat f(\xi)\quad\text{for all }
f\in L^2(\mathbb R^d),\
\xi\in\mathbb R^d.
\end{equation}
\begin{lemm}
\label{l:lag-loc}
Consider the function depending on the parameter $h\in (0,1]$
$$
w(x):=e^{{i\over h}\Phi(x)}b(x),\quad
x\in\mathbb R^d
$$
where the phase function $\Phi\in C^\infty(B(0,1);\mathbb R)$
and the amplitude $b\in \CIc(B(0,1))$ satisfy
for some constants $\widetilde C_0,\widetilde C_1,\widetilde C_2,\dots$ and all multiindices $\alpha$ and points $x$
\begin{align}
  \label{e:lag-loc-1}
\|\partial_x^2\Phi(0)\|&\leq \widetilde C_2h^{\rho\over 2},\\
  \label{e:lag-loc-2}
|\partial^\alpha_x\Phi(x)|&\leq \widetilde C_{|\alpha|},\\
  \label{e:lag-loc-3}
\supp b&\subset B(0,\widetilde C_0h^{\rho\over 2}),\\
  \label{e:lag-loc-4}
|\partial^\alpha_x b(x)|&\leq \widetilde C_{|\alpha|}h^{-{\rho\over 2}|\alpha|}.
\end{align}
Then we have for each $N>{d\over 2}$
\begin{equation}
  \label{e:lag-loc-conclusion}
\|\indic_{\mathbb R^d\setminus B(\partial_x\Phi(0),C_3h^{\rho})}(hD_x)w\|_{L^2(\mathbb R^d)}
\leq C_{N+1} h^{-{d\over 2}+\varepsilon_0 N}
\end{equation}
for some constants $C_{L}$ depending only on $d,\varepsilon_0$, and the constants $\widetilde C_0,\widetilde C_1,\dots,\widetilde C_L$.
\end{lemm}
\Remarks 1. Since $\varepsilon_0>0$ is fixed and $N$ can be arbitrarily large, 
the left-hand side of~\eqref{e:lag-loc-conclusion} is $\mathcal O(h^\infty)$
as long as we control all the constants $\widetilde C_0,\widetilde C_1,\widetilde C_2,\dots$.

\noindent 2.
The function $w$ is a semiclassical Lagrangian distribution associated
to the graph
\begin{equation}
  \label{e:lag-graph}
\{(x,\xi)\mid x\in B(0,\widetilde C_0 h^{\rho\over 2}),\ \xi=\partial_x\Phi(x)\}.
\end{equation}
Under the conditions~\eqref{e:lag-loc-1}--\eqref{e:lag-loc-3} the projection
of the graph~\eqref{e:lag-graph} onto the frequency variables $\xi$ is contained
in the ball $B(\partial_x\Phi(0),Ch^{\rho})$ for sufficiently large $C$
(this graph is `almost horizontal'; see~\eqref{e:lagl-4.5} below).
The statement~\eqref{e:lag-loc-conclusion} says that $w$ is localized in frequency
to such a ball. This is natural because one expects $w$ to be microlocalized
near the graph~\eqref{e:lag-graph}. However, because we study
fine localization on the scale $\sim h^\rho\ll h^{\frac 12}$, one needs to exercise care. 

\noindent 3. A different version of localization
of Lagrangian distributions in frequency was proved in~\cite[Proposition~2.7]{varfup}.
We cannot use this version in the present paper because the symbol~$b$
has derivatives growing as $h\to 0$.
\begin{proof}
Throughout the proof we use the notation $C_L$ for a constant depending
only on~$d,\widetilde C_0,\dots,\widetilde C_L$, whose precise value might change from place to place.

1. We show the following stronger estimate, from which~\eqref{e:lag-loc-conclusion}
follows using unitarity of the Fourier transform:
\begin{equation}
\label{e:lagl-1}
|\widehat w(\xi/h)|\leq C_{N+1}h^{\varepsilon_0 N}\langle\xi\rangle^{-N}\quad\text{for all }
\xi\in\mathbb R^d\setminus B(\partial_x\Phi(0),C_3h^\rho).
\end{equation}
Take arbitrary $\xi\in\mathbb R^d\setminus B(\partial_x\Phi(0),C_3h^\rho)$ and write
\begin{equation}
\label{e:lagl-2}
\widehat w(\xi/h)=\int_{\mathbb R^d}e^{{i\over h}\widetilde\Phi(x)}b(x)\,dx\quad\text{where }\widetilde\Phi(x):=\Phi(x)-\langle x,\xi\rangle.
\end{equation}
We integrate by parts in~\eqref{e:lagl-2} using the first order partial differential operator
$L$ defined by
$$
Lf(x):=\sum_{j=1}^dc_j(x)\partial_{x_j}f(x),\quad
c_j(x):=-i{\partial_{x_j}\widetilde\Phi(x)\over |\partial_x\widetilde\Phi(x)|^2}.
$$
We have $e^{{i\over h}\widetilde\Phi(x)}=hL(e^{{i\over h}\widetilde\Phi(x)})$, thus integrating by parts~$N$ times and using~\eqref{e:lag-loc-3} gives
\begin{equation}
\label{e:lagl-3}
\begin{aligned}
|\widehat w(\xi/h)|&=h^{N}\bigg|\int_{\mathbb R^d}e^{{i\over h}\widetilde\Phi(x)}(L^t)^{N}b(x)\,dx\bigg|\\
&\leq C_0h^{N}\sup_{x\in B\left(0,\widetilde C_0h^{\rho\over 2}\right)}|(L^t)^{N}b(x)|
\end{aligned}
\end{equation}
where the transpose operator $L^t$ is given by 
$$
L^tf(x)=-\sum_{j=1}^d\partial_{x_j}(c_j(x)f(x)).
$$

\noindent 2. We now estimate the derivatives of the coefficients $c_j(x)$
on the ball $B\big(0,\widetilde C_0h^{\rho\over 2}\big)$.
We start with a lower bound on the length of $\partial_x\widetilde\Phi(x)=\partial_x\Phi(x)-\xi$.
By~\eqref{e:lag-loc-1}--\eqref{e:lag-loc-2} we have
\begin{equation}
  \label{e:lagl-4}
\sup_{x\in B\left(0,\widetilde C_0h^{\rho\over 2}\right)}\|\partial_x^2\Phi(x)\|\leq C_3h^{\rho\over 2}.
\end{equation}
This implies
\begin{equation}
  \label{e:lagl-4.5}
\sup_{x\in B\left(0,\widetilde C_0h^{\rho\over 2}\right)}|\partial_x\Phi(x)-\partial_x\Phi(0)|\leq \tfrac12 C_3 h^\rho.
\end{equation}
Fix $C_3\geq 2$ so that~\eqref{e:lagl-4.5} holds.
Since $\xi\notin B(\partial_x\Phi(0),C_3 h^\rho)$, we get
\begin{equation}
  \label{e:lagl-5}
\inf_{x\in B\left(0,\widetilde C_0h^{\rho\over 2}\right)}|\partial_x\widetilde\Phi(x)|\geq h^\rho.
\end{equation}
Next, arguing by induction we see that for each multiindex $\alpha$, the derivative $\partial^\alpha_x c_j(x)$ is a linear combination
with constant coefficients of terms of the form
$$
{\partial^{\alpha_1}_x\widetilde\Phi(x)\cdots\partial_x^{\alpha_{2m-1}}\widetilde \Phi(x)\over |\partial_x\widetilde\Phi(x)|^{2m}}
$$
where $1\leq m\leq |\alpha|+1$, $|\alpha_1|,\dots,|\alpha_{2m-1}|\geq 1$, and
$|\alpha_1|+\cdots+|\alpha_{2m-1}|=|\alpha|+2m-1$.

We have for each $k=1,\dots,2m-1$
$$
|\partial_x^{\alpha_k}\widetilde\Phi(x)|\leq C_{\max(|\alpha_k|,3)}h^{-\frac\rho 2(|\alpha_k|-1)} |\partial_x\widetilde\Phi(x)|\quad\text{for all }x\in B\big(0,\widetilde C_0h^{\rho\over 2}\big).
$$
Indeed, for $|\alpha_k|=1$ this is immediate, for $|\alpha_k|=2$ it follows from~\eqref{e:lagl-4} and~\eqref{e:lagl-5}, and for $|\alpha_k|\geq 3$ it follows from~\eqref{e:lag-loc-2} and~\eqref{e:lagl-5}.

It now follows that for all $\alpha$
\begin{equation}
\label{e:lagl-6}
|\partial_x^\alpha c_j(x)|\leq C_{\max(|\alpha|,2)+1}h^{-{\rho\over 2}|\alpha|}|\partial_x\widetilde\Phi(x)|^{-1}\quad\text{for all }x\in B\big(0,\widetilde C_0h^{\rho\over 2}\big).
\end{equation}

\noindent 3. The function $(L^t)^Nb(x)$ is a linear combination with constant coefficients
of expressions of the form
$$
\partial_x^{\alpha_1}c_{j_1}(x)\cdots \partial_x^{\alpha_N}c_{j_N}(x)\partial_x^{\alpha_0}b(x)
$$
where $|\alpha_0|+|\alpha_1|+\cdots+|\alpha_N|=N$. By~\eqref{e:lag-loc-4} and~\eqref{e:lagl-6}
we have
$$
|(L^t)^Nb(x)|\leq C_{N+1}h^{-{\rho\over 2}N}|\partial_x\widetilde\Phi(x)|^{-N}\quad\text{for all }x\in B\big(0,\widetilde C_0h^{\rho\over 2}\big).
$$
Then~\eqref{e:lagl-3} and~\eqref{e:lagl-5} imply that
$$
|\widehat w(\xi/h)|\leq C_{N+1}h^{(1-{3\rho\over 2})N}=C_{N+1}h^{\varepsilon_0 N}.
$$
This shows~\eqref{e:lagl-1} when $|\xi|$ is bounded. On the other hand, if $|\xi|$ is large enough,
then the bound~\eqref{e:lagl-5} can be improved to
$$
|\partial_x\widetilde\Phi(x)|\geq {|\xi|\over 2}\quad\text{for all }x\in B(0,1)
$$
and we get
$$
|\widehat w(\xi/h)|\leq C_{N+1} h^{(1-\frac\rho 2)N}|\xi|^{-N}\leq C_{N+1}h^{{2\over 3}N}|\xi|^{-N}
$$
which again gives~\eqref{e:lagl-1}.
\end{proof}
A consequence of Lemma~\ref{l:lag-loc} and the general calculus of Fourier
integral operators is the following statement used in the proof of Lemma~\ref{l:tilde-A-localized} below.
Recall the horizontal Lagrangian foliation $L_H$ on $T^*\mathbb R^d$
defined in~\eqref{e:L-H-def}.
\begin{lemm}
  \label{l:fio-lag-loc}
Assume that $\varkappa:U_2\to U_1$ is a symplectomorphism,
where $U_1,U_2\subset T^*\mathbb R^d$ are open subsets containing the origin,
and
\begin{equation}
\label{e:fio-lag-loc-ass-1}
\varkappa(0)=0,\quad
d\varkappa(0)L_H=L_H. 
\end{equation}
Assume moreover that the frequency $\eta^0\in\mathbb R^d$ and the amplitude $b\in\CIc(\mathbb R^d)$
satisfy for some constants $\widetilde C_0,\widetilde C_1,\dots$
and all multiindices $\alpha$ and points $y$
\begin{align}
\label{e:fio-lag-loc-ass-2}
|\eta^0|&\leq \widetilde C_0h^{\rho\over 2},\\
\label{e:fio-lag-loc-ass-3}
\supp b&\subset B\left(0,\widetilde C_0h^{\rho\over 2}\right),\\
\label{e:fio-lag-loc-ass-4}
|\partial_y^\alpha b(y)|&\leq \widetilde C_{|\alpha|}h^{-{\rho\over 2}|\alpha|}.
\end{align}
Let $B\in I^{\comp}_h(\varkappa)$ and define
$$
v(y):=e^{{i\over h}\langle y,\eta^0\rangle}b(y),\quad
w:=Bv.
$$
Take arbitrary $y^0\in B\left(0,\widetilde C_0h^{\rho\over 2}\right)$ and denote $(x^0,\xi^0):=\varkappa(y^0,\eta^0)$.
Then we have for each~$N$
\begin{equation}
  \label{e:fio-lag-loc}
\|\indic_{\mathbb R^d\setminus B\left(\xi^0,C_0h^\rho\right)}(hD_x)w\|_{L^2(\mathbb R^d)}
\leq C_N h^N.
\end{equation}
Here the constant $C_N$ depends only on the constants
$\widetilde C_0,\widetilde C_1,\dots,\widetilde C_L$ for some $L$ depending only on~$N,d,\varepsilon_0$
and also on some $(N,d,\varepsilon_0)$-dependent $C^\infty$-seminorms of $\varkappa,\varkappa^{-1}$ and $I^{\comp}_h(\varkappa)$-seminorm of~$B$.
\end{lemm}
\Remark The function $v$ is a semiclassical Lagrangian distribution associated to
the horizontal leaf
\begin{equation}
  \label{e:hor-leaf}
\Lambda_{\eta^0}:=\left\{(y,\eta^0)\mid y\in B\left(0,\widetilde C_0h^{\rho\over 2}\right)\right\}.
\end{equation}
By property~(4) in~\S\ref{s:fios}, we expect that $w$
is a semiclassical Lagrangian distribution associated to $\varkappa(\Lambda_{\eta^0})$.
By~\eqref{e:fio-lag-loc-ass-1}--\eqref{e:fio-lag-loc-ass-2},
the projection of $\varkappa(\Lambda_{\eta^0})$ onto the frequency variables $\xi$
lies in an $\sim h^\rho$-sized ball centered at $\xi^0$, giving an informal justification for~\eqref{e:fio-lag-loc}; see Figure~\ref{f:curved-projection}.
However, just like in Lemma~\ref{l:lag-loc} the symbol $b$ has derivatives growing
with~$h$ and we need localization on the fine scale $h^\rho$, so one has
to work out the details carefully.
\begin{figure}
\includegraphics{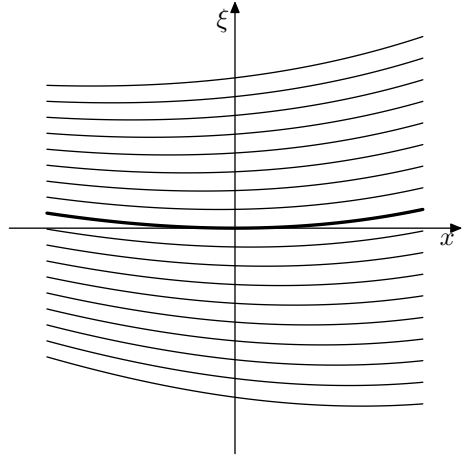}
\caption{The Lagrangians $\varkappa(\Lambda_{\eta^0})$ for different
values of $\eta^0\in B(0,h^{\rho\over 2})$ where $\Lambda_{\eta^0}$ is the horizontal Lagrangian defined in~\eqref{e:hor-leaf} and $\varkappa$ satisfies~\eqref{e:fio-lag-loc-ass-1}, drawn at scale $\sim h^{\rho\over 2}$.
The thicker curve is $\varkappa(\Lambda_{0})$, which has horizontal tangent space at the origin.
The projection of each of the Lagrangians onto the $\xi$ direction
lies in a ball of radius $\sim h^{\rho}$.}
\label{f:curved-projection}
\end{figure}
\begin{proof}
Throughout the proof we denote by $C_N$ some constant depending only
on the constants $\widetilde C_0,\widetilde C_1,\dots,\widetilde C_L$ for some $L$ depending only on~$N,d,\varepsilon_0$ and also on some $(N,d,\varepsilon_0)$-dependent $C^\infty$-seminorms of $\varkappa,\varkappa^{-1}$ and $I^{\comp}_h(\varkappa)$-seminorm of~$B$; the precise
value of $C_N$ might change from place to place.

\noindent 1. By~\eqref{e:fio-lag-loc-ass-2}--\eqref{e:fio-lag-loc-ass-4},
$v$ is microlocalized at the origin $(0,0)\in T^*\mathbb R^d$ in the sense
that $Av=\mathcal O(h^\infty)_{C^\infty}$ for all $A\in \Psi_h^{\comp}(\mathbb R^d)$
such that $\WFh(A)\cap \{(0,0)\}=\emptyset$. Therefore, we may shrink
$U_1,U_2$ to be contained in an arbitrarily small $h$-independent ball centered
at the origin.

By~\eqref{e:fio-lag-loc-ass-1}
the graph of $\varkappa$ passes through $(0,0,0,0)$ and its tangent
space at this point projects isomorphically onto the $(x,\eta)$ variables.
Thus after shrinking $U_1,U_2$ we may assume that the graph of $\varkappa$
projects diffeomorphically onto the $(x,\eta)$ variables and thus has
the form~\eqref{e:kappa-generating} for some generating function $S(x,\eta)$.
Then $B$ has the form~\eqref{e:FIO-generating}:
$$
Bf(x)=(2\pi h)^{-d}\int_{\mathbb R^{2d}}e^{{i\over h}\left(S(x,\eta)-\left\langle y,\eta\right\rangle\right)}q(x,\eta,y;h)f(y)\,dyd\eta
$$
where the symbol $q$ has each derivative bounded uniformly in~$h$.
Our constants $C_N$ are allowed to depend on the $C^\infty$-seminorms of $S$
and~$q$.
Moreover, \eqref{e:fio-lag-loc-ass-1} implies that
\begin{equation}
  \label{e:filla-1}
\partial_xS(0,0)=\partial_\eta S(0,0)=0,\quad
\partial_x^2S(0,0)=0.
\end{equation}

\noindent 2. We now write
$$
w(x)=(2\pi h)^{-d}\int_{\mathbb R^{2d}}e^{{i\over h}(S(x,\eta)+\langle y,\eta^0-\eta\rangle)}
q(x,\eta,y;h)b(y)\,dyd\eta.
$$
Applying the method of stationary phase (similarly to the standard proof of property~(4)
from~\S\ref{s:fios}; for the statement of the method of stationary phase see for example~\cite[Theorem~7.7.5]{Hormander1} and~\cite[Theorem~3.16]{Zworski-Book}), we get
$$
w(x)=e^{{i\over h}S(x,\eta^0)}\tilde b(x)+\mathcal O(h^\infty)_{\CIc(\mathbb R^d)}.
$$
Here the amplitude $\tilde b\in\CIc(\mathbb R^d)$ has an asymptotic expansion
in powers of~$h$: the $k$-th term in the expansion for $\tilde b(x)$
is equal to $h^k$ times some order $2k$ differential operator applied to $q(x,\eta,y;h)b(y)$
at the stationary point $y=\partial_\eta S(x,\eta^0),\eta=\eta^0$.
Note that by~\eqref{e:fio-lag-loc-ass-4} this term
is $\mathcal O(h^{(1-\rho)k})$ and the stationary phase expansion
still applies with $h$-dependent symbols since $\rho<1$.
We moreover get the derivative bounds
\begin{equation}
|\partial^\alpha_x \tilde b(x)|\leq C_{|\alpha|} h^{-{\rho\over 2}|\alpha|}
\end{equation}
and (by~\eqref{e:fio-lag-loc-ass-2}--\eqref{e:fio-lag-loc-ass-3} and~\eqref{e:filla-1}) the support property
$$
\supp\tilde b\subset B\left(0,C_0h^{\rho\over 2}\right).
$$

\noindent 3. We have $\xi^0=\partial_xS(x^0,\eta^0)$ and thus by~\eqref{e:filla-1}
and a Taylor expansion for $\partial_xS$ at~$(0,0)$
$$
|\partial_x S(0,\eta^0)-\xi^0|\leq C_0h^{\rho}.
$$
Now~\eqref{e:fio-lag-loc} follows from Lemma~\ref{l:lag-loc} with $\Phi(x):=S(x,\eta^0)$,
where the property~\eqref{e:lag-loc-1} follows from~\eqref{e:fio-lag-loc-ass-2}
and~\eqref{e:filla-1}.
\end{proof}

\subsection{End of the proof}
\label{s:end-proof}

In this section we give the proof of~\eqref{e:key-estimate-4}. Fix $k$ and let the point $q_k\in T^*M$ and the symplectomorphism $\varkappa_k$ be
as in~\S\ref{s:proof-of-porosity}.

\subsubsection{Microlocal conjugation}

Let $\mathcal B\in I^{\comp}_h(\varkappa_k)$, $\mathcal B'\in I^{\comp}_h(\varkappa_k^{-1})$ be semiclassical Fourier integral operators quantizing $\varkappa_k$ near $\{0\}\times \{q_k\}$
in the sense of~\eqref{e:FIO-quantize}.
Recall that $\mathcal B:L^2(M)\to L^2(\mathbb R^{2n})$, $\mathcal B':L^2(\mathbb R^{2n})\to L^2(M)$
are bounded in norm uniformly in~$h$. Define the conjugated operators on~$L^2(\mathbb R^{2n})$
\begin{equation}
  \label{e:tilde-A}
\begin{aligned}
\widetilde A^-&\,:=\mathcal B\Op_h^{L_s}\left(a^-_{\mathbf w_-}\psi_k\right)\mathcal B',\\
\widetilde A^+&\,:=\mathcal B\Op_h^{L_u}\left(a^+_{\mathbf w_+}\psi_k\right)\mathcal B'.
\end{aligned}
\end{equation}
Recall that $\supp\psi_k\subset B\left(q_k,2h^{\rho\over 2}\right)$ by~\eqref{e:psi-k}.
Since $q_k\notin\WF_h(I-\mathcal B'\mathcal B)$, the nonintersecting support
property~\eqref{e:extra-nonint-supp} implies that 
$$
\begin{aligned}
\Op_h^{L_s}(a^-_{\mathbf w_-}\psi_k)\Op_h^{L_u}(a^+_{\mathbf w_+}\psi_k)
=\mathcal B'\widetilde A^-\widetilde A^+\mathcal B+\mathcal O(h^{\infty})_{L^2(M)\to L^2(M)}.
\end{aligned}
$$
Thus the left-hand side of~\eqref{e:key-estimate-4} is bounded as follows:
\begin{equation}
  \label{e:key-end-1}
\begin{aligned}
\left\|\Op_h^{L_s}(a^-_{\mathbf w_-}\psi_k)\Op_h^{L_u}(a^+_{\mathbf w_+}\psi_k)\right\|_{L^2(M)\to L^2(M)}
\leq &\, C\left\|\widetilde A^-\widetilde A^+\right\|_{L^2(\mathbb R^{2n})\to L^2(\mathbb R^{2n})}\\&+\mathcal O(h^\infty).
\end{aligned}
\end{equation}

\subsubsection{Localization of the conjugated operators}
\label{s:localization-conjugated}

Let $\Omega_\pm\subset\mathbb R$ be the sets in Lemma~\ref{l:porosity}.
For $\alpha>0$, define the neighborhoods
\begin{equation}
\Omega_\pm(\alpha):=\Omega_\pm+B(0,\alpha).
\end{equation}
We show the following microlocalization statements for the operators $\widetilde A^\pm$.
While these seem at first to follow naturally from the properties of the supports
of the symbols $a^\pm_{\mathbf w_\pm}\psi_k$ proved in Lemma~\ref{l:porosity},
the proofs of these statements are technically complicated and rely on Lemmas~\ref{l:change-of-quantization}--\ref{l:fio-lag-loc}. If our symbols were more regular,
one could express the conjugated operators $\widetilde A^\pm$ as standard
quantizations of the conjugated symbols by writing them down
as oscillatory integrals and using the method of stationary phase, which
requires us to differentiate the amplitude 2 times per each power of~$h$ gained.
However, the symbols $a^\pm_{\mathbf w_\pm}$ may grow by $h^{-\rho}$ with each differentiation,
and $\rho>\frac 12$, so we cannot blindly apply stationary phase here.
Instead, our argument has to exploit the anisotropic derivative bounds~\eqref{e:symb-def}
and the precise structure of the oscillatory integral expressions
involved.
\begin{lemm}
\label{l:tilde-A-localized}
We have uniformly in $k$, for some constant $C'$ independent of $h$ and $k$
\begin{align}
\label{e:A-localized}
\|\widetilde A^-\indic_{\mathbb R\setminus \Omega_-(C'h^{\rho})}(hD_{y_1})\|_{L^2(\mathbb R^{2n})\to L^2(\mathbb R^{2n})}&=\mathcal O(h^{\varepsilon_0\over 2}),\\
\label{e:A+localized}
\|\indic_{\mathbb R\setminus \Omega_+(C'h^{\rho})}(y_1)\widetilde A^+\|_{L^2(\mathbb R^{2n})\to L^2(\mathbb R^{2n})}&=\mathcal O(h^{\varepsilon_0\over 2}).
\end{align}
\end{lemm}
\begin{proof}
1. We first relate the quantizations $\Op_h^{L_s}$, $\Op_h^{L_u}$ to the
standard quantization $\Op_h^{(1)}$ on~$\mathbb R^{2n}$ given by~\eqref{e:Op-h-def}.
Recall the horizontal foliation $L_H$ defined in~\eqref{e:L-H-def}.
Similarly to~\cite[Lemma~3.6]{hgap} we construct symplectomorphisms $\varkappa_k^\pm$ from neighborhoods
of~$q_k$ in~$T^*M$ to neighborhoods of $0$ in $T^*\mathbb R^{2n}$ such that
$$
\varkappa_k^\pm(q_k)=0,\quad
(\varkappa_k^-)_*L_s=L_H,\quad
(\varkappa_k^+)_*L_u=L_H.
$$
Note the difference between $\varkappa_k^\pm$ and the symplectomorphism $\varkappa_k$
used above: each of $\varkappa_k^\pm$ straightens out \emph{one of} the foliations $L_s,L_u$ in a neighborhood
of $q_k$ and $\varkappa_k$ straightens out \emph{both} foliations $L_s,L_u$
but only at one point $q_k$. There is no symplectomorphism which straightens out both $L_s,L_u$
in a neighborhood of $q_k$.

Let $\mathcal B_\pm\in I^{\comp}_h(\varkappa_k^\pm),\mathcal B_\pm'\in I^{\comp}_h((\varkappa_k^\pm)^{-1})$ be semiclassical Fourier integral operators quantizing $\varkappa_k^\pm$ near $\{0\}\times \{q_k\}$
in the sense of~\eqref{e:FIO-quantize}. From~\eqref{e:a-pm-symb} and~\eqref{e:der-bounds} we have for each~$\varepsilon>0$
$$
a^-_{\mathbf w_-}\psi_k\in S^{\comp}_{L_s,\rho+\varepsilon,\rho/2}(T^*M\setminus 0),\quad
a^+_{\mathbf w_+}\psi_k\in S^{\comp}_{L_u,\rho+\varepsilon,\rho/2}(T^*M\setminus 0).
$$
Moreover, by~\eqref{e:psi-k} we have $\supp\psi_k\subset B(q_k,2h^{\rho\over 2})$.
Recall from~\eqref{e:our-rho} that $\rho={2\over 3}(1-\varepsilon_0)$.
Then by Lemma~\ref{l:L-H-straighten}
(here $K$ is a small closed neighborhood of~$q_k$ and we use that
$\Op_h^{(1)}(a)^*=\Op_h^{(0)}(\bar a)$ from~\eqref{e:Op-h-def})
$$
\begin{aligned}
\Op_h^{L_s}(a^-_{\mathbf w_-}\psi_k)&=\mathcal B_-'\Op_h^{(1)}(\tilde a_-)^*\mathcal B_-+\mathcal O(h^{\varepsilon_0\over 2})_{L^2(M)\to L^2(M)},\\
\Op_h^{L_u}(a^+_{\mathbf w_+}\psi_k)&=\mathcal B_+'\Op_h^{(1)}(\tilde a_+)\mathcal B_++\mathcal O(h^{\varepsilon_0\over 2 })_{L^2(M)\to L^2(M)},
\end{aligned}
$$
where the symbols $\tilde a_\pm\in S^{\comp}_{L_H,\rho+\varepsilon,\rho/2}(T^*\mathbb R^{2n})$
are defined by
$$
\tilde a_-:=(a^-_{\mathbf w_-}\psi_k)\circ(\varkappa^-_k)^{-1},\quad
\tilde a_+:=(a^+_{\mathbf w_+}\psi_k)\circ(\varkappa^+_k)^{-1}.
$$
Recalling the definitions~\eqref{e:tilde-A} of~$\widetilde A^\pm$, we see that
$$
\begin{aligned}
\widetilde A^-&=\mathcal B\mathcal B'_-\Op_h^{(1)}(\tilde a_-)^*\mathcal B_-\mathcal B'+\mathcal O(h^{\varepsilon_0\over 2})_{L^2(\mathbb R^{2n})\to L^2(\mathbb R^{2n})},\\
\widetilde A^+&=\mathcal B\mathcal B'_+\Op_h^{(1)}(\tilde a_+)\mathcal B_+\mathcal B'+\mathcal O(h^{\varepsilon_0\over 2})_{L^2(\mathbb R^{2n})\to L^2(\mathbb R^{2n})}.
\end{aligned}
$$
Therefore, \eqref{e:A-localized} and~\eqref{e:A+localized} follow from the stronger estimates
(where to pass from~\eqref{e:A-localized-2} to~\eqref{e:A-localized}
we use the fact that the operator norm is preserved when taking the adjoint)
\begin{align}
\label{e:A-localized-2}
\|\indic_{\mathbb R\setminus \Omega_-(C'h^\rho)}(hD_{y_1})(\mathcal B_-\mathcal B')^*\Op_h^{(1)}(\tilde a_-)\|_{L^2(\mathbb R^{2n})\to L^2(\mathbb R^{2n})}&=\mathcal O(h^\infty),\\
\label{e:A+localized-2}
\|\indic_{\mathbb R\setminus \Omega_+(C'h^{\rho})}(y_1)\mathcal B\mathcal B'_+\Op_h^{(1)}(\tilde a_+)\|_{L^2(\mathbb R^{2n})\to L^2(\mathbb R^{2n})}&=\mathcal O(h^\infty).
\end{align}

\noindent 2. Recalling the definitions~\eqref{e:Op-h-def} of the standard
quantization $\Op_h^{(1)}$ and~\eqref{e:F-h-def-2} of the semiclassical
Fourier transform $\mathcal F_h$, we have for any $f\in L^2(\mathbb R^{2n})$ and $y\in\mathbb R^{2n}$
$$
\begin{aligned}
\Op_h^{(1)}(\tilde a_\pm)f(y)&=(2\pi h)^{-n}\int_{\mathbb R^{2n}} \mathcal F_hf(\eta)v^\pm_\eta(y)\,d\eta\\\text{where}\quad v^\pm_\eta(y)&:=e^{{i\over h}\langle y,\eta\rangle}\tilde a_\pm(y,\eta).
\end{aligned}
$$
By~\eqref{e:psi-k} and since $\varkappa_k^\pm(q_k)=0$ we have
\begin{equation}
  \label{e:tilde-a-pm-supp}
\supp\tilde a_\pm\subset B(0,C_0h^{\rho\over 2})
\end{equation}
for some constant $C_0$. In particular, $v^\pm_\eta=0$ when $|\eta|>C_0h^{\rho\over 2}$.
Since $\|\mathcal F_hf\|_{L^2(\mathbb R^{2n})}=\|f\|_{L^2(\mathbb R^{2n})}$,
we see by Cauchy--Schwarz that~\eqref{e:A-localized-2}, \eqref{e:A+localized-2} follow from
uniform estimates in~$\eta$:
\begin{align}
\label{e:A-localized-3}
\sup_{\eta\in B(0,C_0h^{\rho\over 2})}\|\indic_{\mathbb R\setminus \Omega_-(C'h^\rho)}(hD_{y_1})(\mathcal B_-\mathcal B')^*v^-_\eta\|_{L^2(\mathbb R^{2n})}&=\mathcal O(h^\infty),\\
\label{e:A+localized-3}
\sup_{\eta\in B(0,C_0h^{\rho\over 2})}\|\indic_{\mathbb R\setminus \Omega_+(C'h^{\rho})}(y_1)\mathcal B\mathcal B'_+v^+_\eta\|_{L^2(\mathbb R^{2n})}&=\mathcal O(h^\infty).
\end{align}

\noindent 3. We first show~\eqref{e:A-localized-3}. 
By the composition property~(5) and the adjoint property~(2) in~\S\ref{s:fios} the operator $(\mathcal B_-\mathcal B')^*$ lies in the class $I^{\comp}_h(\widetilde\varkappa_-)$
where
$$
\widetilde\varkappa_-:=\varkappa_k\circ(\varkappa_k^-)^{-1}.
$$
Since $\varkappa_k(q_k)=\varkappa_k^-(q_k)=0$, we have $\widetilde\varkappa_-(0)=0$.
We have $d\varkappa^-_k(q_k)L_s(q_k)=L_H$;
by Lemma~\ref{l:straighten-out}
and the definition~\eqref{e:L-def} of~$L_s$ we also have $d\varkappa_k(q_k)L_s(q_k)=L_H$. Therefore
$$
d\widetilde\varkappa_-(0)L_H=L_H.
$$
Fix $\eta\in B(0,C_0h^{\rho\over 2})$ and denote
$$
b(y):=\tilde a_-(y,\eta).
$$
By~\eqref{e:tilde-a-pm-supp} we have $\supp b\subset B\left(0,C_0h^{\rho\over 2}\right)$.
Since $\tilde a_-\in S^{\comp}_{L_H,\rho+\varepsilon,\rho/2}(T^*\mathbb R^{2n})$, we see from~\eqref{e:derb-LH} that $b$ satisfies the derivative bounds
$$
\sup_y |\partial^\alpha_yb(y)|\leq C_\alpha h^{-{\rho\over 2}|\alpha|}.
$$
We may assume that $v^-_\eta\neq 0$, that is there exists $y^0\in\mathbb R^{2n}$ such that
$(y^0,\eta)\in\supp \tilde a_-$. Define $\xi\in\mathbb R^{2n}$ by
$$
\widetilde\varkappa_-(y^0,\eta)=(x^0,\xi).
$$
We now apply Lemma~\ref{l:fio-lag-loc} to get for some constant $C'$
\begin{equation}
  \label{e:locator-1}
\|\indic_{\mathbb R^{2n}\setminus B(\xi,C'h^\rho)}(hD_y)(\mathcal B_-\mathcal B')^* v^-_\eta\|_{L^2(\mathbb R^{2n})}=\mathcal O(h^\infty).
\end{equation}
Finally, $(x^0,\xi)\in \widetilde\varkappa_-(\supp\tilde a_-)=\widetilde\varkappa_-\big(\varkappa_k^-(\supp(a_{\mathbf w_-}^-\psi_k))\big)=\varkappa_k(\supp(a_{\mathbf w_-}^-\psi_k))$.
By Lemma~\ref{l:porosity} the first coordinate $\xi_1$ satisfies $\xi_1\in \Omega_-$.
Therefore
$$
\indic_{\mathbb R\setminus \Omega_-(C'h^\rho)}(hD_{y_1})
=\indic_{\mathbb R\setminus \Omega_-(C'h^\rho)}(hD_{y_1})\indic_{\mathbb R^{2n}\setminus B(\xi,C'h^\rho)}(hD_y)
$$
and \eqref{e:A-localized-3} follows from~\eqref{e:locator-1}.

\noindent 4. It remains to show~\eqref{e:A+localized-3}. We write elements
of $\mathbb R^{2n}$ as $(y',y_{2n})$ where $y'\in\mathbb R^{2n-1}$ and
use the unitary semiclassical partial Fourier transform in the $y'$ variables,
$$
\widetilde{\mathcal F}_hf(y',y_{2n})=(2\pi h)^{\frac12-n}\int_{\mathbb R^{2n-1}}
e^{-{i\over h}\langle y',z'\rangle}f(z',y_{2n})\,dz'.
$$
We have
$$
\indic_{\mathbb R\setminus \Omega_+(C'h^{\rho})}(y_1)=\widetilde{\mathcal F}_h\indic_{\mathbb R\setminus \Omega_+(C'h^{\rho})}(hD_{y_1})\widetilde{\mathcal F}_h^{-1}.
$$
Thus~\eqref{e:A+localized-3} is equivalent to
\begin{equation}
  \label{e:A+localized-4}
\sup_{\eta\in B\left(0,C_0h^{\rho\over 2}\right)}\|\indic_{\mathbb R\setminus \Omega_+(C'h^{\rho})}(hD_{y_1})\widetilde{\mathcal F}_h^{-1}\mathcal B\mathcal B'_+v^+_\eta\|_{L^2(\mathbb R^{2n})}=\mathcal O(h^\infty).
\end{equation}
For any $Z\in\Psi^{\comp}_h(\mathbb R^{2n})$, the operator $\widetilde{\mathcal F}_h^{-1}Z$
lies in $I^{\comp}_h(\widetilde\varkappa_F^{-1})+\mathcal O(h^\infty)_{L^2(\mathbb R^{2n})\to L^2(\mathbb R^{2n})}$ where
$$
\widetilde\varkappa_F:T^*\mathbb R^{2n}\to T^*\mathbb R^{2n},\quad
\widetilde\varkappa_{F}(z',z_{2n},\zeta',\zeta_{2n})=(\zeta',z_{2n},-z',\zeta_{2n}).
$$
Therefore, by the composition property~(5) in~\S\ref{s:fios} we have
$\widetilde{\mathcal F}_h^{-1}\mathcal B\mathcal B'_+\in I^{\comp}_h(\widetilde\varkappa_+)+\mathcal O(h^\infty)_{L^2(\mathbb R^{2n})\to L^2(\mathbb R^{2n})}$
where
$$
\widetilde\varkappa_+=\widetilde\varkappa_F^{-1}\circ \varkappa_k\circ (\varkappa_k^+)^{-1}.
$$
Since $\varkappa_k(q_k)=\varkappa^+_k(q_k)=0$, we have $\widetilde\varkappa_+(0)=0$.
We have $d\varkappa_k^+(q_k)L_u(q_k)=L_H$; by
Lemma~\ref{l:straighten-out} and the definition~\eqref{e:L-def} of~$L_u$, we also have $$d\varkappa_k(q_k)L_u=\Span(\partial_{\eta_1},\dots,\partial_{\eta_{2n-1}},\partial_{y_{2n}})$$
and thus $d(\widetilde\varkappa_F^{-1}\circ\varkappa_k)(q_k)L_u(q_k)=L_H$. Therefore
$$
d\widetilde\varkappa_+(0)L_H=L_H.
$$
Now~\eqref{e:A+localized-4} is shown in the same way as~\eqref{e:A-localized-3},
following Step~3 above. Here we use that if $(y^0,\eta)\in\supp\tilde a_+$, then
the point $(x^0,\xi):=\widetilde\varkappa_+(y^0,\eta)$ lies in
$\widetilde\varkappa_F^{-1}\big(\varkappa_k(\supp(a^+_{\mathbf w_+}\psi_k))\big)$
and thus by Lemma~\ref{l:porosity} the first coordinate $\xi_1$
satisfies $\xi_1\in\Omega_+$.
\end{proof}
\subsubsection{Putting things together}
\label{s:putting-together}

We now finish the proof of~\eqref{e:key-estimate-4}. We have
$$
\begin{aligned}
\left\|\Op_h^{L_s}(a^-_{\mathbf w_-}\psi_k)\Op_h^{L_u}(a^+_{\mathbf w_+}\psi_k)\right\|_{L^2(M)\to L^2(M)}\leq
C\left\|\widetilde A^-\widetilde A^+\right\|_{L^2(\mathbb R^{2n})\to L^2(\mathbb R^{2n})}+\mathcal O(h^\infty)\\
\leq C\left\|\widetilde A^-\indic_{\Omega_-(C'h^{\rho})}(hD_{y_1})\indic_{\Omega_+(C'h^{\rho})}(y_1)\widetilde A^+\right\|_{L^2(\mathbb R^{2n})\to L^2(\mathbb R^{2n})}+\mathcal O(h^{\varepsilon_0\over 2})
\end{aligned}
$$
where the first inequality follows from~\eqref{e:key-end-1} and the second one, from Lemma~\ref{l:tilde-A-localized}.

Since $\widetilde A^\pm$ are bounded in $L^2$ norm uniformly in~$h$, it suffices to show
the bound
\begin{equation}
  \label{e:key-estimate-5}
\left\|\indic_{\Omega_-(C'h^{\rho})}(hD_{y_1})\indic_{\Omega_+(C'h^{\rho})}(y_1)\right\|_{L^2(\mathbb R^{2n})\to L^2(\mathbb R^{2n})}\leq Ch^\beta.
\end{equation}
By Lemma~\ref{l:porosity}, the sets $\Omega_\pm$ are $\nu$-porous on scales $C_0h^\rho$ to~1.
By~\cite[Lemma~2.11]{varfup}, the sets $\Omega_\pm(C'h^{\rho})$ are $\nu\over 3$-porous on scales $\max(C_0,{3\over\nu}C')h^{\rho}$ to~1.
Then the Fractal Uncertainty Principle of Proposition~\ref{l:fup-specific} implies~\eqref{e:key-estimate-5} and finishes the proof of~\eqref{e:key-estimate-4} and thus of Proposition~\ref{l:key-estimate}.

\bibliographystyle{alpha}
\bibliography{General,Dyatlov,chyp,QC}

\end{document}